\documentclass[11pt,reqno]{amsart}
\usepackage{amsmath,amssymb,amsthm}
\usepackage{enumerate}
\usepackage{graphicx}
\usepackage[mathscr]{eucal}
\usepackage{tikz}
\usepackage{tikz-cd}
\usepackage{lineno}

\theoremstyle{plain}
\newtheorem{theorem}{Theorem}[section]
\newtheorem{corollary}[theorem]{Corollary}
\newtheorem{lemma}[theorem]{Lemma}
\newtheorem{proposition}[theorem]{Proposition}

\theoremstyle{definition}
\newtheorem{definition}[theorem]{Definition}

\theoremstyle{remark}
\newtheorem{example}[theorem]{Example}
\newtheorem{remark}[theorem]{Remark}

\newcommand{\dom}{{\mathrm{Dom}}}
\newcommand{\supp}{{\mathrm{supp}}}
\newcommand{\real}{\mathbb{R}}
\newcommand{\ereal}{\mathbb{\overline{R}}}
\newcommand{\dec}{\mathbb{E}}
\newcommand{\cat}[1]{\ensuremath{\text{#1}}}
\newcommand{\im}{{\rm{Im}}}
\newcommand{\sheaf}[1]{#1}
\newcommand{\cor}[2]{C( #1,#2)}
\newcommand{\cmod}[1]{\mathbb{#1}}
\newcommand{\cle}{\preccurlyeq}

\newcommand{\nlnr}[2]{\mathbb{I}\,[#1,#2]}
\newcommand{\nlyr}[2]{\mathbb{I}\,[#1,#2\rangle}
\newcommand{\ylnr}[2]{\mathbb{I}\,\langle#1,#2]}
\newcommand{\ylyr}[2]{\mathbb{I}\,\langle#1,#2\rangle}

\makeatletter
\@namedef{subjclassname@2020}{\textup{2020} Mathematics Subject Classification}
\makeatother

\begin{document}


\title[Correspondence Modules and Persistence Sheaves]{Correspondence Modules and Persistence Sheaves:
A Unifying Perspective on One-Parameter Persistent Homology}
\author[H. Hang]{Haibin Hang}
\address{Department of Mathematics \\
    Florida State University \\
    Tallahassee, FL 32306-4510 USA}
\email{hhang@math.fsu.edu}
\author[W. Mio]{Washington Mio}
\address{Department of Mathematics \\
    Florida State University \\
    Tallahassee, FL 32306-4510 USA}
\email{wmio@fsu.edu}

\thanks{This research was partially supported by NSF grant DMS-1722995.}




\begin{abstract}
We develop a unifying framework for the treatment of various persistent homology architectures
using the notion of correspondence modules. In this formulation, morphisms  between
vector spaces are given by partial linear relations, as opposed to linear mappings. In the
one-dimensional case, among other things, this allows us to: (i) treat persistence modules
and zigzag modules as algebraic objects of the same type; (ii) give a categorical formulation
of zigzag structures over a continuous parameter; and (iii) construct barcodes associated with
spaces and mappings that are richer in geometric information. A structural analysis of one-parameter 
persistence is carried out at the level of sections of correspondence modules that yield sheaf-like 
structures, termed persistence sheaves. Under some tameness hypotheses, we prove interval 
decomposition theorems for persistence sheaves and correspondence modules, as well as 
an isometry theorem for persistence diagrams obtained from interval decompositions.  
Applications include: (a) a Mayer-Vietoris sequence that relates the 
persistent homology of sublevelset filtrations and superlevelset filtrations to the levelset homology 
module of a real-valued function and (b) the construction of slices of 2-parameter 
persistence modules along negatively sloped lines.

\end{abstract}
	
\maketitle

\tableofcontents

\section{Introduction}

\subsection{Context}

Rooted in the works of Frosini \cite{Frosini1990} and Robins \cite{Robins1999}, over the years, 
persistent homology has experienced a vigorous development on many fronts, including
theoretical foundations, computation, and applications. 
Through the assembly of homology across multiple scales into algebraic structures known as 
{\em persistence modules}, or $p$-modules, persistent homology provides a powerful technique
for the study of structural properties of geometric objects and data using
topology.  As discussed below, there are  many variants of persistence such as 
forward persistence, backward persistence, and zigzag persistence. 
One of the goals of this paper is to develop a categorical framework that allows us to view all of
these variants as one. Not only does this novel formulation provide a unifying perspective,
but it also reveals new forms of persistence and new 
relationships between different types of persistent structures. We refer to these generalized 
persistence modules as {\em correspondence modules}, or simply $c$-modules. 
In addition to developing this new framework, this paper carries out a structural analysis of 
persistent structures in this expanded setting, investigating interval decomposability of $c$-modules
and stability properties that are crucial for theoretically sound applications.

A prototypical example of a forward $p$-module over $\real$ is the persistent homology
of the {\em sublevel set filtration} of a space $X$ associated with a continuous function $f \colon X \to \real$.
For each $t \in \real$, let $X^t := f^{-1} (-\infty, t]$. Clearly, $X^s \subseteq X^t$, for any $s \leq t$,
so we obtain a filtration of $X$ indexeds over $\real$.  For an integer $i \geq 0$, applying the $i$th  
homology functor  (with coefficients in a fixed field) to this filtration, we obtain a family of vector spaces 
$V_t := H_i (X^t)$  and vector space morphisms $v_s^t \colon V_s \to V_t$, for any $s \leq t$, 
induced by the inclusions $X^s \subseteq X^t$, satisfying $v_t^t = id$,
$\forall t \in \real$, and $v_r^t = v_s^t \circ v_r^s$, if $r \leq s \leq t$. The family $(V_t, v_s^t)$
forms a {\em forward $p$-module}.
An analogous construction with superlevel sets $X_t := f^{-1} [t, \infty)$ yields a similar algebraic 
object, however, with a contravariant behavior. That is, if $s \leq t$, we have a morphism 
$v_s^t \colon V_t \to V_s$, where $V_t := H_i (X_t)$. The family $(V_t, v_s^t)$ is an example of 
a {\em backward $p$-module}.

Zigzag modules, introduced in \cite{Carlsson2010}, are parameterized
over discrete subsets of $\real$. In a zigzag module, a morphism may point in either direction.
An important example is the levelset persistence of a function $f \colon X \to \real$. 
For $t \in \real$, let $X[t]:= f^{-1} (t)$. Although there is no natural mapping from 
$X[s]$ to $X[t]$, for $s < t$, the interlevel set $X_s^t := f^{-1} [s,t]$ can act
as an interpolant, as there are inclusions $X[s] \hookrightarrow X_s^t 
\hookleftarrow X[t]$. Thus, so long as we restrict the parameter values to a discrete 
set, with the aid of interlevel sets we get a sequence of spaces connected by morphisms
alternating from forward to backward. Passing to homology, we obtain a zigzag
module. This formulation is adequate to study the level sets of Morse-type functions,
but insufficient for more general continuous functions, as there seems
to be something inherently discrete about such structures. The question of how
to give a category theory formulation of zigzag structures over $\real$ has been 
asked by many and posed explicitly in \cite{Oudot2015}. 
Another characteristic of this zigzag formulation is that the homology of interlevel sets are 
introduced somewhat artificially in the sequence, as they are not the objects of interest. 
Correspondence  modules  provide a solution to both problems, as detailed in Section \ref{S:levelset}. 
Our approach via $c$-modules altogether removes the homology of interlevel sets from
the sequence, recasting them as morphisms in the appropriate category. We note that levelset persistence  
over a continuous parameter also has been studied via a 2-parameter formulation of persistence
employing interlevel sets  \cite{Botnan2018,Cochoy2020}. This and connections to our formulation
are further discussed in  Section \ref{S:applications}.

Our constructs are based on two main concepts: $c$-modules 
and persistence sheaves, or $p$-sheaves. In a correspondence
module, a morphism between two vectors spaces $U$ and $V$ is given by a partial 
linear relation; that is, a linear subspace of $U \times V$. Letting $\cat{CVec}$ be the 
category whose objects are the vector spaces over a (fixed) field $k$ and the morphisms 
from $U$ to $V$ are the partial linear 
relations from $U$ to $V$, a correspondence module over a poset $(P, \cle)$
is a functor $F \colon (P, \cle) \to \cat{CVec}$. Partial linear relations previously have been
used in topological data analysis by Burghelea and Haller to study circle-valued
mappings \cite{Burghelea2017}, but the present use is quite distinct.

Formally, a $p$-module over $(P, \cle)$ is a functor from $(P, \cle)$ to
$\cat{Vec}$, the category of vector spaces (over $k$) and linear mappings \cite{Lesnick2015,Chazal2016}.
A linear map $T \colon U \to V$ may be viewed as a $\cat{CVec}$-morphism by replacing it
with its graph $G_T$. Thus, via graphs,  any $p$-module may be thought of as a $c$-module.
Note that a ``backward'' mapping $S \colon V \to U$ also may be viewed
as a $\cat{CVec}$-morphism from $U$ to $V$ by replacing $S$ with $G^\ast_S$, where the 
operation $\ast$ swaps the coordinates of $G_S$. From this viewpoint, persistence
structures in which morphisms are given by linear mappings, regardless of whether
the mappings are forward or backward, can all be formulated under the 
same category theory framework. Thus, zigzag modules \cite{Carlsson2010,Gabriel1972} also 
may be viewed as $c$-modules.

The focus of this paper is on $c$-modules over $(\real, \leq)$. Persistence sheaves are 
introduced mainly because it is difficult to directly analyze persistent structures in the 
$\cat{CVec}$ category, but they also are of interest in their own right. We note that a
sheaf-theoretical approach to persistent homology was first investigated
by Curry \cite{Curry2013} and further studied by Kashiwara and Schapira
\cite{Kashiwara2018}, Berkouk and Ginot \cite{Berkouk2018derived},
and Berkouk et al. \cite{Berkouk2019levelsets}. However, the $p$-sheaves that we use to
investigate the structure of $c$-modules satisfy a gluing property that is weaker than the
standard gluing property for sheaves, as
specified in Definition \ref{D:sheaf} and illustrated in Example \ref{E:sheaf}. 
Persistence sheaves also encode finer information than sheaves over open sets
of $\real$, as shown in Example \ref{E:open}.

Under appropriate tameness hypotheses, a $p$-module admits a decomposition, unique 
up to isomorphism, as a direct sum of  atomic units known as {\em interval modules}. This leads to a 
compact representation of persistence modules as persistence diagrams ($p$-diagrams) or barcodes.
In a seminal piece, in which the expression persistent homology was introduced, Edelsbrunner et al. 
developed an algorithm to  compute the persistent homology of a filtered simplicial  complex
and introduced an early form of $p$-diagrams \cite{Edelsbrunner2002}.  Another landmark is the 
work of Zomorodian and Carlsson \cite{ Zomorodian2005},
where the persistent homology of a discrete simplicial filtration is formulated as a graded module over the
polynomial ring $k[x]$ from which one obtains interval decompositions under appropriate finiteness
hypotheses. Barcodes were introduced in \cite{Carlsson2005} as a  representation
of an interval decomposition. This work also led to the insight that the fundamental structure 
underlying persistent homology is that of an inductive system of vector spaces; that is,
a  persistence module. It is in this algebraic setting that Crawley-Boevey 
showed that any pointwise finite-dimensional $p$-module over $(\real, \leq)$ admits
an interval decomposition \cite{Crawley-Boevey2015,Botnan2019decomposition}. 
As such, these $p$-modules may be represented by barcodes or $p$-diagrams, or alternatively,
by Bubenik's persistence landscapes \cite{Bubenik2015}. For levelset persistence, using 
a series of zigzag structures over progressively denser discrete subsets of $\real$, Carlsson et al. showed that 
one can define  persistence diagrams associated with a continuous parameter 
$t \in \real$ \cite{Carlsson2019}. However, the question remained unanswered at the level of modules. 
This paper develops sheaf-theoretical analogues of the techniques of \cite{Crawley-Boevey2015}
to prove interval decomposition theorems for $c$-modules and $p$-sheaves. The argument
involves a whole hierarchy of decompositions that ultimately leads to an interval
decomposition.


The stability of persistence diagrams is another theme of central interest, 
as it is important to ensure that  persistent homology can be used reliably in 
applications.  A  breakthrough result  in this direction is the celebrated stability
theorem of  Cohen-Steiner, Edelsbrunner  and Harer \cite{Cohen-Steiner2007}. 
If $f,g \colon X \to \real$  satisfy some regularity conditions, then $d_b (D_f, D_g) 
\leq \|f-g\|_\infty$,  where $D_f$ and $D_g$ are the persistence diagrams associated 
with the  sublevelset filtration of $X$ induced by $f$ and $g$, respectively, and $d_b$ 
denotes bottleneck distance.  As one often is interested in comparing functional data 
defined on different  domains, extensions to this setting have been studied in 
\cite{Frosini2019,Hang2019}. For
structural data (finite metric spaces), Chazal et al. showed that
the persistence diagram of the Vietoris-Rips filtration is stable with respect
to the Gromov-Hausdorff distance \cite{Chazal2009dgh}. As the transition from functions, or 
other filtrations, to persistence diagrams goes through persistence modules, 
Chazal et al. \cite{Chazal2009,Chazal2016} introduced an interleaving distance $d_I$ between 
$p$-modules to analyze stability at an algebraic level. The Isometry Theorem,
proven by Lesnick \cite{Lesnick2015} and Chazal et al. \cite{Chazal2009,Chazal2016}, states that 
$d_b (D(\cmod{U}), D (\cmod{V})) = d_I (\cmod{U}, \cmod{V})$, where
$\cmod{U}$ and $\cmod{V}$ are $p$-modules over $\real$ and
$D (\cmod{U})$ denotes the persistence diagram of $\cmod{U}$. We prove 
an isometry theorem for $p$-sheaves as a further  extension of such stability 
results.

\subsection{Main Results}

The {\em sections} of a $c$-module $\cmod{V}$ over any interval  $I \subseteq \real$
(see Definition \ref{D:presheaf})  form a vector space $\sheaf{F}(I)$. 
Moreover, if $I \subseteq J$, there is a restriction homomorphism $F^J_I \colon F(J) \to F(I)$.
Sections satisfy certain locality and gluing properties that yield a sheaf-like structure that 
we term persistence sheaf. Analysis of the structure of $p$-sheaves has the 
advantage of placing  us back in the $\cat{Vec}$ category, at the expense of replacing the 
domain category $(\real, \leq)$ with the category
$(\cat{Int}, \subseteq)$, whose objects are the intervals of the real line with morphisms
given  by inclusions. It is in this framework that we establish the central results of the paper,
which are as follows:

\begin{enumerate}[(I)]
\item Under appropriate tameness hypotheses, we prove interval decomposition theorems 
for  $c$-modules and $p$-sheaves that lead to barcode or persistence diagram representations
of their structures.
\item We prove an isometry theorem that states that, for any two interval 
decomposable $p$-sheaves, the bottleneck distance between their persistence diagrams 
is the same as an interleaving distance between their $p$-sheaves of sections. This distance
extends the usual interleaving distance between $p$-modules (cf.\,\cite{Chazal2009,Chazal2016}).
\end{enumerate}

We should point out that interval decompositions of tame $p$-sheaves may be approached 
via block decompositions for 2-D persistence modules \cite{Botnan2018,Cochoy2020}.
However, this is not sufficient to obtain an interval  decomposition theorem for
the class of pointwise finite-dimensional correspondence modules we are interested in 
because their  $p$-sheaves of sections only satisfy a weaker form of tameness -- see 
Examples \ref{E:vtame1} and \ref{E:vtame2}, and  Remark \ref{R:applications}.  
For this reason we approach interval decompositions in a different way, 
developing a sheaf theoretical analogue of the arguments used by Crawley-Boevey in 
the proof of a decomposition theorem for pointwise finite-dimensional  $p$-modules 
over $\real$ \cite{Crawley-Boevey2015}. This approach also provides a different
perspective on interval 
decompositions of $p$-sheaves. We first prove the result for tame $p$-sheaves, as it is simpler 
to describe the arguments in this setting. The proof is then readily adapted to $p$-sheaves
of sections of pointwise finite-dimensional $c$-modules, our primary objects of study. 
Similar remarks apply to our stability results.

One of the applications discussed in the paper illustrates particularly well the unifying
quality of $c$-modules.  For a function $f \colon X \to \real$ and $t \in \real$, the sublevel and 
superlevel sets $X^t$ and $X_t$, respectively, form a cover of $X$ by two subspaces
whose intersection is the level set $X[t]$. If $X$ is a locally compact polyhedron, $f$ is
proper, and homology is Steenrod-Sitnikov \cite{Milnor1961}, 
which satisfies a strong form of excision, there is a Mayer-Vietoris sequence for the cover
$X = X^t \cup X_t$. We construct a {\em persistent Mayer-Vietoris sequence} that ties 
together the persistent homology modules obtained from the sublevelset and superlevelset 
filtrations of $X$, induced by $f$, and the levelset homology $c$-module of $(X, f)$. 
Here we use in full force the fact that we can perform the
direct sum of forward and backward $p$-modules and construct morphisms involving
levelset $c$-modules, all as objects in the same category.

\subsection{Organization}

The rest of the paper is organized as follows. Section \ref{S:cmodules} introduces
some basic terminology and the notion of correspondence modules. Section \ref{S:psheaf}
is devoted to the basic properties of persistence sheaves. Tameness properties
of $c$-modules and $p$-sheaves that imply interval decomposability are
discussed in Section \ref{S:tameness}. The decomposition theorems for
$c$-modules and $p$-sheaves are proven in Section \ref{S:decomposition}
and the Isometry Theorem in Section \ref{S:isometry}. Section \ref{S:applications}
discusses applications to: (i) levelset persistence; (ii) the construction of a
persistent Mayer-Vietoris sequence; and (iii) 1-dimensional slices of
2-parameter persistence modules along lines of negative slope. We close
the paper with some discussion in Section \ref{S:remarks}.

\section{Correspondence Modules} \label{S:cmodules}

\subsection{Preliminaries}

We denote by $\cat{Vec}$ the category whose objects are the vector spaces
over a fixed field $k$ with linear mappings as morphisms. A poset
$(P,\cle)$ is treated as a category whose objects are the elements of $P$.
For $s, t \in P$, there is a single morphism from $s$ to $t$ if $s \cle t$, and none
otherwise. Abusing notation, we also denote the morphism by $s \cle t$.

A {\em persistence module} ($p$-module) over $P$ is a functor
$\cmod{V} \colon P \to \cat{Vec}$. Correspondence modules, defined next, generalize
$p$-modules by allowing more general morphisms between vector spaces.

\begin{definition} \label{D:correspondence}
Let $U$ and $V$ be vector spaces. A (linear, partial) {\em correspondence} from
$U$ to $V$ is a linear  subspace $C\subseteq U \times V$. We denote the set
of all  such linear correspondences by $\cor{U}{V}$. If $C \in \cor{U}{V}$, then:
\begin{enumerate}[(i)]
\item The {\em domain} of $C$ is defined as $\dom (C) = \pi_U (C)$
and the image of $C$ as $\im (C) = \pi_V (C)$, where $\pi_U$ and $\pi_V$
denote the projections onto $U$ and $V$, respectively. The
kernel of $C$ is defined as $\ker (C) = \{u \in U \,|\, (u,0) \in C\}$.
\item The {\em reverse correspondence} $C^\ast \in \cor{V}{U}$ is
defined as
\[
C^\ast = \{ (v,u) \, \colon \, (u,v) \in C\} \,.
\]
\end{enumerate}
\end{definition}

\begin{example}
Let $T\colon U \to V$ be a linear map and $G_T$ the graph of $T$. Then, $G_T \in 
\cor{U}{V}$ and $G^\ast_T \in \cor{V}{U}$.  The graph of the identity map
$I_V \colon V \to V$ gives the diagonal correspondence $\Delta_V := G_{I_V}$.
\end{example}

\begin{definition}
Let $C_1\in \cor{U}{V}$ and $C_2 \in \cor{V}{W}$. The {\em composition}
$C_2 \circ C_1 \in \cor {U}{W}$ is defined as 
$C_2 \circ C_1 = \{ (u,w) \in U \times W \colon \exists v \in V \text{with} \
(u,v) \in C_1 \ \text{and} \ (v,w) \in C_2\}$.
\end{definition}
	
With this composition operation, we form a category $\cat{CVec}$ having vector
spaces over $k$ as objects and correspondences $C \in \cor{U}{V}$ as morphisms from $U$
to $V$. In this category, the identity morphism of an object $V$ is $\Delta_V$, the graph of
the identity map.

\begin{lemma} \label{L:iso}
Let $C \subseteq U \times V$ be a correspondence. If $\dom (C) = U$ and
$\ker (C^\ast) = 0$, then $C$ is the graph of a linear mapping
$T \colon U \to V$. In particular, isomorphisms in $\cat{CVec}$ are given by linear mappings.
More precisely, let  $C \subseteq U \times V$ and $D \subseteq V \times U$ be
correspondences such that $D \circ C = \Delta_U$ and $C \circ D = \Delta_V$. Then, there
is an isomorphism $T \colon U \to V$ such that $G_T = C$ and $G_T^\ast = D$.
\end{lemma}
\begin{proof}
The assumptions on $C$ imply that, for each $u \in U$, there is a unique
$v \in V$ such that $(u,v) \in C$. Define $T$ by $T(u) =v$, which is linear with
the desired properties. For the statement about isomorphisms,
$D \circ C = \Delta_U$ implies that $\text{Dom} (C) = U$. Thus, to show that $C$ is the
graph of a linear mapping $T \colon U \to V$, it suffices to verify that $\ker \, (C^\ast) = 0$,
which follows from the fact the $C$ and $D$ are inverse morphisms. Similarly, there is 
$S \colon V \to U$ such that $G_S = D$. Note that $D \circ C = \Delta_U$ and 
$C \circ D = \Delta_V$ imply that $S \circ T = I_U$ and $T \circ S = I_V$.
\end{proof}

\begin{definition} \label{D:cmod}
A {\em correspondence module} (abbreviated $c$-module) over a poset
$(P, \cle)$ is a functor
$\cmod{V} \colon P \to \cat{CVec}$.
\end{definition}

We adopt the following notation:
\begin{itemize}
\item[(i)] $V_t$ for the vector space associated with $t \in P$; that is, $V_t := \cmod{V} (t)$;
\item[(ii)] $v_s^t$ for the morphism associated with $s \cle t$;
that is, $v_s^t := \cmod{V} (s \cle t)$.
\end{itemize}

\begin{definition} \label{D:morphism}
Let $\cmod{U}$ and $\cmod{V}$ be $c$-modules over $P$. 
\begin{enumerate}[(i)]
\item A {\em morphism}
$F \colon \cmod{U} \to \cmod{V}$ is a natural transformation from the functor
$\cmod{U}$ to the functor $\cmod{V}$. In other words, a collection of compatible
morphisms $f_t \in \cor{U_t}{V_t}$, $t \in P$, meaning that $f_t \circ u_s^t =
v_s^t \circ f_s$, for any $s \cle t$. A morphism $F$ is an {\em isomorphism} if
$f_t$ is a $\cat{CVec}$ isomorphism, $\forall t \in P$.

\item $\cmod{U}$ is a submodule of $\cmod{V}$ if $U_t$ is a
subspace of $V_t$, $\forall t \in P$, and $u_s^t = (U_s \times U_t) \cap v_s^t$,
for any $s \leq t$.
\end{enumerate}
\end{definition}

\begin{definition}
The {\em graph functor} $\cmod{G} \colon \cat{Vec} \to \cat{CVec}$ is defined 
by $\cmod{G} (V) = V$, for any object $V$, and $\cmod{G} (T) = G_T$, for 
any morphism $T$.
\end{definition}

\begin{example} Persistence Modules and Zigzag Modules as $c$-Modules

\smallskip

\begin{enumerate}[(i)]
\item Let $\cmod{U} \colon P \to \cat{Vec}$ be a persistence module over $P$.
The composition $\cmod{V} = \cmod{G} \circ \cmod{U}$
yields a  correspondence module. Thus, any persistence module may
be viewed as a $c$-module via the graphs of its morphisms. 

\smallskip

\item  Consider the poset $P_n = \{0, 1, 2, \ldots, n\}$ with the usual ordering $\leq$. A
zigzag module  $\cmod{V}$ over $P_n$ is a sequence
\begin{equation}
V_0 \overset{p_1}{\longleftrightarrow} V_1 \overset{p_2}{\longleftrightarrow} V_2 
\longleftrightarrow \ldots \longleftrightarrow V_{n-1}  \overset{p_n}{\longleftrightarrow} V_n \,,
\end{equation}
where each $V_i$ is a $k$-vector space and $\overset{p_i}{\longleftrightarrow}$
denotes either a forward homomorphism $f_i \colon V_{i-1} \to V_i$ or a backward
homomorphism $g_i \colon V_i \to V_{i-1}$ \cite{Carlsson2010}. Let 
$C_i \subseteq V_{i-1} \times V_i$ be defined by:
\begin{enumerate}[(a)]
\item $C_i = G_{f_i}$, the graph of  $f_i$, if $p_i$ is a forward homomorphism;
\item $C_i = G^\ast_{g_i}$, the reverse of the graph of $g_i$, if $p_i$ is a backward
homomorphism.
\end{enumerate}
Set $C_{ii} = \Delta_{V_i} \subseteq V_i \times V_i$, for $0 \leq i \leq n$, and
$C_{ij} = C_j \circ \ldots \circ C_{i+1} \subseteq V_i \times V_j$, for $0 \leq i < j \leq n$. 
These correspondences induce a $c$-module structure on $\cmod{V}$. More precisely,
$\cmod{V} (i) := V_i$ and $\cmod{V} (i\leq j):= C_{ij}$ define a functor 
$\cmod{V} \colon (P_n, \leq) \to \cat{CVec}$.
\end{enumerate}
\end{example}

We denote by $\cat{CMod} \,(P)$, or simply $\cat{CMod}$, the category whose objects
are the $c$-modules over $P$ with natural transformations as morphisms.
Next, we show that morphisms in $\cat{CMod}$ have images in the category theory
sense. Zero morphisms, kernels and cokernels are not well defined in
$\cat{CMod}$. However, in some special situations, we can associate a kernel or a
cokernel $c$-module to a morphism.

\begin{definition}
Let $\cmod{U}$ and $\cmod{V}$ be $c$-modules over $P$ and $F \colon \cmod{U} 
\to \cmod{V}$ a $\cat{CMod}$ morphism given by compatible $\cat{CVec}$ morphisms
$f_t \colon U_t \to V_t$. We adopt the notation $[v_t]$ for the element of cokernel of
$f_t$ represented by $v_t \in V_t$.
\begin{enumerate}[(i)]
\item Let $\im_t = \im (f_t)$ and $\text{im}_s^t := v_s^t \cap (\im_s \times \im_t)$.
We refer to the pair $\im (F) :=(\{\im_t \colon t \in P\}, \{\text{im}_s^t \colon s \cle t\})$
as the {\em image of $F$}.
\item Similarly, letting $K_t = \ker (f_t)$ and $k_s^t = u_s^t \cap (K_s \times K_t)$,
define the {\em kernel of $F$} as the pair $\ker (F) := (\{ K_t \colon t \in P\}, 
\{k_s^t \colon s \cle t\})$.
\item Let $Q_t = \text{coker} (f_t)$ and $q_s^t \subseteq Q_s \times Q_t$ be
the subspace given by $([v_s], [v_t]) \in q_s^t$ if and only if there exist
$a_s \in \im (f_s)$ and $a_t \in \im (f_t)$ such that $(v_s + a_s, v_t + a_t) \in
v_s^t$. Define the {\em cokernel of $F$} as the pair $\text{coker} (F) := 
(\{Q_t \colon t \in P\}, \{q_s^t \colon s \cle t\})$.
\end{enumerate}
\end{definition}

\begin{proposition}[Images and Kernels] \label{P:imker}
If $F \colon \cmod{U} \to \cmod {V}$ is a $\cat{CMod}$ morphism, then
\begin{enumerate}[\rm (i)]
\item $\im (F)$ is a submodule of $\cmod{V}$ and the inclusion
$\im (F) \hookrightarrow \cmod{V}$ is an image of the morphism $F$ in the $\cat{CMod}$
category;
\item If $G \colon \cmod{W} \to \cmod{U}$ is another $\cat{CMod}$ morphism
and the sequence
\[
\begin{tikzcd}
W_t \ar[r, "g_t"] & U_t \ar[r, "f_t"]  &  V_t
\end{tikzcd}
\]
is exact in $\cat{CVec}$, $\forall t \in P$,  then $\ker (F)$ is a submodule of $\cmod{U}$;
\item If $\cmod{V}$ is a persistence module, then $\text{coker} \,(F)$ is a $c$-module.
\end{enumerate}
\end{proposition}
\begin{proof}
(i) To verify that $\im (F)$ is a submodule of $\cmod{V}$, it suffices to check
the composition rule $\text{im}_r^t = \text{im}_s^t \circ \text{im}_r^s$, for any
$r \cle s \cle t$. The inclusion $\text{im}_r^t \supseteq \text{im}_s^t \circ \text{im}_r^s$
is straightforward. For the reverse inclusion, let $(v_r, v_t) \in \text{im}_r^t$. Then,
there exist $v_s \in V_s$, $u_r \in U_r$ and $u_t \in U_t$ such that $(v_r, v_s)
\in v_r^s$, $(v_s, v_t) \in v_s^t$, $(u_r, v_r) \in f_r$ and $(u_t, v_t) \in f_t$. Our goal
is to show that $v_s \in \im (f_s)$, as this implies that $(v_r, v_t) \in 
\text{im}_s^t \circ \text{im}_r^s$. Since $(u_r, v_s) \in v_r^s \circ f_r =
f_s \circ u_r^s$, there exists $u_s \in U_s$ such that $(u_r, u_s) \in u_r^s$
and $(u_s, v_s) \in f_s$, showing that $v_s \in \im f_s$, as desired.
The universal property for images in \cat{CMod} is easily verified for $\im (F)$.

(ii) The assumption that the sequences are exact implies that $\ker (F) = \im (G)$.
Hence, (i) implies that $\ker (F)$ is a submodule of $\cmod{U}$. 

(iii) Since $\cmod{V}$ is a $p$-module, the correspondences $v_s^t$ are given
by graphs of linear mappings; that is, $v_s^t = G_{\phi_s^t}$, where
$\phi_s^t \colon V_s \to V_t$ are linear mappings. Let $\text{coker} (F) := (Q_t, q_s^t)$.
For $r \cle s \cle t$, we first show that $q_r^t \subseteq q_s^t \circ q_r^s$.
Let $([v_r], [v_t]) \in q_r^t$. Then, there exist
$a_r \in \im (f_r)$ and $a_t \in \im (f_t)$ such that $(v_r + a_r, v_t + a_t) \in
v_r^t$, which means that $\phi_r^t (v_r + a_r) = v_t + a_t$. Let
$v_s = \phi_r^s (v_r)$ and $a_s = \phi_r^s (a_r)$. A diagram chase 
shows that $a_s \in \im (f_s)$. Clearly, $\phi_r^s (v_r + a_r) = v_s + a_s$
and $\phi_s^t (v_s + a_s) = v_t + a_t$, showing that $([v_r], [v_s]) \in q_r^s$ 
and $([v_s], [v_t]) \in q_s^t$; that is, $([v_r], [v_t]) \in q_s^t \circ q_r^s$ . For
the converse inclusion, suppose that $([v_r], [v_s]) \in q_r^s$  and
$([v_s], [v_t]) \in q_s^t$. Then, there exist $a_r \in \im (f_r)$, $a_s, b_s \in \im (f_s)$
and $a_t \in \im (f_t)$ such that $\phi_r^s(v_r + a_r ) = v_s + a_s$
and $\phi_s^t (v_s + b_s) = v_t + a_t$. The fact that $F$ is a
$\cat{CMod}$ morphism implies that $c_t = \phi_s^t (a_s - b_s) \in 
\im (f_t)$. Then,
\begin{equation}
\begin{split}
\phi_r^t (v_r + a_r) &= \phi_s^t (v_s + a_s) = \phi_s^t (v_s + b_s + a_s - b_s) \\
&= v_t + a_t + \phi_s^t (a_s - b_s) = v_t + (a_t + c_t) \,,
\end{split}
\end{equation}
which implies that $([v_r], [v_t]) \in q_r^t$, concluding the proof.
\end{proof}

We close this section with a discussion of direct sums of $c$-modules.

\begin{definition}
Let $\{\cmod{V}^\lambda, \lambda \in \Lambda\}$, be an indexed
collection of $c$-modules. Define the {\em direct sum}
$\cmod{V} = \oplus_{\lambda \in \Lambda} \cmod{V}^\lambda$ by
$V_t = \oplus_{\lambda \in \Lambda} V_t^\lambda$, with correspondences
$v_s^t = \oplus_{\lambda \in \Lambda} \cmod{V}^\lambda (s \cle t)
\subseteq V_s \times V_t$, for any $s \cle t$.
\end{definition}
	
\begin{proposition}
If $\{\cmod{V}^\lambda, \lambda \in \Lambda\}$ is an indexed collection of $c$-modules,
then the direct sum $\cmod{V} = \oplus_{\lambda \in \Lambda} \cmod{V}^\lambda$ also is a
$c$-module.
\end{proposition}
	
\begin{proof}
We only need to show that $v_s^t \circ v_r^s = v_r^t$,
for any $r \cle s \cle t$. We first verify that $v_s^t \circ v_r^s
\supseteq v_r^t$. Given $(v_r, v_t ) \in v_r^t$, write
$(v_r, v_t ) = \sum_{\lambda \in \Lambda} c_\lambda (v_r^\lambda, v_t^\lambda)$, where
$(v_r^\lambda, v_t^\lambda) \in \cmod{V}^\lambda (r \cle t)$ and
all but finitely many coefficients $c_\lambda \in k$ vanish. For each $\lambda \in \Lambda$, 
there exists $v_s^\lambda \in \cmod{V}_s^\lambda$, such that 
$(v_r^\lambda, v_s^\lambda) \in \cmod{V}^\lambda (r \cle s)$ and
$(v_s^\lambda, v_t^\lambda) \in \cmod{V}^\lambda (s \cle t)$. Letting $v_s = 
\oplus_{\lambda \in \Lambda} c_\lambda v_s^\lambda$, it follows that
$(v_r, v_s) \in v_r^s$ and $(v_s, v_t) \in v_s^t$. Thus,
$(v_r, v_t) \in v_s^t \circ v_r^s$, as claimed. 

Conversely, let $(v_r, v_t) \in v_s^t \circ v_r^s$. Then,
there exists $v_s \in  \oplus_{\lambda \in \Lambda} V_s^\lambda$ such that
$(v_r, v_s) \in v_r^s$ and $(v_s, v_t) \in v_s^t$. Write
\begin{equation}
(v_r,v_s)=  \sum_{\lambda \in \Lambda} c_\lambda (v_r^\lambda, v_s^\lambda)
\quad \text{and} \quad
(v_s,v_t)=  \sum_{\lambda \in \Lambda} d_\lambda (v_s^\lambda, v_t^\lambda) \,,
\end{equation}
where $(v_r^\lambda, v_s^\lambda) \in \cmod{V}^\lambda (r \cle s)$,
$(v_s^\lambda, v_t^\lambda) \in \cmod{V}^\lambda (s \cle t)$ and all but finitely
many scalars $c_\lambda, d_\lambda \in k$ vanish.
Then, $v_s =  \sum_{\lambda \in \Lambda} c_\lambda v_s^\lambda$ and
$v_s =  \sum_{\lambda \in \Lambda} d_\lambda v_s^\lambda$,
which implies $c_\lambda = d_\lambda$, $\forall \lambda \in \Lambda$. Hence,
$(v_r, v_t) = \sum_{\lambda \in \Lambda} c_\lambda (v_r^\lambda, v_t^\lambda)
\in v_r^t$.
\end{proof}

\begin{definition}
A  correspondence module $\cmod{V}$ is {\em indecomposable} if
$\cmod{V} \cong \cmod{V}_1 \oplus \cmod{V}_2$ implies that either
$\cmod{V}_1=0$ or $\cmod{V}_2=0$.
\end{definition}


\subsection{Interval Correspondence Modules}

We now specialize to correspondence modules over $(\real, \leq)$. We introduce
interval $c$-modules associated with each interval $I \subseteq \real$. Unlike interval
$p$-modules (cf.\,\cite{Chazal2016}), there may be up to four non-isomorphic interval 
$c$-modules associated with $I$ (cf.\,\cite{Botnan2018,Carlsson2019,Cochoy2020}). 

Let $\dec$ be the set of extended, decorated real numbers defined as
\begin{equation}
\dec = \real \times \{+, -\} \cup \{ -\infty, + \infty\} .
\end{equation}
For $t \in \real$, we use the abbreviations $t^+ := (t, +)$, $t^- := (t, -)$,
and $t^\ast$ for either $t^+$ or $t^-$. Throughout the paper, $\dec$ is equipped with the
total ordering $\leq$ given by: 
\begin{enumerate}[(i)]
\item $t_1^\ast \leq t_2 ^\ast$, for any $t_1, t_2 \in \real$ satisfying $t_1 < t_2$;
\item $t^- \leq t^+$, for any $t \in \real$;
\item $- \infty \leq t^\ast \leq +\infty$, $\forall t \in \real$, and $-\infty \leq + \infty$;
\item $p \leq p$, $\forall p \in \dec$. 
\end{enumerate}
For $p, q \in \dec$, we write $p < q$ to mean that $p \leq q$ and $p \ne q$.
Decorated numbers give a uniform notation for intervals in $\real$,
whether open, closed, or half-open. We adopt the following identification between
objects in $\cat{Int}$ and elements of $\{(p,q) \in \dec^2 \,|\, p < q\}$ (cf.\,\cite{Chazal2016}):
\begin{enumerate}[(a)]
\item For $t_1, t_2 \in \real$ with $t_1 < t_2$, $(t_1, t_2) \leftrightarrow (t_1^+, t_2^-)$,
$[t_1, t_2) \leftrightarrow (t_1^-, t_2^-)$, $(t_1, t_2] \leftrightarrow (t_1^+, t_2^+)$, and
$[t_1, t_2] \leftrightarrow (t_1^-, t_2^+)$;
\item For $t \in \real$, the one-point interval $[t,t]$ corresponds to $(t^-, t^+) \in \dec^2$;
\item For $t \in \real$, $(-\infty, t) \leftrightarrow (-\infty, t^-)$,
$(-\infty, t]  \leftrightarrow  (-\infty, t^+)$, $(t, +\infty) \leftrightarrow  (t^+, +\infty)$ and
$[t, +\infty) \leftrightarrow (t^-, +\infty)$;
\item The entire real line $\real$ corresponds to  $(-\infty, +\infty) \in \dec^2$. 
\end{enumerate}

\begin{definition} \label{D:interval}
Let $(p, q) \in \dec^2$, $p < q$, represent an interval in $\real$.  We define $c$-modules
$\nlnr{p}{q}$, $\nlyr{p}{q}$, $\ylnr{p}{q}$ and $\ylyr{p}{q}$ associated with $(p, q)$ as follows
(see Fig.\,\ref{F:intervals}):
\begin{enumerate}[(i)]
\item Let $\cmod{I}$ denote any of the above $c$-modules and $t \in \real$. Define
$\cmod{I} (t) = k$, $\forall t \in (p, q)$, and $\cmod{I} (t) = 0$, otherwise. 
For $s, t \in \real$ and $s \leq t$, set $\cmod{I} (s \leq t) = \Delta_k$ if $s, t \in (p, q)$, 
and $\cmod{I} (s \leq t) = 0 \times 0$ if $s, t \notin (p, q)$;
\item $\nlnr{p}{q} (s \leq t) = 0 \times 0$ if $s \notin (p, q)$ or $t \notin (p, q)$;
\item $\nlyr{p}{q} (s \leq t) = k \times 0$ if $s \in (p, q)$ and $t \in (q, +\infty)$;

$\nlyr{p}{q} (s \leq t) = 0 \times 0$ if $s \in (-\infty, p)$;
\item $\ylnr{p}{q} (s \leq t) = 0 \times k$ if $s \in (-\infty, p)$ and $t \in (p, q)$;

$\ylnr{p}{q} (s \leq t) = 0 \times 0$ if $t \in (q, +\infty)$;
\item $\ylyr{p}{q} (s \leq t) = k \times 0$ if $s \in (p, q)$ and $t \in (q, +\infty)$;

$\ylyr{p}{q} (s \leq t) = 0 \times k$ if $s \in (-\infty, p)$ and $t \in (p, q)$.
\end{enumerate}
\end{definition}

\begin{remark} \label{R:arrows} \

\begin{enumerate}[(a)]
\item If $(p, q) \in \dec^2$ is a finite interval, the four modules in Definition \ref{D:interval}
fall in different isomorphism classes. If $p = -\infty$,
then $\nlyr{p}{q} = \ylyr{p}{q}$ and $\nlnr{p}{q} = \ylnr{p}{q}$. Similarly, if
$q = +\infty$, $\nlyr{p}{q} = \nlnr{p}{q}$ and $\ylnr{p}{q} = \ylyr{p}{q}$. If
$(p, q) = (-\infty, +\infty)$, all four $c$-modules coincide. 

\item We adopt the pictorial representation of these four types of interval modules
indicated in Fig.\,\ref{F:intervals}.
\end{enumerate}
\end{remark}

\begin{center}
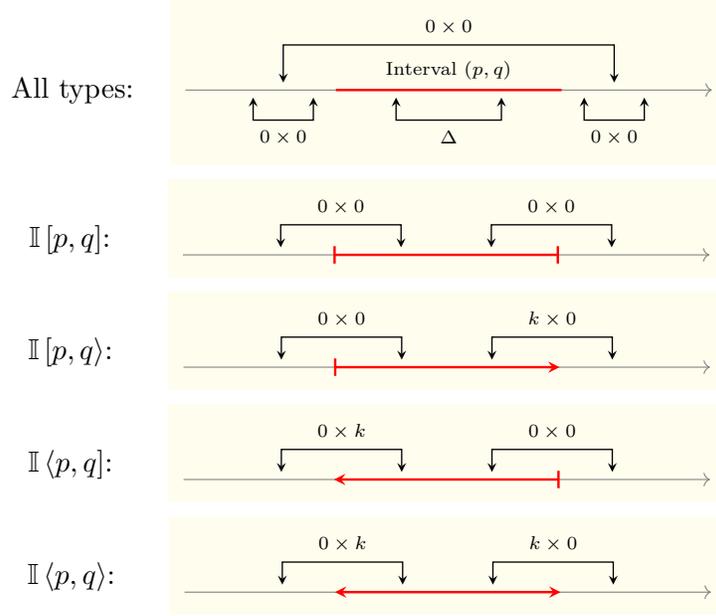
\begin{figure}
\hspace{-0.13in}
\begin{tikzpicture}[line width=0.9pt]
\draw (-5, 0) node {All types:};
\fill[yellow!8!white] (-3.7, -1) rectangle (3.7, 1.2);
\draw [->, thin, gray]  (-3.5, 0) -- (3.5, 0);
\begin{scope}[>=stealth]
\draw[color=red] (-1.5, 0) -- node[above=0.1pt, color=black] {\tiny{Interval $(p,q)$}}(1.5, 0);
\draw[<->, thin, line width=0.5pt] (-2.2, 0.1) -- (-2.2, 0.6) -- node[above=0.1pt] 
{\tiny{$0 \times 0$}} (2.2, 0.6) -- (2.2, 0.1);
\draw[<->, thin, line width=0.5pt] (1.8, -0.1) -- (1.8, -0.4) -- node[below=0.1pt] 
{\tiny{$0 \times 0$}} (2.6, -0.4) -- (2.6, -0.1);
\draw[<->, thin, line width=0.5pt] (-1.8, -0.1) -- (-1.8, -0.4) -- node[below=0.1pt] 
{\tiny{$0 \times 0$}} (-2.6, -0.4) -- (-2.6, -0.1);
\draw[<->, thin, line width=0.5pt] (-0.7, -0.1) -- (-0.7, -0.4) -- node[below=0.1pt] 
{\tiny{$\Delta$}} (0.7, -0.4) -- (0.7, -0.1);
\end{scope}
\end{tikzpicture}

\medskip
\begin{tikzpicture}[line width=0.9pt]
\draw (-5, 0.2) node {$\nlnr{p}{q}$:};
\fill[yellow!8!white] (-3.7, -0.3) rectangle (3.7, 1);
\draw [->, thin, gray]  (-3.5, 0) -- (3.5, 0);
\begin{scope}[>=stealth]
\draw[|-|, color=red] (-1.5, 0) -- (1.5, 0);
\draw[<->, thin, line width=0.5pt] (0.6, 0.1) -- (0.6, 0.4) -- node[above=0.1pt] 
{\tiny{$0 \times 0$}} (2.2, 0.4) -- (2.2, 0.1);
\draw[<->, thin, line width=0.5pt] (-2.2, 0.1) -- (-2.2, 0.4) -- node[above=0.1pt] 
{\tiny{$0 \times 0$}} (-0.6, 0.4) -- (-0.6, 0.1);
\end{scope}
\end{tikzpicture}

\medskip

\begin{tikzpicture}[line width=0.9pt]
\draw (-5, 0.2) node {$\nlyr{p}{q}$:};
\fill[yellow!8!white] (-3.7, -0.3) rectangle (3.7, 1);
\draw [->, thin, gray]  (-3.5, 0) -- (3.5, 0);
\begin{scope}[>=stealth]
\draw[|->, color=red] (-1.5, 0) -- (1.5, 0);
\draw[<->, thin, line width=0.5pt] (0.6, 0.1) -- (0.6, 0.4) -- node[above=0.1pt] 
{\tiny{$k \times 0$}} (2.2, 0.4) -- (2.2, 0.1);
\draw[<->, thin, line width=0.5pt] (-2.2, 0.1) -- (-2.2, 0.4) -- node[above=0.1pt] 
{\tiny{$0 \times 0$}} (-0.6, 0.4) -- (-0.6, 0.1);
\end{scope}
\end{tikzpicture}

\medskip

\begin{tikzpicture}[line width=0.9pt]
\draw (-5, 0.2) node {$\ylnr{p}{q}$:};
\fill[yellow!8!white] (-3.7, -0.3) rectangle (3.7, 1);
\draw [->, thin, gray]  (-3.5, 0) -- (3.5, 0);
\begin{scope}[>=stealth]
\draw[<-|, color=red] (-1.5, 0) -- (1.5, 0);
\draw[<->, thin, line width=0.5pt] (0.6, 0.1) -- (0.6, 0.4) -- node[above=0.1pt] 
{\tiny{$0 \times 0$}} (2.2, 0.4) -- (2.2, 0.1);
\draw[<->, thin, line width=0.5pt] (-2.2, 0.1) -- (-2.2, 0.4) -- node[above=0.1pt] 
{\tiny{$0 \times k$}} (-0.6, 0.4) -- (-0.6, 0.1);
\end{scope}
\end{tikzpicture}

\medskip

\begin{tikzpicture}[line width=0.9pt]
\draw (-5, 0.2) node {$\ylyr{p}{q}$:};
\fill[yellow!8!white] (-3.7, -0.3) rectangle (3.7, 1);
\draw [->, thin, gray]  (-3.5, 0) -- (3.5, 0);
\begin{scope}[>=stealth]
\draw[<->, color=red] (-1.5, 0) -- (1.5, 0);
\draw[<->, thin, line width=0.5pt] (0.6, 0.1) -- (0.6, 0.4) -- node[above=0.1pt] 
{\tiny{$k \times 0$}} (2.2, 0.4) -- (2.2, 0.1);
\draw[<->, thin, line width=0.5pt] (-2.2, 0.1) -- (-2.2, 0.4) -- node[above=0.1pt] 
{\tiny{$0 \times k$}} (-0.6, 0.4) -- (-0.6, 0.1);
\end{scope}
\end{tikzpicture}
\caption{Interval $c$-modules associated with the interval $(p,q)$.}
\label{F:intervals}
\end{figure}
\end{center}

\begin{lemma}[Indecomposability]
Interval $c$-modules are indecomposable.
\end{lemma}
\begin{proof}
The argument is standard. Let $\cmod{I}$ be an interval $c$-module of any of the
types described in Definition \ref{D:interval}. The corresponding interval in $\real$ is
denoted $I$. Suppose $\eta$ is a $c$-module isomorphism from $\cmod{I}$ to
$\cmod{U} \oplus \cmod{V}$. By Lemma \ref{L:iso}, $\forall t \in I$, $\eta_t$ is the
graph of a linear isomorphism $\phi_t \colon k \to U_t \oplus V_t$. Let
$\pi \colon \cmod{U} \oplus \cmod{V} \to  \cmod{U} \oplus \cmod{V}$ denote
projection onto $\cmod{U}$ followed by inclusion into $\cmod{U} \oplus \cmod{V}$. Then, 
 $\psi_t = \phi^{-1}_t \circ \pi \circ \phi_t \colon k \to k$ is an idempotent of $k$.
 Thus, $\psi_t$ is given by multiplication by $0$ or $1$. Since $\phi_t$ induce a $c$-module
 morphism, one can verify that this scalar is independent of $t \in I$. Hence, either
 $\cmod{U} =0$ or $\cmod{V} = 0$.
\end{proof}


\section{Persistence Sheaves} \label{S:psheaf}

Henceforth, all $c$-modules will be over $(\real, \leq)$. In this section, we
develop a framework for the study of the structure of $c$-modules over $\real$
based on the concept of persistence sheaves. 

Let $2^\real$ be the category whose objects are the subsets of $\real$ with inclusion
of sets as morphisms.  We may think of the objects of $2^\real$ as the open sets of
$\real$ in the discrete topology. We denote by $\cat{Int}$ the subcategory of $2^\real$
whose objects are all intervals $I \subseteq \mathbb{R}$, open, closed or half-open 
(half-closed).

\begin{definition} (Presheaves)
\begin{enumerate}[(i)]
\item A {\em discrete presheaf} is a contravariant functor from $2^\real$
to $\cat{Vec}$; that is, a functor $\sheaf{G} \colon (2^\real)^{\text{op}} \to \cat{Vec}$.
\item A {\em persistence presheaf\/} is a functor
$\sheaf{F} \colon \cat{Int}^{\text{op}} \to \cat{Vec}$.
\end{enumerate}
\end{definition}

\noindent
Here, the superscript $\text{op}$ denotes the opposite category.
Clearly, any discrete presheaf defines a persistence presheaf via
restriction to $\cat{Int}$. For a persistence presheaf $\sheaf{F}$, we adopt the
following terminology:
\begin{enumerate}[(a)]
\item We refer to an  element of the vector space $\sheaf{F}(I)$ as a
{\em section} of $\sheaf{F}$ over $I$. 
\item For $I \subseteq J$, we refer to the linear map $F_I^J :=
\sheaf{F} (I \subseteq J) \colon \sheaf{F} (J) \to \sheaf{F} (I)$ as a {\em restriction
homomorphism}. If $s \in \sheaf{F} (J)$, we sometimes use the notation $s|_I$ for $F_I^J (s)$.
\item If $I = \{a\}$ is a singleton, we simplify the notation for $\sheaf{F}(I)$, $F_I^J$ and
$s|_I$ to $F_a$, $F_a^J$ and $s|_a$, respectively.
\item If $s \in F(I)$, the {\em support} of $s$ is defined as
$\supp [s]:= \{a \in I \colon F^I_a (s) \ne 0\}$.
\item If $s$ and $s'$ are sections of $\sheaf{F}$, we write $s \cle s'$ to indicate that
$s$ is a restriction of $s'$.
\end{enumerate}

Similar notation and terminology are adopted for discrete presheaves.

\begin{remark} 
Let $\sheaf{F}$ be the restriction of a discrete presheaf $\sheaf{G}$ to $\cat{Int}$. 
If $J \subseteq \real$ is an interval, $A \subseteq J$, and it is clear from the context what
the presheaf $G$ is, we abuse notation
and write $F_A^J$ for the restriction homomorphism ``inherited'' from $\sheaf{G}$.
Similarly, if $s \in F(J)$, we write $s|_A$ for $G^J_A (s)$.
More generally, if $V$ is a subspace of $\sheaf{F} (J)$, we set
$V|_A = \{ s|_A \,\colon\, s \in V \}$.
\end{remark}

\begin{definition} (Morphisms)

\begin{enumerate}[(i)]
\item Let $F$ and $G$ be persistence presheaves. A {\em morphism} from $F$ to $G$ is
a natural transformation $\Phi:F\rightarrow G$; that is, a collection
$\Phi:=\{\phi_I:F(I) \to G(I) \colon I \in \cat{Int}\}$ of linear mappings such that
$G^J_I\circ\phi_J=\phi_I\circ F^J_I$, for any $I\subseteq J$.

\item $\Phi \colon \sheaf{F} \to \sheaf{G}$ is an {\em isomorphism} if there is a morphism
$\Psi \colon \sheaf{G} \to \sheaf{F}$ such that $\Psi \circ \Phi$ and
$\Phi \circ \Psi$ are the identity morphisms of $F$ and $G$, respectively.
\end{enumerate}
\end{definition}

\begin{definition} \label{D:presheaf}

\medskip
Let $\cmod{V}$ be a $c$-module and $A \subseteq \real$.
\begin{enumerate}[(i)]
\item A {\em section} of $\cmod{V}$ over $A$ is an indexed family $\sigma = (v_t)$, $t \in A$,
such that $v_t \in V_t$ and $(v_r, v_t) \in v_r^t$, for any $r, t \in A$
with $r \leq t$.  The {\em domain} of $s$ is the set $A$ and we
denote it $\text{Dom} (s)$. The {\em support} of $s$ is defined as
$\supp [s]= \{t \in \dom (s) \colon v_t \ne 0\}$.
Pointwise addition and scalar multiplication induce a
$k$-vector space structure on the collection of all sections over $A$. 
If $s$ is a section over $B$ and $A \subseteq B$, the restriction of $s$ to $A$
is denoted $s|_A$.

\item The {\em discrete presheaf of sections of \,$\cmod{V}$}, denoted
$\sheaf{D}_{\cmod{V}}$, is the contravariant  functor that associates to each
$A \subseteq \real$ the vector space $\sheaf{D}_{\cmod{V}} (A)$ of all sections
of $\cmod{V}$ over $A$. If $A \subseteq B$, then $\sheaf{D}_{\cmod{V}} (A \subseteq B)$
is the linear mapping  given by $s \mapsto s|_A$, for any section $s$ over $B$.

\item The {\em persistence presheaf of sections of $\cmod{V}$} is the
restriction of $\sheaf{D}_{\cmod{V}}$ to $\cat{Int}$.
\end{enumerate}
\end{definition}

To define persistence sheaves, we introduce the notion of connected covering.

\begin{definition} \label{D:covering}
Let $X$ be a set and $X = \cup X_\lambda$, $\lambda \in \Lambda$, a covering of $X$ by
subsets $X_\lambda$. The covering is {\em connected} if for any non-trivial partition 
$\Lambda = \Lambda_0 \sqcup \Lambda_1$, the intersection of the sets
$X_0 = \cup_{\lambda \in \Lambda_0} X_\lambda$ and $X_1 =
\cup_{\lambda \in \Lambda_1} X_\lambda$ is non-empty.
\end{definition}

\begin{lemma} \label{L:covering}
Let $X = \cup X_\lambda$, $\lambda \in \Lambda$, be a covering of $X$ with
$X_\lambda \ne \emptyset$, $\forall \lambda \in \Lambda$. The covering is
connected if and only if, for any $\lambda, \mu \in \Lambda$, there is a finite sequence
$\lambda_0, \ldots, \lambda_n \in \Lambda$ such that $\lambda_0 = \lambda$,
$\lambda_n = \mu$ and $X_{\lambda_{i-1}} \cap X_{\lambda_i} \ne \emptyset$,
for $1 \leq i \leq n$.
\begin{proof}
Consider the equivalence relation on $\Lambda$ generated by $\lambda \sim \mu$
if $X_\lambda \cap X_\mu \ne \emptyset$. Then, $\lambda \sim \mu$ if and only if 
a sequence as above exists. It is simple to verify that the covering is connected
if and only if $[\lambda] = \Lambda$, $\forall \lambda \in \Lambda$.
\end{proof}
\end{lemma}

\begin{lemma} \label{L:cover}
Let $\{X_\lambda, \lambda \in \Lambda\}$ be a connected covering of a
set $X$. If $V \subseteq X$ is such that, for each $\lambda \in \Lambda$,
$X_\lambda \subseteq V$ or $X_\lambda \cap V = \emptyset$, 
then $V = X$ or $V = \emptyset$.
\end{lemma}
\begin{proof}
Let $\Lambda_0 = \{ \lambda \in \Lambda \,|\, X_\lambda \subseteq V\}$ and
$\Lambda_1 = \{ \lambda \in \Lambda \,|\, X_\lambda \cap V = \emptyset\}$. This gives a
partition of $\Lambda$ with the property that $X_0 \cap X_1 = \emptyset$, with $X_0$ and
$X_1$ as in Definition \ref{D:covering}. Since the covering is connected, either
$\Lambda_0 = \emptyset$ or $\Lambda_1 = \emptyset$. This implies that $V=X$ or 
$V = \emptyset$.
\end{proof}

\begin{definition} \label{D:sheaf}
Let $\sheaf{F}$ be a persistence presheaf. $\sheaf{F}$ is a {\em persistence sheaf}
(abbreviated $p$-sheaf) if the following conditions are satisfied:
\begin{enumerate}[(i)]
\item (Locality) For any covering $I=\cup I_\lambda$, $\lambda \in \Lambda$,
of $I$ by non-empty intervals $I_\lambda$, if $s \in \sheaf{F}(I)$ is such that
$s|_{I_\lambda} = 0$, $\forall \lambda \in \Lambda$,  then $s = 0$;

\item (Connective Gluing) For any connected covering
$I=\cup I_\lambda$, $\lambda \in \Lambda$, of $I$ by intervals $I_\lambda$,
if $s_\lambda \in
\sheaf{F}(I_\lambda)$, $\lambda \in \Lambda$, are sections of
$\sheaf{F}$ such that $s_\lambda|_{ I_\lambda \cap I_{\lambda'}} = 
s_{\lambda'}|_{ I_\lambda \cap I_{\lambda'}}$,
$\forall \lambda, \lambda' \in \Lambda$, then there is a section
$s \in \sheaf{F}(I)$ such that $s|_{I_\lambda} = s_\lambda$, 
$\forall \lambda \in \Lambda$. 
\end{enumerate}
\end{definition}

\noindent
Note that property (ii) differs from the usual gluing property for sheaves
because of the connectivity condition on the coverings. The next example shows
that the sections of a $c$-module in general do not satisfy the usual gluing property.

\begin{example} \label{E:sheaf}
Using the notation introduced in Definition \ref{D:interval}, let
$\cmod{V} = \cmod{I}[-1^-, 0^-] \oplus \cmod{I}[0^-, 1^+]$ be the direct sum of the interval 
$c$-modules of type $\cmod{I} [\, , ]$ associated with the intervals
$[-1,0)$ and $[0,1]$, respectively. Then, the constant sections $s_1$ over $[-1,0)$ and $s_2$
over $[0,1]$ defined by $s_1 \equiv (1,0)$ and $s_2 \equiv (0,1)$ have disjoint domains.
However, $s_1 \cup s_2$ is not a section over the interval $[-1,1] = [-1,0) \cup [0,1]$.
\end{example}

\begin{example} \label{E:open}
Here we describe a $c$-module $\cmod{V}$ whose $p$-sheaf $F$ of sections 
is  non-trivial, but the space of sections over any open interval is trivial. 
Thus, $p$-sheaves contains finer structural information than sheaves over
open sets of $\real$ equipped with the standard topology. Let
$\cmod{V} := \cmod{I} [0^-, 0^+]$ be the  interval module over the singleton $I=\{0\}$
introduced in Definition \ref{D:interval}. For any  open interval $J \subseteq \real$, 
we have $F(J) = 0$ because the relation $v_s^t = 0$, for any $s < t$. On the 
other hand, the space of sections $F(0)$ is a 1-dimensional vector space.
\end{example}

\begin{proposition}
The persistence presheaf of sections of a $c$-module $\cmod{V}$ is a $p$-sheaf.
\end{proposition}
\begin{proof}
Locality is clearly satisfied, so we verify the connective gluing property. Let
$I = \cup I_\lambda$, $\lambda \in \Lambda$, be a connected covering of
an interval $I$ by intervals $I_\lambda$, and let $s_\lambda \in \sheaf{V} (I_\lambda)$,
$\lambda \in \Lambda$, be sections of $\cmod{V}$ that agree on overlaps. There
is a well-defined family $s = (v_t)$, $t \in I$, such that
$s|_{I_\lambda} = s_\lambda$, $\forall \lambda \in \Lambda$. We need to show that
$s$ is a section. Let $S$ be the collection of all sections $r$ of $\cmod{V}$
satisfying $r \cle s$; that is, sections $r$ that coincide with the restriction of $s$ to some
subinterval of $I$. $S$ is non-empty because the restriction of $s$ to any $I_\lambda$
gives a section. Moreover, $(S, \cle)$ is a partially ordered set such that each chain in 
$S$ has an  upper bound in $S$. By Zorn's lemma, $S$ contains a maximal element $\hat{s}$
defined on an interval $J \subseteq I$. By construction, for each $\lambda \in \Lambda$,
$I_\lambda \cap J = \emptyset$ or $I_\lambda \subseteq J$, for otherwise,
we would be able to extend $\hat{s}$ to a larger interval. Since $J \ne \emptyset$,
Lemma \ref{L:cover} ensures that $J=I$, proving that $s = \hat{s}$.
\end{proof}

\begin{definition}
Let $\sheaf{F}_\lambda$, $\lambda \in \Lambda$, be $p$-sheaves. The direct sum
$\oplus \sheaf{F}_\lambda$ is the $p$-sheaf defined by:
\begin{enumerate}[(i)]
\item $(\oplus \sheaf{F}_\lambda) (I) = \oplus \sheaf{F}_\lambda (I)$, for any interval $I$;
\item $(\oplus \sheaf{F}_\lambda) (I \subseteq J) = 
\oplus \sheaf{F}_\lambda (I \subseteq J)$.
\end{enumerate} 
\end{definition}

\begin{definition}
Let $\sheaf{F}$ and $\sheaf{G}$ be $p$-sheaves. $\sheaf{F}$ is a {\em subsheaf} of
$\sheaf{G}$ if (i) for any interval $I$, $\sheaf{F}(I) \subseteq \sheaf{G}(I)$ and
(ii) for any $I \subseteq J$ and $s \in \sheaf{F}(J)$, $\sheaf{F}^J_I (s) = 
\sheaf{G}^J_I (s)$.
\end{definition}

It is easy to verify that the intersection of a collection of subsheaves of a
$p$-sheaf $\sheaf{F}$ is also a subsheaf of $\sheaf{F}$.


\begin{definition} \label{D:span}
Let $S$ be a set of sections of a $p$-sheaf $\sheaf{F}$. The {\em $p$-sheaf
generated by $S$} is the intersection of all subsheaves $\sheaf{G}$ of $\sheaf{F}$
with the property that if $s \in S$, then $s$ is a section of $\sheaf{G}$. This subsheaf of
$\sheaf{F}$ is denoted $\sheaf{F}\langle S \rangle$. 
\end{definition}

\begin{definition} \label{D:intsheaf}
Let $(p,q) \in \dec^2$ be an interval. We define the {\em interval $p$-sheaves}
$k[p,q]$, $k[p,q\rangle$, $k\langle p,q]$ and $k\langle p,q\rangle$ associated
with $(p,q)$, as follows: 
\begin{enumerate}[(i)]
\item $k[p,q](I)=k$ if $I\subseteq (p,q)$ and $k[p,q](I)=0$ otherwise. The
element $1 \in k[p,q](p,q)$ is called the {\em unit section} of $k[p,q]$.
\item $k[p,q\rangle(I)=k$ if $I\subseteq (p,+\infty)$ and $I\cap(p,q)\neq \emptyset$,
and $k[p,q\rangle(I)=0$ otherwise. The element $1 \in k[p,q\rangle (p, +\infty)$
is called the {\em unit section} of $k[p,q\rangle$.
\item $k\langle p,q](I)=k$ if $I\subseteq (-\infty,q)$ and $I\cap(p,q)\neq \emptyset$,
and $k\langle p,q](I)=0$ otherwise. The element of $1 \in k \langle p,q] (-\infty, q)$
is called the {\em unit section} of $k \langle p,q]$.
\item $k\langle p,q\rangle(I)=k$ if $I\cap(p,q)\neq \emptyset$ and
$k\langle p,q\rangle(I)=0$ otherwise. The element $1 \in k \langle p,q\rangle
(-\infty, +\infty)$ is called the {\em unit section} of $k \langle p,q\rangle$.
\item If $\sheaf{F}$ is any of the above $p$-sheaves and $I \subseteq J$,
then $\sheaf{F}^J_I$ is the identity map if $\sheaf{F}(J)=\sheaf{F}(I)=k$, and
$\sheaf{F}^J_I$ is the trivial map, otherwise.
\end{enumerate}
We refer to each of these four possibilities as the {\em type} of the interval
$p$-sheaf associated with $(p, q)$.
\end{definition}

\begin{proposition} \label{P:sinterval}
For any interval $(p,q) \in \dec^2$, $ p < q$, the following holds: 
\begin{enumerate}[\rm (i)]
\item $k[p,q]$ is isomorphic to the $p$-sheaf of sections of $\cmod{I}[p,q]$;
\item $k[p,q\rangle$ is isomorphic to the $p$-sheaf of sections of $\cmod{I}[p,q\rangle$;
\item $k\langle p,q]$ is isomorphic to the $p$-sheaf of sections of $\cmod{I}\langle p,q]$;
\item $k\langle p,q\rangle $ is isomorphic to the $p$-sheaf of sections of
$\cmod{I}\langle p,q\rangle$.
\end{enumerate}
\end{proposition}
\begin{proof}
The proof is straightforward.
\end{proof}

\begin{proposition} \label{P:isheaf}
Let $s$ be a section of a $p$-sheaf $\sheaf{F}$ with $\dom (s) = (x,y) \in \dec^2$,
$x < y$, and let $\sheaf{F}\langle s \rangle$ be the subsheaf generated by $s$.
Suppose that $\supp [s]$ is an interval $(p, q) \in \dec^2$, where $x \leq p < q \leq y$. 
Then, the following statements hold:

\begin{enumerate}[\rm (i)]
\item If $x=p$ and $q<y$, then $\sheaf{F}\langle s \rangle$ is isomorphic to
$k[p,q\rangle$.
\item If $x<p$ and $q=y$, then $\sheaf{F}\langle s \rangle$ is isomorphic to
$k\langle p,q]$.
\item If $x<p$ and $q<y$, then $\sheaf{F}\langle s \rangle$ is isomorphic to
$k\langle p,q\rangle$.
\item If $x=p$ and $q=y$, then $\sheaf{F}\langle s \rangle$ is isomorphic to
$k[p,q]$.
\end{enumerate}
\end{proposition}
\begin{proof}
For statement (i), by the connective gluing property, there is a section
$\overline{s}$ over $(p,+\infty)$ such that $\overline{s}|_{(p,y)}=s$ and
$\overline{s}|_{(q,+\infty)}=0$. Note that $\overline{s}$ must be a section of any
subsheaf of $\sheaf{F}$ having $s$ as a section. For an interval $I\subseteq(p,+\infty)$,
we denote by
$\langle \overline{s}|_I\rangle$ the subspace of $\sheaf{F} (I)$ spanned
by $\overline{s}|_I$. Let $\cat{Int}_{p,q}$ be the collection of all intervals satisfying
$I \subseteq(p,+\infty)$ and $I\cap(p,q)\neq \emptyset$. Define a persistence presheaf
$\sheaf{G} \subseteq \sheaf{F}$ by $\sheaf{G}(I)=\langle \overline{s}|_I \rangle$ if 
$I \in \cat{Int}_{p,q}$, and $\sheaf{G}(I)=0$ otherwise. Note that $\overline{s}|_I\neq 0$, 
$\forall I \in \cat{Int}_{p,q}$, and $\overline{s}|_I = 0$ if $I \notin \cat{Int}_{p,q}$.
By construction, $G$ is a sub-presheaf of $F\langle s \rangle$.
	
Let $\Phi \colon \sheaf{G} \to k[p,q\rangle$ be
the homomorphism defined as follows: $\phi_I (\overline{s}|_I) = 1$ if $I \in \cat{Int}_{p,q}$,
and $\phi_I=0$, otherwise. Similarly, define $\Psi \colon k[p,q\rangle\rightarrow
\sheaf{G}$ by $\psi_I (1) = \overline{s}|_I$ if $I \in \cat{Int}_{p,q}$ and $\psi^I_I=0$,
otherwise. Then, for any interval $I$, $\phi_I$ and $\psi_I$ are
mutual inverses, so $\Phi$ is a presheaf isomorphism. Since $k[p,q\rangle$ is
a sheaf, $\sheaf{G}$ is not just a presheaf, but a subsheaf of $\sheaf{F}$.
Thus, $G = \sheaf{F} \langle s \rangle$,
showing that $\sheaf{F}\langle s \rangle$ is isomorphic to $k [p,q\rangle$.
The proofs for the other cases are similar.
\end{proof}


\section{Tameness} \label{S:tameness}
This section discusses tameness conditions for $p$-sheaves and $c$-modules
under which we prove interval decomposition theorems in Section \ref{S:decomposition}.

\begin{definition} \label{D:descend1} 
Let $\sheaf{F}$ be a $p$-sheaf.
\begin{enumerate}[(i)]
\item $\sheaf{F}$ satisfies the {\em descending chain condition (DCC) on images} if
for any ascending sequence $I_1\subseteq I_2\subseteq I_3\subseteq \ldots$ of 
intervals and any interval $I\subseteq\cap I_n$, the chain $\im(\sheaf{F}_{I}^{I_1})
\supseteq \im(\sheaf{F}_{I}^{I_2}) \supseteq \ldots$ is stable, that is, it eventually 
becomes constant;
\item $\sheaf{F}$ satisfies the {\em descending chain condition (DCC) on kernels} if
for any ascending sequence $I_1\subseteq I_2\subseteq I_3\subseteq \ldots$ of intervals 
and any interval $I \supseteq \cup I_n$, the chain $\ker (\sheaf{F}^{I}_{I_1})
\supseteq \ker(\sheaf{F}^{I}_{I_2}) \supseteq \ldots$ is stable.
\end{enumerate}
\end{definition}

\begin{definition} \label{D:descend2}
Let $\cmod{V}$ be a $c$-module and $t \in \real$.

\begin{enumerate}[(i)]
\item $\cmod{V}$ satisfies the {\em descending chain condition (DCC) on images} at
$t$ if for any sequence $\ldots \leq \ell_2 \leq \ell_1 \leq t$,
the chain  $\im (v_{\ell_1}^t) \supseteq \im (v_{\ell_2}^t) \supseteq \ldots$ stabilizes
in finitely many steps, and for any sequence $t \leq u_1 \leq u_2 \leq \ldots$,
the chain $\dom(v_t^{u_1}) \supseteq \dom(v_t^{u_2}) \supseteq \ldots$ is stable.

\item $\cmod{V}$ satisfies the {\em descending chain condition (DCC) on kernels}
at $t$ if for any sequence $t \leq \ldots \leq u_2 \leq u_1$, the chain  
$\ker (v_t^{u_1}) \supseteq \ker (v_t^{u_2}) \supseteq \ldots$ is stable,
and for any sequence $\ell_1 \leq \ell_2 \leq \ldots \leq t$,
the chain  $\ker \, (v_{\ell_1}^t)^\ast \supseteq \ker \, (v_{\ell_2}^t)^\ast \supseteq \ldots$
is stable. Here, $^\ast$ denotes the operation of reversing
correspondences (see Definition \ref{D:correspondence}).
\end{enumerate}
\end{definition}

\begin{definition}\label{D:tame} (Tameness)

\begin{enumerate}[(i)]
\item A $p$-sheaf $\sheaf{F}$ is {\em tame} if it satisfies the DCC
on both images and kernels.

\item A $c$-module is {\em virtually tame} if it satisfies
the DCC on both images and kernels at each $t \in \real$.

\item A $c$-module is {\em tame} if its $p$-sheaf of sections is tame.
\end{enumerate}
\end{definition}

\begin{remark}
If a $c$-module $\cmod{V}$ has the property that all correspondences $v_s^t$,
$s < t$, are finite-dimensional, then $\cmod{V}$ is virtually tame. In particular, a
pointwise finite-dimensional $c$-module is virtually tame.
\end{remark}

The next examples show that the sheaf of sections of a virtually tame $c$-module
is not necessarily tame.

\begin{example} \label{E:vtame1}
Let $k$ be a field. Consider the $c$-module $\cmod{V}$ given by $V_t = k$,
$\forall t \in \real$, with connecting morphisms $v_s^t = k \times k$, for any $s < t$,
and $v_t^t = \Delta_k$, the graph of the identity map, $\forall t \in \real$.
$\cmod{V}$ is virtually tame because each $V_t$ is 1-dimensional. Furthermore, it 
admits the interval decomposition $\cmod{V} \cong \oplus_{x \in \real} \langle x \rangle$, 
where $\langle x \rangle$ denotes the interval module of type $\langle \ \rangle$
supported on the singleton $\{x\}$.
Note, however, that $F \ncong \oplus_{x \in \real} F_x$, where $F_x$ is the $p$-sheaf
of sections of $\langle x \rangle$. Indeed, for any interval $I \subseteq \real$, 
$F(I)$ comprise all sequences $(v_t)$ with $v_t \in k$ and $t \in I$. On the
other hand, the sections of the direct sum $p$-sheaf are the sections over $I$
whose supports are finite sets. Although $\sheaf{F}$ satisfies the 
DCC on images, $\sheaf{F}$ does not satisfy the DCC on kernels
and therefore is not tame. Indeed, consider the chain $I_n = [0,n]$ and let
$I = [0, +\infty)$. Then, the chain $\ker (\sheaf{F}^I_{I_n})$ is not stable.
\end{example}

\begin{example} \label{E:vtame2}
This example shows that the virtual tameness of $\cmod{V}$ does not imply
the tameness of $\sheaf{F}$ even for persistence modules. Let $I_n = (0, 1/n]$,
$\cmod{I}_n := \cmod{I} [0^+, 1/n^+\rangle$ be the associated interval $p$-module of
type $[ \ \rangle$, and
$\cmod{V} = \oplus_n \cmod{I}_n$. $\cmod{V}$ is pointwise finite dimensional
and thus virtually tame. However, its sheaf of sections is not tame, as the DCC 
on kernels is not satisfied. Indeed, let $I = (0, +\infty)$ and consider the chain
$J_n = (1/n, +\infty)$. Then, the chain $\ker (\sheaf{F}^I_{J_n})$ 
is strictly decreasing, thus not stable.
\end{example}

\begin{remark} \label{R:applications}
In spite of the above examples, if the $c$-module $\cmod{V}$  is pointwise 
finite-dimensional and there is a finite set   $T=\{t_1,\cdots,t_n\}\subseteq \real$ with 
$t_i<t_{i+1}$, $1 \leq i < n$, 
such  that if $s, t \in \real$ satisfy $s \leq t < t_1$, $t_n < s \leq t$, or  
$t_i < s \leq t < t_{i+1}$, for some $i$, then the relation $v_s^t$ is an 
isomorphism. Under this assumption, $\dim \sheaf{F} (I) < \infty$, for any 
interval $I$, implying that $\sheaf{F}$ is tame.
\end{remark}

\begin{proposition} \label{P:descending}
Let $\sheaf{F}$ be a $p$-sheaf and $I_1\subseteq I_2\subseteq I_3\subseteq \ldots$
an ascending sequence of intervals. Then, the following holds:
\begin{enumerate}[\rm (i)]
\item If $\sheaf{F}$ satisfies the DCC on images
and $I\subseteq \cap I_n$, then $\im(\sheaf{F}^{\cup I_n}_I) =
\im(\sheaf{F}^{I_m}_I)$ for $m$ large enough;
\item If $\sheaf{F}$ satisfies the DCC on kernels
and $\cup I_n \subseteq I$, then $\ker (\sheaf{F}_{\cup I_n}^I)
= \ker (\sheaf{F}_{I_m}^I)$ for $m$ large enough.
\end{enumerate}
\end{proposition}
\begin{proof}
(i) Set $I_0=I$. For any $n \geq 0$, the DCC on images,
applied to the interval $I_n$ and the chain $I_{n+1} \subseteq I_{n+2} \subseteq \ldots$, 
ensures that there exists $N (n) > n$ such that $\im(\sheaf{F}^{I_m}_{I_n}) = 
\im(\sheaf{F}^{I_{N(n)}}_{I_n})$, $\forall m\geq N(n)$. Set $V_n =
\im(\sheaf{F}^{I_{N(n)}}_{I_n}) \subseteq \sheaf{F}(I_n)$. Then,
$\sheaf{F}^{I_m}_{I_n}(V_m)=V_n$, for any $m\geq n$. By construction, for any
$s_0 \in V_0$, there is a sequence of sections $s_n \in V_n$, $n \geq 1$,
such that $\sheaf{F}^{I_m}_{I_n}(s_m)=s_n$, for $m \geq n$. By the connective
gluing property, there exists $s \in \sheaf{F}(\cup I_n)$ such that 
$\sheaf{F}^{\cup I_n}_{I_n}(s) = s_n$, $\forall n \geq 0$. Thus, 
$\sheaf{F}^{\cup I_n}_{I}(s)=s_0$. This implies that 
$\im(\sheaf{F}^{\cup I_n}_{I}) \supseteq V_0 = \im(\sheaf{F}^{I_m}_{I})$,
$\forall m\geq N(0)$. The inclusion $\im(\sheaf{F}^{\cup I_n}_{I})\subseteq
\im(\sheaf{F}^{I_m}_{I})$ is clearly satisfied $\forall m\geq N(0)$, so this
proves the claim.
	
(ii) By the DCC on kernels, we can choose $n_0 > 0$ such that
$\ker(\sheaf{F}^{I}_{I_n}) = \ker (\sheaf{F}^{I}_{I_{n_0}})$, for any $n \geq n_0$.
Clearly, $\ker (\sheaf{F}^{I}_{\cup I_n}) \subseteq 
\ker(\sheaf{F}^{I}_{I_{n_0}})$. For the opposite inclusion, let
$s \in \ker(\sheaf{F}^{I}_{I_{n_0}})$, so that $s|_{I_n}=0$, $\forall n>n_0$. 
By locality, $s|_{\cup I_n}=0$, which implies $s \in \ker(\sheaf{F}^{I}_{\cup I_n})$. 
Hence, $ \ker(\sheaf{F}^{I}_{I_{n_0}}) \subseteq \ker(\sheaf{F}^{I}_{\cup I_n})$,
concluding the proof.
\end{proof}

Our next goal is to prove a version of Proposition \ref{P:descending} for
virtually tame $c$-modules. We begin with an extension result for sections of
a $c$-module.

\begin{proposition}\label{P:extension}
Let $\cmod{V}$ be a $c$-module and $A \subseteq \real$. If $\cmod{V}$ is
virtually tame, then any section of $\cmod{V}$ over $A$ can be extended
to a section over an interval containing $A$.
\end{proposition}

\begin{proof}
Let $f$ be a section over $A$ and let $S$ be the set of all sections that extend
 $f$ over some set containing $A$. For $g,h\in S$, write $g\leq h$  to mean that
$h$ extends $g$. Note that $(S,\leq)$ is a poset in which each
ascending chain has an upper bound. By Zorn's lemma, there exists a maximal element
$\bar{f}\in S$. We claim that the domain of $\bar{f}$ is an interval. Suppose not,
then there exists $r \notin \dom(\bar{f})$ such that
$L = \{t \in \dom(\bar{f})\, |\, t < r\}$ and $U =\{t\in \dom(\bar{f}) \,|\, r < t\}$
are not empty. Choose an increasing sequence $\{\ell_n\}\subseteq L$ with
$\lim \ell_n=\sup L$ and a decreasing sequence $\{u_n\}\subseteq U$ with
$\lim u_n=\inf U$. If $\sup L \in L$, we assume that $\ell_n = \sup L$, $\forall n$.
Similarly, $u_n = \inf U$, for every $n$, if $\inf U \in U$. For $n \geq 1$, set
\begin{equation}
\im^{\ell_n,u_n}_r =\{v \in V_r \,|\, (\bar{f}(\ell_n),v) \in v_{\ell_n}^r
\ \text{and} \ (v, \bar{f}(u_n)) \in v_r^{u_n}\} \,,
\end{equation}
a non-empty affine subspace of $V_r$. Note that these affine subspaces
form a nested sequence
\begin{equation}
\im^{\ell_1,u_1}_r \supseteq \im^{\ell_2,u_2}_r \supseteq
\im^{\ell_3,u_3}_r\supseteq \ldots
\label{E:nested}
\end{equation}
We show that this sequence is stable. To this end, let
\begin{equation}
\ker^r_{\ell_n,u_n} = \{v \in V_r \, |\,  (0,v) \in v_{\ell_n}^r \ \text{and} \
(v, 0) \in v_r^{u_n} \} \,,
\end{equation}
which also form a nested sequence
\begin{equation}
\ker^r_{\ell_1,u_1} \supseteq \ker^r_{\ell_2,u_2} \supseteq \ker^r_{\ell_3,u_3}
\supseteq \ldots
\label{E:nested1}
\end{equation}
Note that the stability of the sequence $\ker^r_{\ell_n,u_n}$ implies the
stability of $\im^{\ell_n,u_n}_r$. Indeed, suppose $\exists N$ such that
$\ker^r_{\ell_n,u_n} = \ker^r_{\ell_N,u_N}$, $\forall n \geq N$. For any
$v_n \in \im^{\ell_n,u_n}_r$, we may write
\begin{equation}
\im^{\ell_n,u_n}_r = v_n + \ker^r_{\ell_n,u_n} \,.
\label{E:stab}
\end{equation}
Since the choice of $v_N \in \im^{\ell_N,u_N}_r$ in \eqref{E:stab} is arbitrary,
using \eqref{E:nested} and the stability of kernels, we have
\begin{equation}
\im^{\ell_N,u_N}_r = v_N + \ker^r_{\ell_N,u_N} = v_n + \ker^r_{\ell_N,u_N}
= v_n + \ker^r_{\ell_n,u_n} = \im^{\ell_n,u_n}_r \,,
\end{equation}
$\forall n \geq N$, as claimed.
The stability of \eqref{E:nested1} follows from
$\ker^r_{\ell_n,u_n} = \ker(v_{\ell_n}^r)^\ast \cap  \ker (v_r^{u_n})$ and the
fact that the right-hand side of this equation stabilizes by virtual tameness.
To conclude, pick $v \in \cap \, \im^{\ell_n,u_n}_r$ and extend $\bar{f}$ to $r$
via the assignment $\bar{f} (r) = v$. This contradicts the maximality of $\bar{f}$.
\end{proof}

\begin{proposition} \label{P:cdescending}
If $\sheaf{F}$ is the $p$-sheaf of sections of a virtually tame $c$-module $\cmod{V}$,
then for any ascending sequence of intervals $I_1\subseteq I_2\subseteq I_3\subseteq \ldots$
the following holds:
\begin{enumerate}[\rm (i)]
\item If $A \subseteq \cap I_n$ is finite, then $\im(F^{\cup I_n}_A) 
= \im(F^{I_m}_A)$ for $m$ large enough;
\item If $A \subseteq I\supseteq \cup I_n$, where $I$ is an interval and $A$ is
a finite set, then $\ker (F_{\cup I_n}^I)|_A = \ker (F_{I_m}^I)|_A$ for $m$
large enough.
\end{enumerate}
\end{proposition}
\begin{proof}
(i) Since $F^{\cup I_n}_A = F^{I_m}_A \circ F^{\cup I_n}_{I_m}$, it follows that
$\im(F^{\cup I_n}_A) \subseteq  \im(F^{I_m}_A)$, $\forall m$. For the
reverse inclusion, let $f \in \im(F^{I_m}_A)$. We show that, for $m$
sufficiently large, there is a section $g$ defined over $\cup I_n$ such
that $g|_A = f$. Let $s_0 = \min A$ and $t_0 = \max A$. We construct $g$
by gluing sections over the following three intervals:
$I_0 = [s_0, t_0]$, $I_- = \{x \in \cup I_n \colon x \leq s_0 \}$,
and $I_+ = \{x \in \cup I_n \colon x \geq t_0 \}$. 

Let $[s_n, t_n] \subseteq I_n$, $n \geq 1$, be a sequence of intervals such
that $[s_n, t_n] \subseteq [s_{n+1}, t_{n+1}]$, $A \subseteq \cap [s_n, t_n]$,
and $\cup [s_n, t_n] =  \cup I_n$, $\forall n \geq 1$. Note that we also have
$A \subseteq [s_0, t_0] \subseteq [s_1, t_1]$. For each $n \geq 0$, the virtual
tameness of $\cmod{V}$ implies that the descending chains
\[
\im (v_{s_{n + 1}}^{s_n}) \supseteq \im (v_{s_{n+2}}^{s_n}) \supseteq \ldots
\quad
\text{and}
\quad
\dom (v_{t_n}^{t_{n+1}}) \supseteq \dom (v_{t_n}^{t_{n+2}}) \supseteq \ldots
\]
are stable. Thus, $\exists N_n \geq n$ such that $\im (v_{s_m}^{s_n})$
and $\dom (v_{t_n}^{t_m})$ are constant for $m \geq N_n$. Set
$L_n := \im (v_{s_m}^{s_n})$ and $U_n := \dom (v_{t_n}^{t_m})$,
for $m \geq N_n$. By construction, for any $x_n \in L_n$ and $y_n \in U_n$,
$n \geq 0$,
there exist $x_{n+1} \in L_{n+1}$ and $y_{n+1} \in U_{n+1}$ such that
\begin{equation}
(x_{n+1}, x_n) \in v_{s_{n + 1}}^{s_n} 
\quad \text{and} \quad
(y_n, y_{n+1}) \in v_{t_n}^{t_{n+1}} . 
\label{E:sequence}
\end{equation}

If $f \in \im(F^{I_m}_A)$ and $m \geq N_0$, then $x_0 := f_{s_0} \in L_0$
and $y_0 = f_{t_0} \in U_0$. Iteratively, as described above,
construct $x_n \in L_n$ and $y_n \in U_n$ satisfying \eqref{E:sequence}. 
The sequences $\{x_n\}$ and $\{y_n\}$, $n \geq 0$, together with $f$,
yield a section of $\cmod{V}$ over the set
$\{s_n \colon n \geq 0\} \cup A \cup \{t_n \colon n \geq 0\}$. By
Proposition \ref{P:extension}, this section can be
extended to a section over $\cup I_n$ with the desired properties.

\medskip
	
(ii) $F^I_{I_m} = F^{\cup I_n}_{I_m} \circ F^I_{\cup I_n}$ implies that
$\ker (F^I_{\cup I_n}) \subseteq \ker(F^I_{I_m} )$. Therefore,
$\ker (F^I_{\cup I_n})|_A \subseteq \ker(F^I_{I_m})|_A$, for every $m$.
For the reverse inclusion, let  $A_0 = A \cap (\cup I_n)$,
$A_- = \{a \in A \colon a < t, \forall t \in \cup I_n\}$ and
$A_+ = \{a \in A \colon a > t, \forall t \in \cup I_n\}$. Set
$r_- = \max A_-$, $r_+ = \min A_+$, 
\[
I_- = I \setminus (r_- , + \infty) 
\quad \text{and} \quad
I_+ = I \setminus (-\infty, r_+) \,.
\]
Let $f \in \ker (F_{I_m}^I)|_A$. Since $A$ is finite, $\exists N_0 > 0$ such that
$A_0 \subseteq I_m$, for $m \geq N_0$. Hence, $f|_{A_0} \equiv 0$ if
$m \geq N_0$. Pick a section $\hat{f} \in \ker (F_{I_m}^I)$ such that
$\hat{f}|_A = f$ and set $g = \hat{f}|_{I_- \cup I_+}$. Next, we show that we
can extend $g$ to a section $h$ over $I_- \cup I_+  \cup_n I_n$ such that
$h|_{\cup I_n} \equiv 0$, provided that $m$ is sufficiently large. Note that
any such $h$ will have the property that $h|_A = f$.

Let $[s_n, t_n]$, $n \geq 1$, be as in the proof of (i). By construction,
$r_- < \ldots \leq s_2 \leq s_1$ and $t_1 \leq t_2 \leq \ldots < r_+$. By
virtual tameness, there exists $n_0 \geq N_0$ such that the chains 
\begin{equation}
\ker v_{r_-}^{s_1} \supseteq \ker v_{r_-}^{s_2} \supseteq \ldots
\quad \text{and} \quad
\ker (v_{t_1}^{r_+})^\ast \supseteq \ker (v_{t_2}^{r_+})^\ast \supseteq \ldots
\end{equation}
are stable at $s_n$ and $t_n$, $n \geq n_0$, respectively. Hence, if
$f \in \ker (F_{I_m}^I)|_A$, 
$m \geq n_0$, we have that $g_{r_-} \in \ker (v_{r_-}^{s_n})$ and
$g_{r_+} \in \ker (v^{r_+}_{t_n})^\ast$, for any $n \geq n_0$, which implies
that the section $g$ can be extended, as claimed. By Proposition
\ref{P:extension}, we can further extend $g$ to a section over $I$.
This concludes the proof.
\end{proof}


\section{Decomposition Theorems} \label{S:decomposition}

In this section, we prove interval decomposition theorems for tame $p$-sheaves
and virtually tame $c$-modules that lead to representations of their structures by
barcodes or persistence diagrams. We develop a sheaf-theoretical analogue of
the techniques employed by Crawley-Boevey to obtain such decompositions for
pointwise finite-dimensional persistence modules \cite{Crawley-Boevey2015}. 

\subsection{Coverings}
The arguments and constructions in
\cite{Crawley-Boevey2015} use the notion of {\em sections} of a vector space 
whose definition we recall next. To avoid confusion with sections of $c$-modules 
and $p$-sheaves, we rename them splittings. 

\begin{definition} Splittings of Vector Spaces (cf.\,\cite{Crawley-Boevey2015})

\begin{enumerate}[(i)]
\item A {\em splitting} of a vector space $V$ is a pair $(F^-, F^ +)$ of subspaces
$F^-\subseteq F^+ \subseteq V$.
\item A collection $\{(F_\lambda^-, F_\lambda^+) : \lambda \in \Lambda\}$ of
splittings of $V$ is {\em disjoint} if for all $\lambda\neq\mu$, either
$F_\lambda^+ \subseteq F_\mu^-$ or $F_\mu^+ \subseteq F_\lambda^-$;
\item A collection of splittings $\{(F_\lambda^-, F_\lambda^+) : \lambda \in \Lambda\}$
{\em covers $V$} if for any subspace $U \subseteq V$, with $U \neq V$, 
$\exists \lambda \in \Lambda$ such that $U + F_\lambda^- \neq U + F_\lambda^+$;
and it {\em strongly covers $V$} provided that for all subspaces $U, W \subseteq V$
with $W \nsubseteq U$, $\exists \lambda \in \Lambda$ such that
$U + (F_\lambda^- \cap W) \neq U + (F_\lambda^+ \cap W)$.
\end{enumerate}
\end{definition}

\begin{proposition}[Crawley-Boevey \cite{Crawley-Boevey2015}] \label{P:cover}
Let $\{(F_\lambda^-, F_\lambda^+) \colon \lambda \in \Lambda\}$ be a
set of splittings that is disjoint and covers $V$. 
\begin{enumerate}[\rm (i)]
\item If $W_\lambda$ is a complement of $F^-_\lambda$ in $F_\lambda^+$,
$\forall \lambda \in \Lambda$, then the inclusions $W_\lambda \subseteq V$ induce
a direct sum decomposition $V =\oplus_{\lambda\in\Lambda} W_\lambda$.
\item If $\{(G^-_\sigma,G^+_\sigma):\sigma\in\Sigma\}$ is another set of splittings
that is disjoint and strongly covers $V$, then the family
\[
\{(F_\lambda^- + G^-_\sigma\cap F_\lambda^+ , F_\lambda^- + G^+_\sigma\cap F_\lambda^+)
\,\colon(\lambda,\sigma)\in\Lambda\times\Sigma\}
\]
of splittings is disjoint and covers $V$.
\end{enumerate}
\end{proposition}

\begin{corollary}\label{C:directsum}
Suppose that $\{(F_\lambda^-, F_\lambda^+) : \lambda \in \Lambda\}$ is disjoint
and covers $V$ and $\{(G^-_\sigma,G^+_\sigma):\sigma\in\Sigma\}$
is disjoint and strongly covers $V$. If $W_{\sigma,\lambda}$ is a complement of 
$(F_\lambda^-\cap G^+_\sigma) + (G^-_\sigma\cap F_\lambda^+)$ in $G^+_\sigma
\cap F_\lambda^+$, then:
\begin{enumerate}[\rm (i)]
\item The inclusions $W_{\sigma,\lambda} \subseteq V$ induce a
decomposition $V=\oplus_{\sigma,\lambda}W_{\sigma,\lambda}$;
\item For any $\lambda \in \Lambda$, the inclusions
$F^-_\lambda, W_{\sigma,\lambda} \subseteq F^+_\lambda$ induce a direct
sum decomposition $F_\lambda^+ = F_\lambda^-\oplus
\left(\oplus_{\sigma}W_{\sigma,\lambda}\right)$.
\end{enumerate}
\end{corollary}
\begin{proof}
(i) Note that
\begin{equation} \label{E:iso}
\frac{G^+_\sigma\cap F_\lambda^+}{F_\lambda^-\cap G^+_\sigma + G^-_\sigma
\cap F_\lambda^+} 
\simeq \frac{F_\lambda^- + G^+_\sigma\cap F_\lambda^+}{F_\lambda^- +
G^-_\sigma\cap F_\lambda^+} \,.
\end{equation}
Since the isomorphism in \eqref{E:iso} is induced by inclusion, a complement
of $(F_\lambda^-\cap G^+_\sigma) + (G^-_\sigma\cap F_\lambda^+)$ in $G^+_\sigma
\cap F_\lambda^+$ is also a complement of $F_\lambda^- + G^-_\sigma\cap F_\lambda^+$
in $F_\lambda^- + G^+_\sigma\cap F_\lambda^+$. Thus, the claim follows from
Proposition \ref{P:cover}.

\smallskip

(ii) Since $W_{\sigma,\lambda}$ is a complement of 
$F_\lambda^- + G^-_\sigma\cap F_\lambda^+$ in $F_\lambda^- +
G^+_\sigma\cap F_\lambda^+$, we have $W_{\sigma,\lambda}\cap F_\lambda^-=0$
and $W_{\sigma,\lambda}\subseteq F_{\lambda}^+$, $\forall\sigma \in \Sigma$.
For a fixed $\lambda$,  to simplify notation, set $H_{\sigma}^-:=F_\lambda^- +
G^-_\sigma\cap F_\lambda^+$ and $H_{\sigma}^+:=F_\lambda^- + 
G^+_\sigma\cap F_\lambda^+$. Now, we show that
$\oplus_\sigma W_{\sigma,\lambda}\cap F_\lambda^-=0$. Let
$v_{\lambda} + v_{\sigma_1}+\cdots + v_{\sigma_n}=0$ be a relation
with $v_{\lambda}\in F_{\lambda}^-$ and $v_{\sigma_i} \in
W_{\sigma_i,\lambda}$. Since the splittings
$\{(G_\sigma^-, G_\sigma^+) \colon \sigma \in \Sigma\}$ of $V$ are disjoint,
so are the splittings $\{(H_\sigma^-, H_\sigma^+) \colon \sigma \in \Sigma\}$
of $F^+_\lambda$. By disjointness, we may assume that
$H_{\sigma_i}^+\subseteq H_{\sigma_{i+1}}^-$, for all $i<n$.  Then,
 $v_{\sigma_n} =- v_{\lambda} - v_{\sigma_1} - \cdots - v_{\sigma_{n-1}}
 \in H_{\sigma_n}^-$ because $v_\lambda \in F^-_\lambda \subseteq H_{\sigma_n}^-$
 and $v_{\sigma_i} \in H_{\sigma_i}^+ \subseteq  H_{\sigma_n}^-$. Hence,
 $v_{\sigma_n} \in H_{\sigma_n}^- \cap W_{\sigma_n, \lambda} = 0$.
Similarly, we show that all other terms in the relation vanish. Therefore,
$\oplus_\sigma W_{\sigma,\lambda}\cap F_\lambda^-=0$ and
$\oplus_\sigma W_{\sigma, \lambda} \subseteq F^+_\lambda$.

Choose $W_\lambda  \subseteq F^+_\lambda$ such that
 $W_\lambda\oplus\big(\oplus_{\sigma}W_{\sigma,\lambda}\big)$ is a complement
of $F_\lambda^-$ in $F_\lambda^+$. Since
$\{(F^-_\lambda, F^+_\lambda) \colon \lambda \in \Lambda\}$ is disjoint
and covers $V$, it follows that
\begin{equation}
\Big(\bigoplus_\lambda W_\lambda\Big) \oplus \Big(\bigoplus_{\sigma,\lambda}
W_{\sigma,\lambda}\Big)=V .
\end{equation}
On the other hand, by (i), $\oplus_{\sigma,\lambda}W_{\sigma,\lambda}=V$.
Thus, $W_\lambda=0$, $\forall \lambda \in \Lambda$. This proves that
$F_\lambda^+ = F_\lambda^-\oplus \left(\oplus_{\sigma}W_{\sigma,\lambda}\right)$.
\end{proof}

Let $\sheaf{F}$ be a $p$-sheaf and $x,y,p,q \in \dec$ satisfy
$x \leq p < q \leq y$. We use the following abbreviations:
\begin{align*}
\im^{(x,y)}_{(p,q)} &= \im \big(F^{(x,y)}_{(p,q)} \big) &
\ker^{(x,y)}_{(p,q)} &= \ker\big(F^{(x,y)}_{(p,q)}\big) \\
\im^{(\bar{x},y)}_{(p,q)} &= \cup_{z<x} \im^{(z,y)}_{(p,q)} &
\im^{(x,\bar{y})}_{(p,q)} &= \cup_{y<z} \im^{(x,z)}_{(p,q)} \\
\ker^{(x,y)}_{(\bar{p},q)} &= \cup_{x\leq z<p}
\ker^{(x,y)}_{(z,q)} &
\ker^{(x,y)}_{(p,\bar{q})} &=\cup_{q<z\leq y} 
\ker^{(x,y)}_{(p,z)} \,,
\end{align*}
with the convention that $\im^{(\bar{x},y)}_{(p,q)} = 0$ if  $x = -\infty$,
and $\im^{(x,\bar{y})}_{(p,q)} = 0$ if $y = +\infty$. Similarly,
for kernels, we make the convention that $\ker^{(x,y)}_{\ \, \emptyset} = F(x,y)$,
for any $x<y$.
\begin{align*}
\ker^{(x,\bar{y})}_{(p,q)} &= \im^{(x,\bar{y})}_{(x,y)}\cap \ker^{(x,y)}_{(p,q)}&
\ker^{(\bar{x},y)}_{(p,q)}&=\im^{(\bar{x},y)}_{(x,y)}\cap \ker^{(x,y)}_{(p,q)}\\
\ker^{(x,y)}_{)p,q(}&=\ker^{(x,y)}_{(x,p)}\cap \ker^{(x,y)}_{(q,y)}&
\ker^{(x,y)}_{)\bar{p},q(}&=\ker^{(x,y)}_{(x,\bar{p})}\cap \ker^{(x,y)}_{(q,y)} \\
\ker^{(x,y)}_{)p,\bar{q}(}&=\ker^{(x,y)}_{(x,p)}\cap \ker^{(x,y)}_{(\bar{q},y)} . &
\end{align*}

\begin{lemma}[Covering Lemma for Sheaves] \label{L:scover1}
Let $\sheaf{F}$ be a $p$-sheaf and $p,q \in \dec$ with $p < q $. Then,
\begin{enumerate}[\rm (i)]
\item $\{\big(\im ^{(\bar{x},q)}_{(p,q)}, \im^{(x,q)}_{(p,q)}\big),
x\leq p\}$ is a disjoint set of splittings of $\sheaf{F}(p,q)$;
\item $\{\big(\im^{(p,\bar{y})}_{(p,q)}, \im^{(p,y)}_{(p,q)}\big),
y \geq q\}$ is a disjoint set of splittings of $\sheaf{F}(p,q)$;
\item $\{\big(\ker_{(-\infty,\bar{x})}^{(-\infty,q)}\big|_I,
\ker_{(-\infty,x)}^{(-\infty,q)}\big|_I\big), x<q\}$ is a disjoint set of splittings of
$\sheaf{F}(-\infty,q)\big|_I$, for any interval $I \subseteq (-\infty, q)$;
\item $\{\big(\ker^{(p,+\infty)}_{(\bar{y},+\infty)}\big|_I,
\ker^{(p,+\infty)}_{(y,+\infty)}\big|_I \big), y>p\}$ is a disjoint set of splittings of
$\sheaf{F}(p,+\infty)\big|_I$, for any interval $I \subseteq (p, +\infty)$.
\end{enumerate}
Furthermore, if $\sheaf{F}$ satisfies the DCC on images,
the splittings in (i) and (ii) form strong coverings. Similarly, if $\sheaf{F}$ satisfies
the DCC on kernels, the splittings in (iii) and (iv) form
strong coverings. In particular, if $\sheaf{F}$ is tame, each of the above
sets of splittings strongly covers the corresponding vector space of sections.
\end{lemma}
\begin{proof}
(i) For $x \leq p$, write $F^-_x = \im^{(\bar{x},q)}_{(p,q)}$ and
$F^+_x = \im^{(x,q)}_{(p,q)}$. It is simple to check that
$\{ (F^-_x, F^+_x) \colon  x \leq p\}$ is a disjoint set of splittings. To verify the
strong covering property under the assumption that $\sheaf{F}$ satisfies the
DCC on images,  let $U, W \subseteq \sheaf{F}(p,q)$
be subspaces with $W \nsubseteq U$. Set
\begin{equation}
S = \{(z,q) \, \colon z \leq p \mbox{ and } \im ^{(z,q)}_{(p,q)} \cap W \nsubseteq U\}
\end{equation}
Note that $S \ne \emptyset$ because $(p, q) \in S$.
Let $I_1\subseteq I_2\subseteq I_3\subseteq \ldots$ be a sequence of intervals from
$S$ such that  $\cup_{I \in S} I=\cup_{n \geq 1} I_n$. Write the interval $\cup I_n$
as $\cup I_n= (x_0,q)$, with $x_0 \leq p$. By Proposition \ref{P:descending},
\begin{equation} \label{E:cont1}
F^+_{x_0} \cap W = \im^{(x_0,q)}_{(p,q)} \cap W =\im^{I_m}_{(p,q)}
\cap W \nsubseteq U \,,
\end{equation}
for $m$ sufficiently large. On the other hand, for any $z<x_0$, we have that
$(z,q) \notin S$. This implies that $\im ^{(z,q)}_{(p,q)}\cap W \subseteq U$. 
Therefore,
\begin{equation} \label{E:cont2}
F^-_{x_0} \cap W = \im^{(\bar{x}_0,q)}_{(p,q)} \cap W = \cup_{z<x_0}\im ^{(z,q)}_{(p,q)}
\cap W \subseteq U \,.
\end{equation}
It follows from \eqref{E:cont1} and \eqref{E:cont2} that
\begin{equation}  \label{E:scover}
U + F^-_{x_0} \cap W \ne U + F^+_{x_0} \cap W ,
\end{equation}
concluding the argument. The proofs of the other statements are similar.
\end{proof}
\begin{lemma}[Covering Lemma for Modules] \label{L:scover2}
Let $\sheaf{F}$ be the $p$-sheaf of sections of a $c$-module $\cmod{V}$
and $p, q \in \dec$ with $p < q$. Then,
\begin{enumerate}[\rm (i)]
\item For any $A \subseteq (p, q)$,
$\{\big(\im^{(\bar{x},q)}_{(p,q)}\big|_A, \im^{(x,q)}_{(p,q)}\big|_A\big),
x\leq p\}$ is a disjoint set of splittings of $\sheaf{F}(p,q)\big|_A$;
\item For any $A \subseteq (p, q)$,
$\{\big(\im^{(p,\bar{y})}_{(p,q)}\big|_A, \im^{(p,y)}_{(p,q)}\big|_A\big),
y\geq q\}$ is a disjoint set of splittings of $\sheaf{F}(p,q)\big|_A$;
\item  For any $A \subseteq (-\infty, q)$, 
$\{\big(\ker_{(-\infty,\bar{x})}^{(-\infty,q)}\big|_A,
\ker_{(-\infty,x)}^{(-\infty,q)}\big|_A\big), x<q\}$ is a disjoint set of splittings
of $\sheaf{F}(-\infty,q)\big|_A$;
\item For any $A \subseteq (p, +\infty)$,
$\{\big(\ker^{(p,+\infty)}_{(\bar{y},+\infty)}\big|_A,
\ker^{(p,+\infty)}_{(y,+\infty)}\big|_A\big), y>p\}$ is a disjoint set of splittings
of $\sheaf{F}(p,+\infty)\big|_A$.
\end{enumerate}
Moreover, if $\cmod{V}$ is virtually tame and $A$ is finite,
then each of the above set of splittings is a strong cover.
\end{lemma}
\begin{proof}
The proof is nearly identical to that of Lemma \ref{L:scover1}. The only changes
needed are to restrict the relevant vector spaces of sections to $A$ and to use
Proposition \ref{P:cdescending} in lieu of Proposition \ref{P:descending} in
the tameness argument.
\end{proof}

\subsection{Decomposition of Tame $p$-Sheaves}
Our next goal is to decompose a tame $p$-sheaf $\sheaf{F}$ as a direct sum
of ``atomic'' subsheaves that are sheaf-theoretical analogues of interval
$c$-modules.  The building blocks of this decomposition are described next.

\begin{definition}\label{D:xysheaf}
For any $p$-sheaf $\sheaf{F}$ and $x,y\in\dec$ with $x<y$, we let:
\begin{enumerate}[\rm (i)]
\item $\sheaf{F}[x,y]$ be a complement of $\im^{(\bar{x},y)}_{(x,y)}+\im^{(x,\bar{y})}_{(x,y)}$
in $\im^{(x,y)}_{(x,y)} = F(x,y)$;
\item $\sheaf{F}[x,y\rangle$ be a complement of
$\ker^{(\bar{x},+\infty)}_{(y,+\infty)}+\ker^{(x,+\infty)}_{(\bar{y},+\infty)}$ in $\ker^{(x,+\infty)}_{(y,+\infty)}$;
\item $\sheaf{F}\langle x,y]$ be a complement of
$\ker^{(-\infty,\bar{y})}_{(-\infty,x)}+\ker^{(-\infty,y)}_{(-\infty,\bar{x})}$ in $\ker^{(-\infty,y)}_{(-\infty,x)}$;
\item $\sheaf{F}\langle x,y\rangle$ be a complement of
$\ker^{(-\infty,+\infty)}_{)\bar{x},y(}+\ker^{(-\infty,+\infty)}_{)x,\bar{y}(}$ in $\ker^{(-\infty,+\infty)}_{)x,y(}$.
\end{enumerate}
\end{definition}

\begin{proposition} \label{P:xysheaf}
For any $p$-sheaf $\sheaf{F}$, the following statements hold:

\begin{enumerate}[\rm (i)]

\item If $x, y \in \dec$, $x < y$, and $I \subseteq (x,y)$ is an interval, then
\[
\im^{(x,y)}_{(x,y)}\big|_I = \sheaf{F}[x,y]\big|_I \oplus 
\left(\im^{(\bar{x},y)}_{(x,y)}\big|_I + \im^{(x,\bar{y})}_{(x,y)}\big|_I\right) ;
\]
\item If $x, y \in \dec$, $x < y$, and $I \subseteq (x,+\infty)$ is an interval, then
\[
\ker^{(x,+\infty)}_{(y,+\infty)}\big|_I = \sheaf{F}[x,y\rangle\big|_I 
\oplus \left(\ker^{(\bar{x},+\infty)}_{(y,+\infty)}\big|_I + \ker^{(x,+\infty)}_{(\bar{y},+\infty)}\big|_I\right);
\]
\item If $x, y \in \dec$, $x < y$, and $I \subseteq (-\infty, y)$ is an interval, then
\[
\ker^{(-\infty,y)}_{(-\infty,x)}\big|_I = \sheaf{F}\langle x,y]\big|_I 
\oplus \left(\ker^{(-\infty,\bar{y})}_{(-\infty,x)}\big|_I + \ker^{(-\infty,y)}_{(-\infty,\bar{x})}\big|_I\right);
\]
\item If $x, y \in \dec$, $x < y$, and $I$ is any interval, then
\[
\ker^{(-\infty,+\infty)}_{)x,y(}\big|_I = \sheaf{F}\langle x,y\rangle\big|_I
\oplus \left(\ker^{(-\infty,+\infty)}_{)\bar{x},y(}\big|_I + \ker^{(-\infty,+\infty)}_{)x,\bar{y}(}\big|_I\right).
\]
\end{enumerate}
\end{proposition}
\begin{proof}
Here we just prove (ii), the proofs of the other statements being similar.
Since, by definition, $\ker^{(x,+\infty)}_{(y,+\infty)} = \sheaf{F}[x,y\rangle \oplus
\left(\ker^{(\bar{x},+\infty)}_{(y,+\infty)} + \ker^{(x,+\infty)}_{(\bar{y},+\infty)}\right)$,
we just need to show that the two summands in this direct sum decomposition remain
independent after restriction to $I$; that is,
\begin{equation}
\sheaf{F}[x,y\rangle\big|_I \cap  \left(\ker^{(\bar{x},+\infty)}_{(y,+\infty)}\big|_I +
\ker^{(x,+\infty)}_{(\bar{y},+\infty)}\big|_I\right)=0.
\label{E:int1}
\end{equation}

Given $s_0\in \sheaf{F}[x,y\rangle$ and $s_1\in \left(\ker^{(\bar{x},+\infty)}_{(y,+\infty)}
+ \ker^{(x,+\infty)}_{(\bar{y},+\infty)}\right)$ with $s_0|_I = s_1|_I$, let $s=s_0-s_1$.
Clearly, $s|_I=0$. If $I\cap(x,y)=\emptyset$, the left-hand side of \eqref{E:int1} equals $0$, so
the proof is trivial. If $I\cap(x,y)\neq \emptyset$, by the gluing property, $s$ is the sum of two
sections from $\ker^{(\bar{x},+\infty)}_{(y,+\infty)}+\ker^{(x,+\infty)}_{(\bar{y},+\infty)}$.
Hence, $s\in \ker^{(\bar{x},+\infty)}_{(y,+\infty)}+\ker^{(x,+\infty)}_{(\bar{y},+\infty)}$, which
implies that $s_0=s_1+s\in \ker^{(\bar{x},+\infty)}_{(y,+\infty)}+\ker^{(x,+\infty)}_{(\bar{y},+\infty)}$,
so that $s_0=s_1=0$. The result follows.
\end{proof}


For $x,y \in \dec$, $x<y$, we may view $\sheaf{F}[x,y\rangle$ as a persistence presheaf with
$\sheaf{F}[x,y\rangle (I) = \sheaf{F}[x,y\rangle|_I$, if $I \subseteq (x, +\infty)$,
and $\sheaf{F}[x,y\rangle (I) = 0$, otherwise. Morphisms are induced by restriction
of sections of $\sheaf{F}$. Similarly, we may treat $\sheaf{F}\langle x,y]$,
$\sheaf{F}\langle x,y\rangle$ and $\sheaf{F}[x,y]$ as persistence presheaves, where
$\sheaf{F}\langle x,y] (I)=0$  if $I \nsubseteq (-\infty,y)$ and $\sheaf{F}[x,y] (I)=0$
if $I \nsubseteq (x,y)$. Henceforth,
we refer to these interchangeably as $p$-sheafs or spaces of sections, the meaning 
determined by the context.

Let $m[x,y] = \dim(\sheaf{F}[x,y])$, $m[x,y\rangle = \dim(\sheaf{F}[x,y\rangle) $, 
$m\langle x,y] = \dim(\sheaf{F}\langle x,y])$,
and $m\langle x,y\rangle = \dim(\sheaf{F}\langle x,y \rangle) $. 
\begin{proposition} \label{P:intervals}
If $\sheaf{F}$ is a $p$-sheaf and $x, y \in \dec$, then:
\begin{enumerate}[\rm (i)]
\item $\sheaf{F}[x,y] \cong m[x,y] \, k[x,y]$, if $-\infty<x<y<+\infty$;
\item $\sheaf{F}[x,y\rangle \cong m[x,y\rangle\, k[x,y\rangle$, if 
$-\infty<x<y\leq +\infty$;
\item $\sheaf{F}\langle x,y] \cong m\langle x,y]\, k\langle x,y]$,
if $-\infty\leq x<y<+\infty$;
\item $\sheaf{F}\langle x,y\rangle \cong m\langle x,y\rangle\, k\langle x,y\rangle$,
if $-\infty\leq x<y\leq +\infty$.
\end{enumerate}
\end{proposition}
\begin{proof}
We prove (ii), the proofs of the other statements being similar.
Choose a basis $B = \{s_\lambda \colon \lambda \in \Lambda\}$ of the space of
sections $\sheaf{F}[x,y\rangle$.  By definition, for each $\lambda \in \Lambda$,
$\dom(s_\lambda)=(x,+\infty)$. Note that $\supp [s_\lambda] = (x,y)$. Indeed,
suppose $\exists t \in (x,y)$ such that $s_\lambda (t) = 0$. Then, we may write
$s_\lambda$ as the sum of sections in $\ker^{(\bar{x},+\infty)}_{(y,+\infty)}$ and
$\ker^{(x,+\infty)}_{(\bar{y},+\infty)}$, which implies that $s_\lambda = 0$, a contradiction.
By Proposition \ref{P:isheaf}, the $p$-sheaf generated by $s_\lambda$ satisfies
$\sheaf{F}\langle s_\lambda \rangle\cong k[x,y\rangle$. Since $B$ is a basis,
$\sheaf{F}[x,y\rangle\cong m[x,y\rangle \, k[x,y\rangle$. 
\end{proof}

\begin{lemma}[Decomposition Lemma] \label{L:sheafsum}
Let $\sheaf{F}$ be a tame $p$-sheaf. If $p, q \in \dec$, $p < q$, then the space
of sections $\sheaf{F}(p,q)$ may be decomposed as
\[
\begin{split}
\sheaf{F}(p,q) = &\bigoplus_{-\infty<x\leq p<q\leq y<+\infty} \sheaf{F}[x,y]\big|_I
\bigoplus_{-\infty<x\leq p<y\leq+\infty} \sheaf{F}[x,y\rangle\big|_I \\
&\ \,  \bigoplus_{-\infty\leq x<q\leq y<+\infty} \sheaf{F}\langle x,y]\big|_I 
\ \ \, \bigoplus_{-\infty\leq x<y\leq +\infty} \sheaf{F}\langle x,y\rangle\big|_I \,,
\end{split}
\]
where $I = (p,q)$.
\end{lemma}
\begin{proof}
The proof of the lemma is in three steps, each providing a decomposition
of $\sheaf{F}(p,q)$ that is gradually refined to the target decomposition.

\medskip
\noindent
{\em Step 1.} Since $\sheaf{F}$ satisfies the DCC on
images, by Lemma \ref{L:scover1} the families of splittings
\begin{equation}
(F^-_x, F^+_x) = 
\big(\im^{(\bar{x},q)}_{(p,q)}, \im^{(x,q)}_{(p,q)} \big) ,
\end{equation}
$-\infty\leq x \leq p$, and
\begin{equation}
(G^-_y, G^+_y) = \big(\im^{(p,\bar{y})}_{(p,q)}, \im^{(p,y)}_{(p,q)} \big) ,
\end{equation} 
$q\leq y\leq +\infty$, are disjoint and strongly cover $\sheaf{F}(p,q)$.
By the connective gluing property for sections, we have:
\begin{equation}
\begin{split}
F_x^+ \cap G_y^+ &= \im^{(x,q)}_{(p,q)} \cap \im^{(p,y)}_{(p,q)}
= \im^{(x,y)}_{(p,q)}, \\
F_x^- \cap G_y^+ &= \im^{(\bar{x},q)}_{(p,q)} \cap
\im^{(p,y)}_{(p,q)} = \im^{(\bar{x},y)}_{(p,q)}, \\
F_x^+ \cap G_y^- &= \im^{(x,q)}_{(p,q)} \cap \im^{(p,\bar{y})}_{(p,q)}
=  \im^{(x,\bar{y})}_{(p,q)} .
\end{split}
\end{equation}
Therefore,
\begin{equation} 
\begin{split}
\frac{F_x^+ \cap G_y^+}{(F_x^- \cap G_y^+) + (F_x^+ \cap G_y^-)}
&= \frac{\im^{(x,y)}_{(p,q)}}
{\im^{(\bar{x},y)}_{(p,q)} + \im^{(x,\bar{y})}_{(p,q)}} \\
&=  \frac{\im^{(x,y)}_{(x,y)}\big|_I}
{\im^{(\bar{x},y)}_{(x,y)}\big|_I + \im^{(x,\bar{y})}_{(x,y)}\big|_I} \,,
\end{split}
\label{E:complement1}
\end{equation}
where the last equality follows from the fact that $I = (p,q)
\subseteq (x,y)$. By Proposition \ref{P:xysheaf}(i) and \eqref{E:complement1},
$\sheaf{F}[x,y]|_I$ is a  complement of $(F_x^- \cap G_y^+) + (F_x^+ \cap G_y^-)$
in $F_x^+ \cap G_y^+$. Thus, Corollary \ref{C:directsum}(i) implies that
\begin{equation}\label{E:decompose}
\begin{split}
\sheaf{F}(p,q) = &\bigoplus_{-\infty < x \leq p<q\leq y < +\infty}\sheaf{F}[x,y]\big|_I
\bigoplus_{-\infty < x \leq p} F[x, +\infty]\big|_I \\
&\quad \ \ \ \bigoplus_{q \leq y < +\infty} F[-\infty, y]\big|_I \quad \ \ \,
\bigoplus F[-\infty, +\infty]\big|_I \,.
\end{split}
\end{equation}
Note that, up to this point in the proof, we only have used the DCC
on images, not the full tameness of $\sheaf{F}$. Before proceeding
to the next step recall that, according to Definition \ref{D:xysheaf}(i),
\begin{enumerate}[(a)]
\item $F[x, +\infty]\big|_I$ is a complement of 
$\im^{(\bar{x},+\infty)}_{(x,+\infty)}\big|_I$ in $\im^{(x,+\infty)}_{(x,+\infty)}\big|_I$;
\item $F[-\infty, y]\big|_I$ is a complement of $\im^{(-\infty,\bar{y})}_{(-\infty,y)}\big|_I$
in $\im^{(-\infty, y)}_{(-\infty,y)}\big|_I$;
\item $F[-\infty,+\infty]\big|_I = F(-\infty,+\infty)|_I$.
\end{enumerate}

\smallskip
\noindent
{\em Step 2.} Now we show that, in \eqref{E:decompose}, the summands
$F[x, +\infty]\big|_I$, $-\infty < x \leq p$, may be replaced with
$\bigoplus_{p< y\leq+\infty}\sheaf{F}[x,y\rangle\big|_I$. Consider  the families
of splittings
\begin{equation}
(F^-_x, F^+_x) = 
\big(\im^{(\bar{x},+\infty)}_{(p,+\infty)}\big|_I, \im^{(x,+\infty)}_{(p,+\infty)}\big|_I \big) ,
\end{equation}
$-\infty\leq x \leq p$, and
\begin{equation}
(G^-_y, G^+_y) = \big(\ker^{(p,+\infty)}_{(\bar{y},+\infty)}\big|_I, 
\ker^{(p,+\infty)}_{(y,+\infty)}\big|_I \big) ,
\end{equation} 
$p< y\leq +\infty$, that are disjoint and strongly cover $\sheaf{F}(p,+\infty)\big|_I$. 
Using the gluing property, one may verify that
\begin{equation}
\frac{F_x^+ \cap G_y^+}{(F_x^- \cap G_y^+) + (F_x^+ \cap G_y^-)}
= \frac{\ker^{(x,+\infty)}_{(y,+\infty)}\big|_I}
{\ker^{(\bar{x},+\infty)}_{(y,+\infty)}\big|_I + \ker^{(x,+\infty)}_{(\bar{y},+\infty)}\big|_I} \,.
\label{E:complement2}
\end{equation}
It follows from Proposition \ref{P:xysheaf}(ii) and \eqref{E:complement2} that
$\sheaf{F} [x,y \rangle\big|_I$ is a complement of $(F_x^- \cap G_y^+) + (F_x^+ \cap G_y^-)$
in $F_x^+ \cap G_y^+$. Corollary \ref{C:directsum}(ii), applied to $x=p$, gives
\begin{equation}
\im^{(x,+\infty)}_{(x,+\infty)}\big|_I = \im^{(\bar{x},+\infty)}_{(x,+\infty)}\big|_I
\bigoplus_{p< y\leq+\infty}\sheaf{F}[x,y\rangle\big|_I \,.
\label{E:sum}
\end{equation}
Thus, we may choose $F[x, +\infty]\big|_I$ to be
$\bigoplus_{p< y\leq+\infty}\sheaf{F}[x,y\rangle\big|_I$.

Similarly, we may choose $F[-\infty, y]\big|_I$ to be  
$\oplus_{p< y\leq+\infty}\sheaf{F}\langle x,y] \big|_I$. Therefore, we may rewrite 
\eqref{E:decompose} as
\begin{equation}
\begin{split}
\sheaf{F}(p,q) = &\bigoplus_{-\infty<x\leq p<q\leq y<+\infty} \sheaf{F}[x,y]\big|_I
\bigoplus_{-\infty<x\leq p<y\leq+\infty} \sheaf{F}[x,y\rangle\big|_I \\
&\ \, \bigoplus_{-\infty\leq x<q\leq y<+\infty} \sheaf{F}\langle x,y]\big|_I  \qquad \ \ \,
\bigoplus F[-\infty,+\infty]\big|_I.
\end{split}
\label{E:decompose1}
\end{equation}
\smallskip
\noindent
{\em Step 3.}
To complete the proof, we decompose the last summand in \eqref{E:decompose1}.
Consider the families of splittings
\begin{equation}
(F^-_{x}, F^+_{x}) = 
\big(\ker^{(-\infty,+\infty)}_{(-\infty,\bar{x})}\big|_I, \ker^{(-\infty,+\infty)}_{(-\infty,x)}\big|_I \big) ,
\end{equation}
$-\infty\leq x< +\infty$, and
\begin{equation}
(G^-_{y}, G^+_{y}) = \big(\ker^{(-\infty,+\infty)}_{(\bar{y},+\infty)}\big|_I, 
\ker^{(-\infty,+\infty)}_{(y,+\infty)}\big|_I \big) ,
\end{equation} 
$-\infty< y\leq +\infty$, that are disjoint and strongly cover
$\sheaf{F}(-\infty,+\infty)\big|_I$.  Arguing as in Step 2 and using the fact that
\begin{equation}
\frac{F_x^+ \cap G_y^+}{(F_x^- \cap G_y^+) + (F_x^+ \cap G_y^-)}
= \frac{\ker^{(-\infty,+\infty)}_{)x,y(}\big|_I}
{\ker^{(-\infty,+\infty)}_{)\bar{x},y(}\big|_I + \ker^{(-\infty,+\infty)}_{)x,\bar{y}(} \big|_I} \,,
\end{equation}
we obtain the decomposition
\begin{equation}
\sheaf{F}[-\infty,+\infty]\big|_I =
\bigoplus_{-\infty\leq x<y\leq+\infty}\sheaf{F}\langle x,y\rangle\big|_I \,.
\label{E:decompose2}
\end{equation}
Combining \eqref{E:decompose1} and \eqref{E:decompose2}, we obtain the
desired decomposition.
\end{proof}

\begin{remark} \label{R:decomposition} (Tameness Conditions)
\begin{enumerate}[(i)]
\item As pointed in the proof of Lemma \ref{L:sheafsum}, to obtain the decomposition
in \eqref{E:decompose}, we do not need the full tameness hypothesis on $\sheaf{F}$,
it suffices to assume that $\sheaf{F}$ satisfies the DCC on images. The DCC on kernels
only is needed in Steps 2 and 3 in the proof of the Decomposition Lemma.
\item Similarly, to obtain \eqref{E:decompose1}, in addition to the DCC on images 
for $F$, we only need the DCC on kernels for $\sheaf{F}\langle x,y]$ and
$\sheaf{F}[x,y\rangle$, not the tameness of $F$.
\end{enumerate}
\end{remark}

\begin{theorem}[Interval Decomposition of $p$-Sheaves] \label{T:sheafdec}
If $\sheaf{F}$ is a tame $p$-sheaf, then
\begin{equation}
\begin{split}
\sheaf{F} = &\bigoplus_{-\infty<x < y <+\infty} \sheaf{F}[x,y]
\bigoplus_{-\infty <x< y \leq +\infty} \sheaf{F}[x,y\rangle \\
&\bigoplus_{-\infty\leq x< y<+\infty} \sheaf{F}\langle x,y]  
\bigoplus_{-\infty\leq x<y\leq +\infty} \sheaf{F}\langle x,y\rangle \,.
\end{split}
\label{E:sheafdec1}
\end{equation}
Moreover,
\begin{equation}
\begin{split}
\sheaf{F} \cong &\bigoplus_{-\infty<x<y<+\infty} m[x,y]\, k[x,y]
\,  \bigoplus_{-\infty<x<y\leq+\infty} m[x,y\rangle \, k[x,y\rangle\\
&\bigoplus_{-\infty\leq x<y<+\infty} m\langle x,y] \, k\langle x,y]
\bigoplus_{-\infty\leq x<y\leq +\infty} m\langle x,y\rangle\, k\langle x,y\rangle,
\end{split}
\label{E:sheafdec2}
\end{equation}
where $m[x,y]$, $m[x,y\rangle$, $m\langle x,y]$, and $m\langle x,y\rangle$ denote the
multiplicities of the corresponding interval $p$-sheaves.
\end{theorem}
\begin{proof}
To prove \eqref{E:sheafdec1}, since all summands are subsheaves of $\sheaf{F}$, it suffices
to verify that the direct sum decomposition is satisfied by the space of sections of
$\sheaf{F}$ over each interval $I$. This is an immediate consequence of
Lemma \ref{L:sheafsum} because $\sheaf{F}[x,y] (I) = 0$ if $I \nsubseteq (x,y)$,
$\sheaf{F}\langle x,y] (I) = 0$ if $I \nsubseteq (-\infty, y)$, and
$\sheaf{F}[x,y\rangle (I) = 0$ if $I \nsubseteq (x, +\infty)$.

The isomorphism in \eqref{E:sheafdec2} follows from \eqref{E:sheafdec1}
and Proposition \ref{P:intervals}.
\end{proof}

\begin{remark}\label{R:udecomp}
Suppose that $\sheaf{F}$ is decomposable into interval $p$-sheaves, where
$\sheaf{F}$ is not necessarily tame. Then, $m[x,y]= \dim \sheaf{F}[x,y\rangle$.
Similarly, $m[x,y\rangle=\dim \sheaf{F}[x,y\rangle$, $m\langle x,y]=
\dim \sheaf{F}\langle x,y]$ and $m\langle x,y\rangle =\dim \sheaf{F}\langle x,y \rangle$.
Hence, for any interval decomposable $p$-sheaf, the decomposition is unique.
\end{remark}


\subsection{Decomposition of Virtually Tame $c$-Modules}
The main goal of this section is to prove an interval decomposition theorem for 
$c$-modules for which all correspondences $v_s^t$ are finite dimensional;
in particular, for pointwise  finite dimensional $c$-modules. The result is discussed in
the more general setting of virtually tame correspondence modules. We begin
with an analogue of Lemma \ref{L:sheafsum} for $c$-modules.

\begin{lemma} \label{L:cmodsum}
Let $\cmod{V}$ be a virtually tame $c$-module and $\sheaf{F}$ its
$p$-sheaf of sections. If $p, q \in \dec$, $p < q$, then the space of sections
$\sheaf{F} (p,q)$ satisfies
\[
\begin{split}
\sheaf{F}(p,q)|_A = &\bigoplus_{-\infty<x\leq p<q\leq y<+\infty} \sheaf{F}[x,y]\big|_A
\bigoplus_{-\infty<x\leq p<y\leq+\infty} \sheaf{F}[x,y\rangle\big|_A  \\
&\ \, \bigoplus_{-\infty\leq x<q\leq y<+\infty} \sheaf{F}\langle x,y]\big|_A \ \ \,
\bigoplus_{-\infty\leq x<y\leq +\infty} \sheaf{F}\langle x,y\rangle\big|_A \,,
\end{split}
\]
for any finite set $A \subseteq (p,q)$.
\end{lemma}
\begin{proof}
The proof is identical to that of Lemma \ref{L:sheafsum}. The only changes
needed are to restrict the relevant vector spaces of sections to $A$ and to
replace the use of Lemma \ref{L:scover1} with the Covering Lemma for Modules
(Lemma \ref{L:scover2}).
\end{proof}

\begin{theorem}[Interval Decomposition of $c$-Modules] \label{T:intdec}
If $\cmod{V}$ is a virtually tame $c$-module, then
\[
\begin{split}
\cmod{V} \cong &\bigoplus_{-\infty<x<y<+\infty} m[x,y]\,  \cmod{I}[x,y]
\, \bigoplus_{-\infty<x<y\leq+\infty} m[x,y\rangle \, \cmod{I}[x,y\rangle \\
&\bigoplus_{-\infty\leq x<y<+\infty} m\langle x,y] \,\cmod{I}\langle x,y]
\bigoplus_{-\infty\leq x<y\leq +\infty} m\langle x,y\rangle\, \cmod{I}\langle x,y\rangle ,
\end{split}
\]
where $m[x,y]$, $m\langle x,y]$, $m[x,y\rangle$, and $m\langle x,y\rangle$ denote the
multiplicities of the corresponding interval $c$-modules.
\end{theorem}

\begin{proof}
We first show that each vector space $V_t$, $t \in \real$, decomposes as claimed.
Letting $p = t^-$, $q = t^+$ and $A = \{t\}$, Lemma \ref{L:cmodsum} and
Proposition \ref{P:intervals} imply that
\begin{equation}
\begin{split}
V_t \cong &\bigoplus_{-\infty<x<y<+\infty} m[x,y]\,  \cmod{I}[x,y]_t
\, \bigoplus_{-\infty<x<y\leq+\infty} m[x,y\rangle \, \cmod{I}[x,y\rangle_t \\
&\bigoplus_{-\infty\leq x<y<+\infty} m\langle x,y] \,\cmod{I}\langle x,y]_t
\bigoplus_{-\infty\leq x<y\leq +\infty} m\langle x,y\rangle\, \cmod{I}\langle x,y\rangle_t .
\end{split}
\end{equation}
To verify the decomposition for correspondences, let $s,t \in \real$,
with $s < t$. Set $p = s^-$ and $q = t^+$, and $A = \{s,t\}$. Note that
$(p,q)$ is the closed interval $[s,t]$. Proposition \ref{P:extension} implies that
$v_s^t = F(p,q)|_A$. The desired decomposition for correspondences now follows
from Lemma \ref{L:cmodsum} and Proposition \ref{P:intervals}.
\end{proof}

\begin{remark}\label{R:unique-decomp-cmod}
Suppose that a $c$-module $\cmod{V}$ is virtually tame and let $\sheaf{F}$ be its
 $p$-sheaf of sections. As in Remark \ref{R:udecomp}, the multiplicity 
 $m[x,y]$ of the interval component $\cmod{I}[x,y]$ equals $\dim\sheaf{F}[x,y]$. 
 Similarly, $m[x,y\rangle=\dim \sheaf{F}[x,y\rangle$, $m\langle x,y]=
 \dim \sheaf{F}\langle x,y]$ and $m\langle x,y\rangle =\dim \sheaf{F}\langle x,y \rangle$.
Hence, the decomposition of a virtually tame $c$-module is unique. Moreover, if
$\sheaf{F}$ is interval decomposable (not necessarily tame), then the multiplicity
of each interval component obtained from the decomposition of $\cmod{V}$ is the
same as the multiplicity of the corresponding interval summand of $\sheaf{F}$.
\end{remark}


\section{The Isometry Theorem} \label{S:isometry}

In this section, we prove one of the main results of this paper, the stability
of persistence diagrams associated with interval decomposable $p$-sheaves.
We define the interleaving distance $d_I (\sheaf{F}, \sheaf{G})$ between any 
two $p$-sheaves $\sheaf{F}$ and $\sheaf{G}$ and, assuming that the sheaves
are decomposable, we also define the bottleneck distance 
$d_b(dgm (\sheaf{F}), dgm (\sheaf{G}))$ between their persistence diagrams. 
The Isometry Theorem states that
\begin{equation}
d_b(dgm (\sheaf{F}), dgm (\sheaf{G})) = d_I (\sheaf{F}, \sheaf{G}),
\end{equation}
for any decomposable $p$-sheaves $\sheaf{F}$ and $\sheaf{G}$. The
inequality
\begin{equation}
d_b(dgm (\sheaf{F}), dgm (\sheaf{G})) \geq d_I (\sheaf{F}, \sheaf{G})
\end{equation}
follows directly from the definition of $d_I$ and $d_b$. However, the proof of
the algebraic stability statement
\begin{equation}
d_b(dgm (\sheaf{F}), dgm (\sheaf{G})) \leq  d_I (\sheaf{F}, \sheaf{G})
\end{equation}
involves rather delicate arguments.

\subsection{Interleavings}
This section introduces the notions of interleaving and interleaving distance for
$p$-sheaves, extending the corresponding concepts for persistent modules
\cite{Chazal2009,Chazal2016} to our setting, as needed in the formulation of the 
Isometry Theorem.

Given a decorated number $p = t^\ast \in \dec$, $t \in \real$, and
$\epsilon \in \real$, let $p + \epsilon := (t+\epsilon)^\ast$, with the additional
convention that $-\infty + \epsilon = -\infty$ and $+\infty + \epsilon = +\infty$. 

\begin{definition}
(Dilations and Erosions)
\begin{enumerate}[(i)]
\item If $I = (p, q) \in \dec^2$ is an interval, the $\epsilon$-dilation of $I$,
$\epsilon \geq 0$, is defined as $I^\epsilon :=(p-\epsilon,q+\epsilon)$. 
\item The $\epsilon$-erosion of $I = (p, q)$, $\epsilon > 0$, is defined as
$I^{-\epsilon} :=(p+\epsilon,q-\epsilon)$, if $p+\epsilon < q-\epsilon$.
Otherwise, $I^{-\epsilon} = \emptyset$.
\item For $\epsilon > 0$ and $Z \subseteq \real$, if $Z =
\sqcup_{\lambda \in \Lambda} I_\lambda$ is its representation as the disjoint
union of its connected components, define
$Z^{-\epsilon} = \sqcup_{\lambda \in \Lambda} I_\lambda^{-\epsilon}$.
\end{enumerate}
\end{definition}

\begin{definition}
Let $\sheaf{F}$ and $\sheaf{G}$ be $p$-sheaves and $\epsilon \geq 0$.
\begin{enumerate}[(i)]
\item  An $\epsilon$-homomorphism $\Phi \colon \sheaf{F}\rightarrow\sheaf{G}$ is a
collection
\[
\{\phi^I_{I^{-\epsilon}} \colon \sheaf{F}(I) \to \sheaf{G}(I^{-\epsilon}) \colon I \in \cat{Int}\}
\]
of linear maps such that $\sheaf{G}^{J^{-\epsilon}}_{I^{-\epsilon}} \circ
\phi^J_{J^{-\epsilon}} = \phi^I_{I^{-\epsilon}} \circ \sheaf{F}^J_I$, for any  $I\subseteq J$,
with the convention that $\sheaf{G}(I^{-\epsilon}) = 0$ if $I^{-\epsilon} = \emptyset$. 
We refer to a $0$-homomorphism simply as a homomorphism.
\item The $\epsilon$-erosion of $\sheaf{F}$ is the $\epsilon$-homomorphism
$e^\epsilon_\sheaf{F} \colon \sheaf{F} \to \sheaf{F}$ given by
$\{\sheaf{F}^I_{I^{-\epsilon}} \colon I \in \cat{Int}\}$. (Note that $e^0_\sheaf{F}$ is
the identity.)
\end{enumerate}
\end{definition}

To motivate the definition of interleaving and explain how it generalizes the notion
of interleaving of persistence modules (cf.\,\cite{Chazal2009,Chazal2016}), let 
$\cmod{V}$ be a $p$-module and $\sheaf{F}$ be its sheaf of sections. Note that 
vectors in $V_s$, $s \in \real$, are in
one-to-one correspondence with sections of $\sheaf{F}$ over $[s, +\infty)$, where
$v_s \in V_s$  corresponds to the section $(v_t)_{t \geq s}$ given by $v_t = v_s^t (v_s)$.
Under this correspondence, the transition map $v_s^t$ becomes a $(t-s)$-erosion map.
The difference for more general $p$-sheaves is that we need to consider sections
over all intervals, not just those of the form $[s, +\infty)$.

\begin{definition}[Interleaving] \label{D:inter}
Let $\sheaf{F}$ and $\sheaf{G}$ be $p$-sheaves.
\begin{enumerate}[(i)]
\item An $\epsilon$-{\em interleaving} between $\sheaf{F}$ and $\sheaf{G}$,
$\epsilon \geq 0$, is a pair of $\epsilon$-homomorphisms $\Phi^\epsilon\colon \sheaf{F}
\to \sheaf{G}$ and $\Psi^\epsilon \colon \sheaf{G} \to \sheaf{F}$ such that
$\Psi^\epsilon\circ\Phi^\epsilon=e^{2\epsilon}_\sheaf{F}$ and 
$\Phi^\epsilon\circ\Psi^\epsilon=e^{2\epsilon}_\sheaf{G}$. (Note that a $0$-interleaving
is an isomorphism.) 
\item Given $\epsilon \geq 0$, $\sheaf{F}$ and $\sheaf{G}$ are $\epsilon^+$-{\em interleaved}
if they are $(\epsilon+\delta)$-interleaved for every $\delta>0$.
\item The {\em interleaving distance} between $\sheaf{F}$ and $\sheaf{G}$ is defined as
\[
\begin{split}
d_I(\sheaf{F},\sheaf{G}) &= \inf \{\epsilon>0 \colon \sheaf{F} \text{ and }\sheaf{G} \text{ are } 
\epsilon \text{-interleaved}\} \\
&= \min \{\epsilon \geq 0 \colon \sheaf{F} \text{ and }\sheaf{G} \text{ are } 
\epsilon^+ \text{-interleaved}\},
\end{split}
\]
with the convention that $d_I(\sheaf{F},\sheaf{G}) = \infty$ if no interleaving exists.
\end{enumerate}
\end{definition}

\begin{remark} \label{R:interleave}
The two forms of Definition \ref{D:inter}(iii) are equivalent because
if $\sheaf{F}$ and $\sheaf{G}$ are $\epsilon$-interleaved, then they are
$\epsilon^+$-interleaved. Note, however, that the converse is not necessarily true.
Moreover, $d_I$ is an extended pseudo-metric on the space
of isomorphism classes of $p$-sheaves.
\end{remark}

\begin{remark} \label{R:blind}
Given $s,t \in \ereal$, $s \leq t$, there are up to four intervals $(x,y) \in \dec^2$, defined
by $s$ and $t$, given by the different combinations of decorations for $s$ and $t$. More
precisely, let
\begin{equation} \label{E:blind}
A_{s,t} = \{(x,y) \in \{s^-, s^+\} \times \{t^-, t^+\} \colon x < y\},
\end{equation}
with the convention that $-\infty^\ast = -\infty$ and $+\infty^\ast = +\infty$. Then, the intervals
associated with $s$ and $t$ are $(x,y) \in A_{s,t}$. For a fixed interval-sheaf
type, say $k [\,,]$, any two $p$-sheaves in the collection
$\{k[x,y] \colon (x,y) \in A_{s,t}\}$ are $0^+$-interleaved.  The same applies to the
other three types: $k\langle x,y]$, $k[x,y\rangle$, and $k \langle x,y\rangle$.  Hence,
the interleaving distance is blind to changes in the decorations of the endpoints
of an interval. 
\end{remark}


\subsection{Persistence Diagrams}
To motivate the definition of persistence diagrams,  suppose that a $p$-sheaf $\sheaf{F}$
is interval decomposable; that is,
\begin{equation}
\begin{split}
\sheaf{F} \cong &\bigoplus_{-\infty<x<y<+\infty} m[x,y]\, k[x,y]
\, \bigoplus_{-\infty<x<y\leq+\infty} m[x,y\rangle \, k[x,y\rangle \\
&\bigoplus_{-\infty\leq x<y<+\infty} m\langle x,y] \, k\langle x,y]
\bigoplus_{-\infty\leq x<y\leq +\infty} m\langle x,y\rangle\, k\langle x,y\rangle,
\end{split}
\label{E:sdecomposition}
\end{equation}
where $m[x,y]$, $m\langle x,y]$, $m[x,y\rangle$, and $m\langle x,y\rangle$ denote
the multiplicities of the various interval summands. Then, $m[x,y]$ coincides with
the dimension of the space of sections $\sheaf{F}[x,y]$, defined in 
Section \ref{S:decomposition}, subsequently reinterpreted as a $p$-sheaf. 
Similarly, for $m\langle x,y]$, $m [x,y\rangle$, and $m\langle x,y \rangle$.
Hence, the decomposition of $\sheaf{F}$, if it exists, is uniquely determined by $m[x,y]$,
$m\langle x,y]$, $m[x,y\rangle$, and $m\langle x,y\rangle$. Although the construction 
of the spaces of sections $\sheaf{F}[x,y]$, $\sheaf{F}\langle x,y]$, $\sheaf{F}[x,y\rangle$, 
and $\sheaf{F}\langle x,y \rangle$ involve choices of complementary subspaces to certain
spaces of sections, the dimensions of these complements are independent
of the choices made. Therefore, the following definition of persistence diagrams is 
well posed.

\begin{definition}
Given any $p$-sheaf $\sheaf{F}$, define its {\em decorated persistence diagram} 
$\text{Dgm}(\sheaf{F})$  as the quadruple of functions $([m], \langle m], 
[m \rangle,\langle m\rangle)$ with domains
\begin{enumerate}[(i)]
\item $\dom ([m])=\{(x,y)\in\mathbb{E}^2:-\infty<x<y<+\infty\}$,
\item $\dom([m\rangle)=\{(x,y)\in\mathbb{E}^2:-\infty<x<y\leq +\infty\}$,
\item $\dom(\langle m])=\{(x,y)\in\mathbb{E}^2:-\infty\leq x<y<+\infty\}$,
\item $\dom(\langle m\rangle)=\{(x,y)\in\mathbb{E}^2:-\infty\leq x<y\leq +\infty\}$,
\end{enumerate}
given by
\begin{align*}
[m](x,y) &=\dim\sheaf{F}[x,y], &[m \rangle (x,y) &=\dim\sheaf{F} [x,y\rangle,  \\ 
\langle m](x,y) &= \dim \sheaf{F} \langle x, y ],
&\langle m\rangle(x,y)&=\dim\sheaf{F}\langle x,y\rangle.
\end{align*}
\end{definition}

As pointed out in Remark \ref{R:blind}, for a given interval $p$-sheaf type, the
interleaving distance is not sensitive to changes in the decoration of the endpoints of an
interval. Thus, to obtain an isometry theorem, we should not distinguish persistence
diagrams associated with those intervals. To this end, as in \cite{Chazal2016}, 
we introduce ``undecorated'' versions of persistence diagrams.

\begin{definition} \label{D:undec}
Given any $p$-sheaf $\sheaf{F}$, we define its {\em (undecorated) persistence diagram}
$dgm (\sheaf{F})$ as the quadruple of functions $([\overline{m}],[\overline{m}\rangle,
\langle\overline{m}],\langle\overline{m}\rangle)$ with domains
\begin{enumerate}[(i)]
\item $\dom([\overline{m}])=\{(s,t)\in \real \times \real \colon s< t\}$,
\item $\dom([\overline{m}\rangle)=\{(s,t)\in \real \times \ereal \colon s< t\}$,
\item $\dom(\langle\overline{m}])=\{(s,t)\in \ereal \times \real \colon s< t\}$,
\item $\dom(\langle\overline{m}\rangle)=\{(s,t)\in \ereal \times \ereal \colon s \leq t\}$,
\end{enumerate}
given by $[\overline{m}](s,t)=\sum_{(x,y)\in A_{s,t}}[m](x,y)$, with $A_{s,t}$ as in
\eqref{E:blind}. The function values $[\overline{m}\rangle(s,t)$, $\langle\overline{m}](s,t)$,
and $\langle\overline{m}\rangle(s,t)$ are defined similarly.
\end{definition}

\begin{remark}
We often refer to the domain of each of the four functions in a $p$-diagram as a multiset
and to the value of the function at $(s,t)$ as the multiplicity of $(s,t)$. We also frequently
treat the multiplicity functions as defined on $\ereal \times \ereal$ by extending them to be zero
outside their original domains. 
\end{remark}

Notice that for singletons $(s^-, s^+) \in \dec^2$, $s \in \real$,  forgetting decorations returns 
points on the diagonal
$\Delta=\{(s,s) \colon s \in \real\}$. Nonetheless, in Definition \ref{D:undec}(i)-(iii), $\Delta$
is not included in the domain of the multiplicity function because $k[s^-,s^+]$, $k\langle s^-,s^+]$
and $k[s^-,s^+\rangle$ are all $0^+$-interleaved with the trivial $p$-sheaf. This is not the case,
however, for $k\langle s^-,s^+\rangle$. Indeed, the $\delta$-erosion map for global sections
of this $p$-sheaf is the identity, for any $\delta>0$, implying that $k\langle s^-,s^+\rangle$
is not $\delta$-interleaved with the trivial $p$-sheaf, for any $\delta>0$.

Henceforth, we refer to undecorated persistence diagrams simply as persistence diagrams,
or $p$-diagrams. In order to define the bottleneck distance between $p$-diagrams, we
first discuss a variant of the notion of matching of multisets that is suited to our goals.
Abusing terminology, we often refer to a multiset $(A,f)$, with multiplicity function $f$,
simply as $A$.

A partial matching between the multisets $(A, f)$ and $(B, g)$ is a multiset injection
from a subset of $A$ to $B$. More formally, a {\em resolution} of $A$ is a mapping
$\pi_A \colon \hat{A} \to A$ such that  $|\pi_A^{-1} (a)| = f(a)$, $\forall a \in A$.
A {\em partial matching} between $A$ and $B$ is an injection
$\sigma \colon \bar{A} \to \hat{B}$, where $\bar{A} \subseteq \hat{A}$. The
partial matching $\sigma$ is {\em surjective} if $\im(\sigma) = \hat{B}$. We
abuse terminology and refer to elements of $\hat{A}$ as elements of the multiset $A$,
and to $\sigma (a)$ as an element of the multiset $B$.

As usual, the bottleneck distance will be based on the $\ell_\infty$-distance on the
extended plane, denoted $d_\infty \colon \ereal \times \ereal \to \ereal$ and given by 
\begin{equation}
d_\infty \left((s_1,t_1),(s_2,t_2)\right) = \max\{|s_1-s_2|,|t_1-t_2|\},
\end{equation}
with the convention that $|\infty-\infty|=0$. Note that 
$d_\infty\big((s,t),\Delta\big)=|s-t|/2$.

\begin{definition}
Let $A$ and $B$ be multisets with domain $\ereal \times \ereal$, $\epsilon > 0$, and $L>0$. A
partial matching $\sigma$ between $A$ and $B$ is an $(\epsilon,L)$-{\em matching} 
provided that:
\begin{enumerate}[(i)]
\item if $\sigma(a)=b$, then $d_\infty(a,b)\leq \epsilon$;
\item if $a \in \supp (A) \setminus \dom (\sigma)$, then $d_\infty(a,\Delta)\leq L\epsilon$;
\item if $b \in \supp(B) \setminus \im(\sigma)$, then $d_\infty(b,\Delta)\leq L\epsilon$.
\end{enumerate}
A partial matching $\sigma$ is a {\em full $\epsilon$-matching} if it is a multiset bijection
between $A$ and $B$ that satisfies condition (i) above.
\end{definition}

\begin{definition}
The persistence diagrams $dgm_i=(f_1^i,f_2^i,f_3^i,f_4^i)$, $i \in \{ 1,2\}$, are
{\em $\epsilon$-matched} if the following holds:
\begin{enumerate}[(i)]
\item $f_1^1$ and $f_1^2$ are $(\epsilon,2)$-matched;
\item $f_2^1$ and $f_2^2$ are $(\epsilon,1)$-matched;
\item $f_3^1$ and $f_3^2$ are $(\epsilon,1)$-matched;
\item $f_4^1$ and $f_4^2$ are fully $\epsilon$-matched.
\end{enumerate}
\end{definition}

\begin{definition}
The {\em bottleneck distance} between persistence diagrams is the
extended pseudo-metric defined as
\[
d_b(dgm_1, dgm_2) = \inf \{\epsilon>0 \colon dgm_1 \text{ and } dgm_2
\text{ are $\epsilon$-matched} \}.
\]
\end{definition}

The next proposition shows that a converse to the stability of persistence 
diagrams holds.

\begin{proposition} \label{P:cstability}
Let $\sheaf{F}$ and $\sheaf{G}$ be decomposable $p$-sheaves. If $dgm(\sheaf{F})$
and $dgm(\sheaf{G})$ are $\epsilon$-matched, $\epsilon > 0$, then $\sheaf{F}$ 
and $\sheaf{G}$ are $\epsilon^+$-interleaved. Thus,
\[
d_I(\sheaf{F},\sheaf{G})\leq d_b(dgm(\sheaf{F}),dgm(\sheaf{G})).
\]
\end{proposition}

\begin{proof}
Let $\epsilon \geq 0$. If $dgm(\sheaf{F})=(f_1, f_2, f_3, f_4)$ and $dgm(\sheaf{G}) =
(g_1, g_2, g_3, g_4)$ are $\epsilon$-matched, there exist:
\begin{enumerate}[(a)]
\item an $(\epsilon,2)$-matching between $f_1$ and $g_1$;
\item an $(\epsilon,1)$-matching between $f_2$ and $g_2$;
\item an $(\epsilon,1)$-matching between $f_3$ and $g_3$;
\item a full $\epsilon$-matching between $f_4$ and $g_4$.
\end{enumerate}

We show that the subsheaves $[\sheaf{F}]$ and $[\sheaf{G}]$ are
$\epsilon^+$-interleaved. The argument for the components of type $\langle \ ]$,
$[\  \rangle$, and $\langle\ \rangle$ of $\sheaf{F}$ and $\sheaf{G}$ are
similar. As the diagrams disregard decorations of the endpoints of an interval,
suppose that for $s_i,t_i \in \real$, $s_i < t_i$, $i=1,2$, the interval modules
$k[s_1^\ast, t_1^\ast]$ and $k[s_2^\ast, t_2^\ast]$ are $(\epsilon, 2)$-matched
under the given $\epsilon$-matching of diagrams. Then, $d^\infty((s_1,t_1),(s_2,t_2))
\leq\epsilon$,  which implies that $k[s_1^\ast,t_1^\ast]$ and $k[s_2^\ast,t_2^\ast]$ are 
$(\epsilon+\delta)$-interleaved for any $\delta>0$. Assembling all of these pairwise
interleavings, yields the desired $(\epsilon + \delta)$-interleaving.
\end{proof}


\subsection{Algebraic Stability}

This section is devoted to the proof of the algebraic stability of persistence diagrams;
that is, the statement that
\[
d_b(dgm(\sheaf{F}), dgm(\sheaf{G})) \leq 
d_I (\sheaf{F}, \sheaf{G}).
\]
The general strategy for the proof of stability, particularly in Lemma \ref{L:fmatching} and
Theorem \ref{T:stab} below, is similar to that used by Bjerkevik \cite{Bjerkevik2016stability} 
to study algebraic stability in the context of multi-parameter persistence modules.
The arguments needed for $p$-sheaves, however, are quite distinct. 

\begin{definition}
Let $\sheaf{F}$ be a $p$-sheaf. A set $B_{\sheaf{F}}$ of non-trivial sections of
$\sheaf{F}$ is a {\em basis} of $\sheaf{F}$ if for any interval $I \subseteq \real$, 
the set
\[
B_\sheaf{F} (I) := \{f|_I \colon f \in B_\sheaf{F}, I \subseteq \dom (f) \text{ and }
f|_I \not\equiv 0\}
\]
is a basis of $\sheaf{F} (I)$.
\end{definition}
\begin{lemma}\label{L:csupp}
If $B_{\sheaf{F}}$ is a basis of a $p$-sheaf $\sheaf{F}$, then each section in
$B_{\sheaf{F}}$ has connected support. Moreover, if $f \in B_{\sheaf{F}}$, then
the subsheaf $\sheaf{F}\langle f \rangle$ spanned by $f$ (see Definition \ref{D:span})
is isomorphic to an interval $p$-sheaf.
\end{lemma}

\begin{proof}
Suppose that $f\in B_{\sheaf{F}}$ has disconnected support. Then, there exist
$r<s<t$ such that $f|_r\neq 0$, $f|_s=0$ and $f|_t\neq 0$. By the gluing property,
$f$ may be decomposed as $f=a+b$, where $a_x=f_x$, for
$x\in L:=\dom(f) \cap (-\infty, s]$, and $a_x=0$, for $x\in R:=\dom(f) \cap [s, +\infty)$. 
This implies that $b_x = f_x$, $\forall x \in R$ and $b|_L \equiv 0$. Note that both
$a$ and $b$ are non-trivial sections. By the definition of basis, $a$ and $b$ can be
uniquely expressed as linear combinations
\begin{equation}
a=c_0f+\sum_{f_\lambda\in B_{\sheaf{F}(I)}\setminus \{f\}}c_\lambda f_\lambda
\quad \text{and} \quad
b=d_0f+\sum_{f_\lambda\in B_{\sheaf{F}}\setminus \{f\}}d_\lambda f_\lambda.
\end{equation}
Since $f = a+b$, it follows that $c_0 + d_0 =1$. On the other hand,
\begin{equation}
a|_L =c_0f|_L + \sum c_\lambda f_\lambda|_L \quad \text{and} \quad
b|_R = d_0 f|_R + \sum d_\lambda f_\lambda|_R.
\end{equation}
Since both $a|_L$ and $b|_R$ are non-trivial and coincide with $f|_L$ and
$f|_R$, respectively, it follows that $c_0 = d_0 = 1$. This contradicts the fact that
$c_0 + d_0 =1$.

The fact that, for any $f \in B_{\sheaf{F}}$, the subsheaf $\sheaf{F}\langle f \rangle$
is isomorphic to an interval sheaf now follows from Proposition \ref{P:isheaf}.
\end{proof}

Let $\sheaf{F}$ be interval decomposable. If $\Phi \colon \sum_{\lambda \in 
\Lambda} \sheaf{F}_\lambda \to \sheaf{F}$ is a $p$-sheaf isomorphism,
where each $\sheaf{F}_\lambda$ is an interval $p$-sheaf, then the images under
$\Phi$ of the unit sections $s_\lambda$ of $\sheaf{F}_\lambda$
(see Definition \ref{D:intsheaf}) form a basis of $\sheaf{F}$. The following 
proposition provides a converse statement. Hence, we may view an interval
decomposition of a $p$-sheaf as a choice of basis.

\begin{proposition}
If $B_\sheaf{F}$ is a basis of $\sheaf{F}$, then there are interval modules
$F_\lambda$, $\lambda \in \Lambda$, and an isomorphism 
$\Phi \colon \sum_{\lambda \in  \Lambda} \sheaf{F}_\lambda \to \sheaf{F}$
that maps the unit sections $s_\lambda$ of $\sheaf{F}_\lambda$
bijectively onto $B_\sheaf{F}$.
\end{proposition}

\begin{proof}
Set $\sheaf{F}_\lambda :=  \sheaf{F}\langle f_\lambda \rangle$.
Lemma \ref{L:csupp} implies that each $F_\lambda$ is isomorphic to
an interval $p$-sheaf. Moreover, the isomorphism may be chosen to map 
the unit section $s_\lambda$ of $F_\lambda$ to $f_\lambda$. The
fact that $B_\sheaf{F}$ is a basis implies that these induce an isomorphism
$\Phi \colon \sum_{\lambda \in  \Lambda} \sheaf{F}_\lambda \to \sheaf{F}$ with
the desired properties.
\end{proof}

Let $f$ and $g$ be sections of $\sheaf{F}$ and $\sheaf{G}$, respectively,
with connected support. By Proposition \ref{P:isheaf}, $\sheaf{F}\langle f\rangle$
and $\sheaf{G}\langle g\rangle$ are (isomorphic to) interval $p$-sheaves.

\begin{definition} \label{D:smatch}
Under the above assumptions, define $f$ and $g$ to be $\epsilon$-{\em matched} if:
\begin{enumerate}[(i)]
\item $\sheaf{F}\langle f\rangle$ and $\sheaf{G}\langle g\rangle$ are interval
$p$-sheaves of the same type;
\item $d_H(\supp(f),\supp(g))\leq \epsilon$, where $d_H$ denotes (extended)
Hausdorff distance. 
\end{enumerate} 
\end{definition}

\begin{lemma}
Under the assumptions of Definition \ref{D:smatch}, if the sections $f$ and $g$
are $\epsilon$-matched, then $\sheaf{F}\langle f\rangle$
and $\sheaf{G}\langle g \rangle$ are $\epsilon$-interleaved.
\end{lemma}
\begin{proof}
Define $\epsilon$-homomorphisms $\Phi^\epsilon \colon \sheaf{F}\langle f\rangle
\to \sheaf{G}\langle g\rangle $ and $\Psi^\epsilon \colon \sheaf{G}\langle g\rangle
\to \sheaf{F}\langle f\rangle$ by $\Phi^\epsilon(f)=g|_{\dom(f)^{-\epsilon}}$ and
$\Psi^\epsilon(g)=f|_{\dom(g)^{-\epsilon}}$. This completes characterize these
homomorphisms because the $p$-sheaves are spanned by $f$ and $g$. Then,
$(\Phi^\epsilon, \Psi^\epsilon)$  is an $\epsilon$-interleaving.
\end{proof}
	
\begin{definition} \label{D:significant}	
Given $\delta>0$, a section is said to be $\delta$-{\em trivial} if it gets mapped
to the zero section under the $\delta$-erosion map. Otherwise, the section is called
$\delta$-{\em significant}. We denote by $X_\sheaf{F}^\epsilon \subseteq B_{\sheaf{F}}$
the {\em set of all $2\epsilon$-significant sections in $B_{\sheaf{F}}$}.
\end{definition}

\begin{definition}
($\epsilon$-Mappings and $\epsilon$-Matchings)
\begin{enumerate}[(i)]
\item A mapping $\sigma \colon \dom(\sigma) \subseteq B_\sheaf{F} 
\to B_\sheaf{G}$ is an $\epsilon$-{\em mapping} if $f$ and $\sigma(f)$ are 
$\epsilon$-matched, $\forall f\in\dom(\sigma)$.
\item $B_\sheaf{F}$ and $B_\sheaf{G}$ are said to be 
$\epsilon$-{\em matched} if there exists an injective $\epsilon$-map
$\sigma \colon \dom(\sigma) \subseteq B_\sheaf{F} \to B_\sheaf{G}$
such that $X^\epsilon_\sheaf{F} \subseteq \dom(\sigma)$ and
$X_\sheaf{G}^\epsilon \subseteq\im(\sigma)$.
\end{enumerate}
\end{definition}

\begin{proposition}\label{P:basis_diagram}
If $B_\sheaf{F}$ and $B_\sheaf{G}$ are $\epsilon$-matched, then
$dgm(\sheaf{F})$ and $dgm(\sheaf{G})$ are $\epsilon$-matched.
\end{proposition}
\begin{proof}
This follows directly from the definitions of $\epsilon$-matchings for
sections and for persistence diagrams.
\end{proof}

By Proposition \ref{P:basis_diagram}, we can translate a diagram-matching
problem into a question of matching bases and this is how we approach
the proof of the stability theorem. 

\begin{lemma} \label{L:injectives}
Let $\epsilon>0$. If there exist injective $\epsilon$-mappings
$\sigma \colon X^\epsilon_\sheaf{F} \to B_\sheaf{G}$
and $\tau \colon X^\epsilon_\sheaf{G} \to B_\sheaf{F}$, then there is an 
$\epsilon$-matching between $B_\sheaf{F}$ and $B_\sheaf{G}$.
\end{lemma}
\begin{proof}
Let $G$ be the bipartite graph with vertex set partitioned as  $B_{\sheaf{F}} 
\sqcup B_{\sheaf{G}}$ and whose edge set is the union of the graphs
of the maps $\sigma$ and $\tau$. We denote an edge $\{u,v\}$, where $\sigma (u) =v$,
by $u \xrightarrow{\sigma} v$. Similarly, $u \xrightarrow{\tau} v$ if $v = \tau (u)$. Since
$\sigma$ and $\tau$ are injections, each vertex of $G$ has degree $\leq 2$.

We construct an $\epsilon$-matching over each connected component of
$G$ and take the union of these to obtain the desired matching. Let $C$ be any
connected component of $G$ and $V_C$ its vertex set. Let
$C^\epsilon_\sheaf{F} := X^\epsilon_\sheaf{F} \cap V_C$ and 
$C^\epsilon_\sheaf{G} := X^\epsilon_\sheaf{G} \cap V_C$, the sets of
vertices in $C$ comprised of $\epsilon$-significant sections of $\sheaf{F}$ and
$\sheaf{G}$, respectively. Then, at least one of the following properties is satisfied:
(a) $C^\epsilon_\sheaf{G}  \subseteq \im\,\sigma|_{C^\epsilon_\sheaf{F}}$ or
(b) $C^\epsilon_\sheaf{F}  \subseteq \im\,\tau|_{C^\epsilon_\sheaf{G}}$.
Indeed, there is a sequence
\begin{equation}
\ldots \xrightarrow{\tau} v_{-1} \xrightarrow{\sigma} v_0 \xrightarrow{\tau}
v_1\xrightarrow{\sigma} \ldots
\end{equation}
such that $V_C = \cup_{i} \{v_i\}$. This sequence may be infinite or finite on
either end. Note that if it is finite on the right, then the last element is a section that
is $2\epsilon$-trivial. If the sequence is infinite on the left or finite starting
with an edge of type $\xrightarrow{\sigma}$, then the construction implies that 
$C^\epsilon_\sheaf{G}  \subseteq \im\,\sigma|_{C^\epsilon_\sheaf{F}}$,
regardless of the behavior of the sequence on the right end.
Similarly, if it is finite on the left, starting with a type $\xrightarrow{\tau}$
edge, then $C^\epsilon_\sheaf{F} \subseteq \im\,\tau|_{C^\epsilon_\sheaf{G}}$.

If property (a) above is satisfied, then $\sigma|_{C^\epsilon_\sheaf{F}} \colon
C^\epsilon_\sheaf{F} \to B_\sheaf{G} \cap C$ is an $\epsilon$-matching between
$B_\sheaf{F} \cap V_C$ and $B_\sheaf{G} \cap V_C$. Let
$\zeta:= \tau|_{C^\epsilon_\sheaf{G}} \colon C^\epsilon_\sheaf{G} \to 
B_\sheaf{F} \cap V_C$. If property (b) is satisfied, then $\tau^{-1}|_{\im (\zeta)}
\colon \im (\zeta) \supseteq C^\epsilon_\sheaf{F} \to C^\epsilon_\sheaf{G}$ is 
an $\epsilon$-matching
between $B_\sheaf{F} \cap V_C$ and $B_\sheaf{G} \cap V_C$. Assembling
the matchings over all connected components of $G$ yields an
$\epsilon$-matching between $B_\sheaf{F}$ and $B_\sheaf{G}$.
\end{proof}

We now proceed to the core argument in the proof of the algebraic stability of
persistence diagrams, namely, the construction of the injective 
$\epsilon$-mappings called for in the hypothesis of Lemma \ref{L:injectives}.
In the remainder of this section, we index the bases of
$\sheaf{F}$ and $\sheaf{G}$ as
$B_\sheaf{F} = \{f_\lambda \colon \lambda \in \Lambda_{\sheaf{F}}\}$ and
$B_\sheaf{G} = \{g_\mu \colon \mu \in  \Lambda_{\sheaf{G}}\}$, respectively. 
We also adopt the following abbreviations:
\begin{align}
D_\lambda &:=  \dom (f_\lambda),  &D_\mu &:= \dom (g_\mu), \\
Z_\lambda &:= D_\lambda \setminus \supp (f_\lambda),
&Z_\mu &:= D_\mu \setminus \supp (g_\mu)\,.
\end{align}
For $\epsilon \geq 0$, $\Phi^\epsilon:\sheaf{F} \to \sheaf{G}$ 
and $\Psi^\epsilon:\sheaf{G} \to \sheaf{F}$ denote $\epsilon$-homomorphisms. Using
$B_\sheaf{F}$ and $B_\sheaf{G}$,  an $\epsilon$-morphism $\Phi^\epsilon$ may be
represented by a ``matrix'' $\left(\phi^\mu_\lambda\right)$, where
\begin{equation}
\Phi^\epsilon(f_{\lambda}) = \sum_{\mu \in \Lambda_G} \phi^\mu_\lambda \,
g_{\mu}|_{D_\lambda^{-\epsilon}} \,,
\end{equation}
for any $\lambda \in \Lambda_F$. Here, we make the convention that
$\phi^\mu_\lambda=0$ if $g_\mu|_{D_\lambda^{-\epsilon}}=0$; that is,
$D_\lambda^{-\epsilon} \subseteq Z_\mu$. Thus, for a fixed $\lambda$, only
finitely many entries $\phi^\mu_\lambda$ can be non-zero. 
Similarly, $\Psi^\epsilon$ may be represented by a matrix
$\left(\psi^\lambda_\mu\right)$ whose entries, for a fixed $\mu$, vanish for
all but finitely many values of $\lambda$.

\begin{lemma} \label{L:match1}
If $\phi^\mu_\lambda \neq 0$, then $D_\lambda^{-\epsilon} \subseteq
D_\mu$ and $Z_\lambda^{-\epsilon}\subseteq Z_\mu$. Moreover, the interval
$p$-sheaves $F \langle f_\lambda \rangle$ and $G\langle g_\mu \rangle$ 
are of the same type.

\end{lemma}
\begin{proof}
$D_\lambda^{-\epsilon} \subseteq D_\mu$ follows directly from the definition of 
$\phi^\mu_\lambda$. Suppose that $Z_\lambda \neq \emptyset$ and let
$I \subseteq \real$ be a connected component of $Z_\lambda$. Then,
\begin{equation}
0 = \Phi^{\epsilon}(f_{\lambda}|_I) = \sum_{\mu \in \Lambda_G} 
\phi^\mu_\lambda \, g_{\mu}|_{I^{-\epsilon}} \,.
\end{equation}
Since non-trivial sections in $\{g_{\lambda}|_{I^{-\epsilon}}\}$ are independent, 
we have $\phi_{\lambda}^{\mu}g_{\mu}|_{I^{-\epsilon}}=0$ for any $\mu$. Hence,
$\phi^{\mu}_{\lambda}\neq 0$ implies that $g_\mu|_{I^{-\epsilon}}=0$.
Thus, $Z_\lambda^{-\epsilon}\subseteq Z_\mu$.

Lemma \ref{L:csupp} shows that both $F \langle f_\lambda \rangle$ and 
$G\langle g_\mu \rangle$  are isomorphic to interval $p$-sheaves. The 
fact that these interval $p$-sheaves are of the same type follows  directly from
$D_\lambda^{-\epsilon} \subseteq D_\mu$ and 
$Z_\lambda^{-\epsilon}\subseteq Z_\mu$.
\end{proof}

To state the next lemma, we introduce some terminology.
Let $\alpha \colon \dec \to \ereal \times\mathbb{N}$ be given by $\alpha(t^+)=(t,1)$
and $\alpha(t^-)=(t,-1)$, $\forall t \in \real$, $\alpha (+\infty) = (+\infty, 1)$, and 
$\alpha (-\infty) = (-\infty, -1)$.  Define ``$+$'' and ``$-$'' operations in
$\ereal \times\mathbb{N}$ that perform coordinate-wise addition or subtraction. 
If we order the elements of $\ereal \times\mathbb{N}$ by $(a,n) < (b,m)$ if 
$a < b$, or if $a=b$ and $n < m$,  then $\alpha$ is order
preserving.

\begin{lemma} \label{L:match2}
Let $\lambda, \lambda' \in \Lambda_\sheaf{F}$ and suppose that
$\phi^{\mu}_{\lambda} \psi^{\lambda'}_{\mu} \ne 0$ for some 
$\mu \in \Lambda_\sheaf{G}$. If any one of the conditions
\begin{enumerate}[\rm (i)]
\item $\sheaf{F}\langle f_\lambda \rangle\simeq k[p,q]$,  
$\sheaf{F}\langle f_{\lambda'} \rangle 
\simeq k[p',q']$ and $\alpha(q)-\alpha(p)\geq \alpha(q')-\alpha(p')$,
\item $\sheaf{F}\langle f\lambda  \rangle\simeq k[p,q\rangle$, 
$\sheaf{F}\langle {f\lambda'}\rangle 
\simeq k[p',q'\rangle$ and $\alpha(p)+\alpha(q)\leq \alpha(p')+\alpha(q')$,
\item $\sheaf{F}\langle f_\lambda \rangle\simeq k\langle p,q]$, 
$\sheaf{F}\langle f_{\lambda'}  \rangle\simeq k\langle p',q']$ and 
$\alpha(p)+\alpha(q)\geq \alpha(p')+\alpha(q')$,
\item $\sheaf{F}\langle f_\lambda \rangle\simeq k\langle p,q\rangle$, 
$\sheaf{F}\langle f_{\lambda'}\rangle 
\simeq k\langle p',q'\rangle$ and $\alpha(q)-\alpha(p)\leq \alpha(q')-\alpha(p')$.
\end{enumerate}
is satisfied, then $g_{\mu}$ is $\epsilon$-matched
with either $f_{\lambda}$ or $f_{\lambda'}$.
\end{lemma}

\begin{proof}
By Lemma \ref{L:match1}, the interval $p$-sheaf $G\langle g_\mu \rangle$ 
has the same type as the intervals $p$-sheaves $F \langle f_\lambda \rangle$
and $F \langle f_\lambda' \rangle$ and 
\begin{enumerate}[(1)]
\item $D_\lambda^{-\epsilon} \subseteq D_\mu$ and $Z_\lambda^{-\epsilon}
\subseteq Z_\mu$;
\item $D_\mu^{-\epsilon} \subseteq D_{\lambda'}$ and $Z_\mu^{-\epsilon} 
\subseteq Z_{\lambda'}$.
\end{enumerate}
Thus, to prove the lemma, it suffices to show that the assumptions that
\begin{enumerate}[(3)]
\item[(3)] $d_H(\supp(f_\lambda),\supp(g_\mu))>\epsilon$ and
\item[(4)] $d_H(\supp(f_{\lambda'}),\supp(g_\mu))>\epsilon$
\end{enumerate}
lead to a contradiction. We divide the argument into four cases.

{\em Case 1.} Suppose that (i) is satisfied, so that  we may write
$\sheaf{G}\langle g_{\mu}\rangle\simeq k[x,y]$. By (1), we have $
x\leq p+\epsilon$ and $y\geq q-\epsilon$, whereas (3) implies 
that $x<p-\epsilon$ or $y>q+\epsilon$. If $x<p-\epsilon$, the inequality
$y\geq q-\epsilon$ ensures that
\begin{equation}
\alpha(y)-\alpha(x)>\alpha(q)-\alpha(p) \,.
\label{E:alpha1}
\end{equation}
If $y>q+\epsilon$, the inequality $x<p+\epsilon$ also implies \eqref{E:alpha1}.
Similarly, (2) and (4) yield
\begin{equation}
\alpha(q')-\alpha(p')>\alpha(y)-\alpha(x) \,.
\label{E:alpha2}
\end{equation}
Then, \eqref{E:alpha1} and \eqref{E:alpha2} imply that
$\alpha(q)-\alpha(p)<\alpha(q')-\alpha(p')$, which contradicts
the hypothesis.
	
{\em Case 2.} Suppose that (ii) is satisfied and write $g_{\mu}\simeq k[x,y\rangle$. 
By (1), we have $x\leq p+\epsilon$ and $y\leq q+\epsilon$, whereas (3)
implies that $x<p-\epsilon$ or $y<q-\epsilon$. Either possibility yields
$\alpha(x)+\alpha(y)<\alpha(p)+\alpha(q)$. Similarly, by (2) and (4), we have 
$\alpha(p')+\alpha(q')<\alpha(x)+\alpha(y)$. Then,
$\alpha(p')+\alpha(q')<\alpha(p)+\alpha(q)$, which is a contradiction.
	
{\em Case 3.} Suppose that (iii) is satisfied and let  $g_{\mu}\simeq k\langle x,y]$. 
By (1) we have $x\geq p-\epsilon$ and $y\geq q-\epsilon$, whereas (3)
implies that $x>p+\epsilon$ or $y>q+\epsilon$, either giving 
$\alpha(x)+\alpha(y)>\alpha(p)+\alpha(q)$. Similarly, by (2) and (4), we 
have $\alpha(p')+\alpha(q')>\alpha(x)+\alpha(y)$. Then, $\alpha(p')+
\alpha(q')>\alpha(p)+\alpha(q)$, which is a contradiction.
	
{\em Case 4.} Suppose that (iv) holds and write $g_{\mu}\simeq k\langle x,y\rangle$. 
By (1) we have $x\geq p-\epsilon$ and $y\leq q+\epsilon$, whereas (3) 
implies $x>p+\epsilon$ or $y<q-\epsilon$, either yielding 
$\alpha(y)-\alpha(x)<\alpha(q)-\alpha(p)$. Similarly, by (2) and (4), we 
have $\alpha(q')-\alpha(p')<\alpha(y)-\alpha(x)$. Hence, 
$\alpha(q')-\alpha(p')<\alpha(q)-\alpha(p)$, contradicting the hypothesis.
\end{proof}

\begin{definition}
Given $f \in B_\sheaf{F}$, define $m^\epsilon_\sheaf{G}(f) \subseteq 
B_\sheaf{G}$ as the the subset of all sections $g \in B_\sheaf{G}$ such
that $f$ and $g$ are $\epsilon$-matched.
\end{definition}

Let $A=\{f_{\lambda_1},\ldots,f_{\lambda_m}\}\subseteq B_{\sheaf{F}}$ be a
finite, indexed sub-collection of basis elements with the property that all interval 
subsheaves $\sheaf{F} \langle f_{\lambda_i} \rangle$, $1 \leq i \leq m$, are 
of the same type and the inequalities specified in the hypotheses of
Lemma \ref{L:match2} are satisfied for any $i \leq j$. More precisely, if 
$\sheaf{F}\langle f_{\lambda_i} \rangle\simeq k[p_i,q_i]$ and
$\sheaf{F}\langle f_{\lambda_j} \rangle  \simeq k[p_j,q_j]$, 
$i \leq j$, then $\alpha(q_i)-\alpha(p_i)\geq 
\alpha(q_j)-\alpha(p_j)$, and similarly for other interval sheaf types. 
Define
\begin{equation}
\nu(A):=\{g_\mu\in B_{\sheaf{G}}: \phi_{\lambda_i}^\mu 
\psi_\mu^{\lambda_j}\neq 0 \mbox{ for some pair }i\leq j\} .
\end{equation}
Note that $|\nu (A)| < \infty$ because, for each fixed pair $i \leq j$,
only finitely many entries $\phi_{\lambda_i}^\mu$ 
and $\phi^{\lambda_j}_\mu$ can be non-zero. Moreover, 
Lemma \ref{L:match2} implies that each section in $\nu(A)$ is 
$\epsilon$-matched  with at least one section in $A$. Thus, we 
have
\begin{equation}
\nu(A) \subseteq \cup_{f\in A}\, m^\epsilon_\sheaf{G}(f) \,.
\label{E:mgf}
\end{equation}
We remark that the special ordering of the elements of $A$ is essential in
the argument that shows that \eqref{E:mgf} holds.

\begin{proposition}\label{P:cardinality}
Let $A$ be as above. If $(\Phi^\epsilon, \Psi^\epsilon)$ is an
$\epsilon$-interleaving and $A \subseteq X^\epsilon_\sheaf{F}$ (that is,
all sections in $A$ are $2\epsilon$-significant),  then
\[
|A|\leq |\nu(A)|<\infty.
\]
\end{proposition}

\begin{proof}
Let $\nu(A)=\{g_{\mu_1},\ldots,g_{\mu_n}\}$. Since 
$\Psi^\epsilon\circ\Phi^\epsilon= e_{\sheaf{F}}^{2\epsilon}$ and all sections in
$A$ are $2\epsilon$-significant, we have that  
$\sum_{l=1}^n\phi^{\mu_l}_{\lambda_i}\psi^{\lambda_j}_{\mu_l} 
= \delta_{ij}$, for any $i \leq j$, where $\delta_{ij}$ is Kronecker's delta.  Thus, the
product of the matrices $K:=(\phi^{\mu_l}_{\lambda_i})_{m\times n}$ and 
$L:=(\psi^{\lambda_j}_{\mu_l})_{n\times m}$ is a triangular matrix with diagonal 
entries all equal to $1$. Hence, the product $KL$ has rank $m$. This implies that
$|\nu(A)|=n\geq rank(K)\geq rank(KL)=m=|A|$.
\end{proof}

As pointed out earlier, the general strategy for the proof of stability is similar
to that used by Bjerkevik in the study of multi-parameter persistence 
\cite{Bjerkevik2016stability}, although the arguments for $p$-sheaves differ
substantially. A key matching result needed in the proof of the stability theorem, 
we employ the following classic result in combinatorics and graph theory.

\begin{theorem}[Hall's Matching Theorem] \label{T:hall}
Let $\{Z_x \colon x \in X\}$ be a collection of non-empty subsets of a set $Z$
such that $|A|\leq|\cup_{ x\in A} Z_x| <\infty$, for any finite set $A\subseteq X$.
Then, there exists an injective map $\zeta \colon X \to Z$ such that
$\zeta (x)\in Z_x$, $\forall x \in X$. 
\end{theorem}

\begin{lemma}[The Matching Lemma] \label{L:fmatching}
Suppose that $\sheaf{F}$ and $\sheaf{G}$ are decomposable $p$-sheaves 
with bases $B_{\sheaf{F}}$ and $B_{\sheaf{G}}$, respectively, and let 
\[
X := \{ f \in X^\epsilon_\sheaf{F} \colon |m_{\sheaf{G}}^\epsilon(f)| < \infty\}.
\]
If $\sheaf{F}$ and $\sheaf{G}$ are $\epsilon$-interleaved, then there exists an
injective $\epsilon$-map $\zeta \colon X \to B_{\sheaf{G}}$.
\end{lemma}

\begin{proof}
We construct $\zeta$ by applying Hall's Matching Theorem to $Z := B_\sheaf{G}$,
$Z_f := m^\epsilon_\sheaf{F} (f)$, for any $f \in X$. The fact that
$m^\epsilon_\sheaf{F} (f) \ne \emptyset$ follows from \eqref{E:mgf} and
Proposition \ref{P:cardinality} applied to the singleton $\{f\}$. Given any non-empty
finite  set $A \subseteq X$, write it as the disjoint union 
\begin{equation}
A = [A] \sqcup [A\rangle \sqcup \langle A] \sqcup \langle A \rangle\,,
\end{equation}
where $[A]$ only contains sections $f$ such that $\sheaf{F}\langle f \rangle
\cong k [p,q]$, the subset $[A\rangle$ only contains sections $f$ such that 
$\sheaf{F}\langle f \rangle \cong k [p,q\rangle$,
and similarly for the other two terms. If any of these subsets is empty,
we simply discard it. We can index the elements of $[A]$ appropriately, so
that $[A]$ satisfies the assumptions of Proposition \ref{P:cardinality}. Therefore,
for this labeling of $[A]$, we have $|[A]| \leq |\cup_{f\in [A]}m^\epsilon_\sheaf{G}(f)|$. 
Similar inequalities can be obtained for the other three subsets of $A$ after
appropriate labeling of their elements. Note that the sets
$\cup_{f\in [A]}m^\epsilon_\sheaf{G}(f)$, $\cup_{f\in [A\rangle}m^\epsilon_\sheaf{G}(f)$, 
$\cup_{f\in \langle A]}m^\epsilon_\sheaf{G}(f)$, and 
$\cup_{f\in \langle A \rangle}m^\epsilon_\sheaf{G}(f)$ are pairwise disjoint 
because they contain sections that span interval sheaves of different types.
Therefore, 
\begin{equation}
|A| = |[A]| + |[A\rangle| + |\langle A]| + |\langle A \rangle| \leq
 |\cup_{f\in A}m^\epsilon_\sheaf{G}(f)| \,.
\end{equation}
By Theorem \ref{T:hall}, there is an injection $\zeta \colon
X \to B_\sheaf{G}$ such that $\zeta (f) \in m^\epsilon_\sheaf{G}(f)$,
$\forall f \in X$. In other words, $\zeta$ is an injective $\epsilon$-map.
\end{proof}

\begin{theorem}[The Algebraic Stability Theorem] \label{T:stab}
Let $\sheaf{F}$ and $\sheaf{G}$ be decomposable $p$-sheaves and $\epsilon>0$.
If $\sheaf{F}$ and $\sheaf{G}$ are $\epsilon$-interleaved, then $B_{\sheaf{F}}$ 
and $B_{\sheaf{G}}$ are $(\epsilon + \delta)$-matched, for any $\delta>0$.
Therefore,
\[
d_b (dgm(F), dgm(G)) \leq d_I (F, G).
\]
\end{theorem}

\begin{proof}
We begin by reducing the argument to the case in which both $B_{\sheaf{F}}$ and 
$B_{\sheaf{G}}$ are countable sets. Let $(\Phi^\epsilon, \Psi^\epsilon)$ be an 
$\epsilon$-interleaving between $\sheaf{F}$ and $\sheaf{G}$ represented in the 
bases $B_{\sheaf{F}}$ and $B_{\sheaf{G}}$ by the matrices 
$(\phi_\lambda^\mu)$ and $(\psi_\mu^\lambda)$, respectively. Let $G$ be the
bipartite graph with vertex set partitioned as
$B_\sheaf{F}\sqcup B_\sheaf{G}$, and with an edge between $f_\lambda$
and $g_\mu$ if $\phi_\lambda^\mu\neq 0$ or $\psi^\lambda_\mu\neq 0$.
Then, any connected component $C$ of $G$ is a bipartite subgraph whose
vertex set $V$ is countable. This can be seen as follows. If we fix $v_0\in V$, 
then for any integer $n \geq 0$, the set $B_n := \{ v\in C \colon d(v_0,v)\leq n\}$ 
is finite, where $d(v_0,v)$ is the distance in the graph $C$. Hence, 
$V=\cup_{n \geq 0} \,B_n$ is countable. Moreover, it is a straightforward
consequence of Definition \ref{D:inter} that if $\sheaf{F}$  and $\sheaf{G}$
are $\epsilon$-interleaved, so are $\sheaf{F}\langle C_1\rangle$ and
$\sheaf{G}\langle C_2\rangle$, the subsheaves of $\sheaf{F}$ and $\sheaf{G}$
spanned by $C_1$ and $C_2$, respectively. 

If there is an $(\epsilon+\delta)$-matching between $C_1$ and $C_2$, for any
connected component $C$ of $G$, then assembling the matchings over all 
connected components of $G$, we obtain an  $(\epsilon+\delta)$-matching
between $B_\sheaf{F}$ and $B_\sheaf{G}$. Thus, without loss of generality, 
we assume that $B_\sheaf{F}$ and $B_\sheaf{G}$ are countable.

By Lemma \ref{L:injectives}, to construct an $(\epsilon+\delta)$-matching 
between $B_\sheaf{F}$ and $B_\sheaf{G}$, it suffices to construct
injective $(\epsilon + \delta)$-maps $\sigma \colon X^\epsilon_\sheaf{F} \to 
B_\sheaf{G}$ and $\tau \colon X^\epsilon_\sheaf{G} \to  B_\sheaf{F}$.
To construct $\sigma$, we filter the countable set $X^\epsilon_\sheaf{F} 
\subseteq B_\sheaf{F}$ and define it inductively over the filtration.
Let 
\begin{equation}
X := \{ f \in X^\epsilon_\sheaf{F} \colon |m_{\sheaf{G}}^\epsilon(f)| < \infty\}
\end{equation}
and $\zeta \colon X \to B_\sheaf{G}$ be an injective $\epsilon$-map whose
existence is guaranteed by the Matching Lemma (Lemma \ref{L:fmatching}).

Before proceeding with the construction, we recall a standard fact:  if 
$N_1, N_2, \cdots \subseteq Z$ are infinite subsets of a set $Z$,
there exist infinite subsets $N'_i\subseteq N_i$, $i \geq 1$,
such that $N'_i\cap N'_j = \emptyset$, for any $i\neq j$.

List the sections in $X^\epsilon_\sheaf{F} \setminus X$ as $\{f_i \colon i \geq 1\}$.
For any fixed $\delta > 0$, since $|m_\sheaf{G}^\epsilon (f_i)|=\infty$, 
by Lemma \ref{L:tight}  below, there exists a sequence
\begin{equation}
S_i:=\{g_i^n \colon n \geq 1\}  \subseteq m_\sheaf{G}^\epsilon (f_i)
\end{equation}
such that any two sections in $S_i$ are
$\delta$-matched. Furthermore, by the remark in the previous paragraph,
we can assume  that $S_i \cap S_j = \emptyset$, for any $i\ne j$. If 
$X = \emptyset$, we simply define $\sigma (f_i)$, $i \geq 1$,  to be any 
element of $S_i$. This concludes the construction of $\sigma$ because it
is an injective  $\epsilon$-mapping, in particular, an injective 
$(\epsilon + \delta)$-map. If $X \ne \emptyset$, let 
\begin{equation}
X_0 := X \setminus \cup_{i \geq 1} \,\zeta^{-1} (S_i) 
\end{equation}
and define $\sigma_0 \colon X_0 \to B_\sheaf{G}$ as $\sigma_0 = \zeta|_{X_0}$,
which is an injective $\epsilon$-map.
For $i \geq 1$, inductively define
\begin{equation}
X_i := X_{i-1} \cup \zeta^{-1} (S_i) \cup \{f_i\}.
\end{equation}
Assuming that an injective $(\epsilon + \delta)$-map 
$\sigma_{i-1} \colon X_{i-1} \to B_\sheaf{G}$ has been constructed,
we extend it to $\sigma_i \colon X_i \to B_\sheaf{G}$, as follows:
\begin{enumerate}[(i)]
\item $\sigma_i (f_i) = g_i^1$;
\item if $f \in \zeta^{-1} (S_i) \subseteq X$ and $\zeta (f) = g_i^j$,
set $\sigma_i (f) = g_i^{j+1}$.
\end{enumerate}
Since $f$ and $g_i^j$ are $\epsilon$-matched and $g_i^j$ and $g_i^{j+1}$
are $\delta$-matched, we have that $f$ and $\sigma_i (f) = g_i^{j+1}$ are 
$(\epsilon + \delta)$-matched. Thus, $\sigma_i$ is an injective 
$(\epsilon + \delta)$-map. As $X^\epsilon_\sheaf{F} = \cup_{i \geq 1} X_i$,
the mapping $\sigma \colon X^\epsilon_\sheaf{F} \to B_\sheaf{G}$ given by
$\sigma|_{X_i} = \sigma_i$ has the desired properties.

Similarly, we construct an injective $(\epsilon + \delta)$-map 
$\tau \colon X^\epsilon_\sheaf{G} \to B_\sheaf{F}$. By Lemma \ref{L:injectives}, 
there exists an $(\epsilon+\delta)$-matching between $B_\sheaf{F}$ and
$B_\sheaf{G}$.

For the stability statement, let $\epsilon > d_I (F, G)$. Then, for
any $\delta> 0$, there is an $(\epsilon+\delta)$-matching between $B_\sheaf{F}$ 
and $B_\sheaf{G}$. Taking the infimum over $\delta>0$, it follows that
$d_b (dgm(\sheaf{F}), dgm(\sheaf{G})) \leq \epsilon$. Taking the infimum over $\epsilon>0$, 
we obtain $d_b (dgm(\sheaf{F}), dgm(\sheaf{G})) \leq d_I (F, G)$.
\end{proof}

\begin{lemma}\label{L:tight}
Let $f \in B_\sheaf{F}$ be a $2\epsilon$-significant section and 
$S\subseteq B_\sheaf{G}$ an infinite subset with the property that each section
$g \in S$ is $\epsilon$-interleaved with $f$. Then, for any $\delta>0$, there exists
an infinite sub-collection $S_\delta \subseteq S$ such that any two sections in
$S_\delta$ are $\delta$-interleaved.
\end{lemma}
\begin{proof}
Without loss of generality, we may assume that $S$ is countable.
Suppose that $\supp(f)=(s^\ast, t^\ast)$ and let $S=\{g_n \colon n\geq 1\}$ with
$\supp(g_n)=(s_n^\ast,t_n^\ast)$, $\forall n \geq 1$. Then, we have 
$s-\epsilon\leq s_n\leq s+\epsilon$ and $t-\epsilon\leq t_n\leq t+\epsilon$ for all $n$. 
Hence, $\{(s_n,t_n)\}$ is a bounded set under the $d^\infty$ metric and therefore
has at least one accumulation point, say, $(s_0,t_0)$. This implies that, for any $\delta>0$, 
we can choose a sub-collection of intervals $\{(s_{n_i}^\ast,t_{n_i}^\ast)\}_{i=1}^\infty$,
each contained in the $\delta/2$-neighborhood of $(s_0^-, t_0^+)$, with $s_{n_i} \to s_0$
and $t_{n_i} \to t_0$. By construction, any pair of sections in 
$S_\delta=\{g_{n_i} \colon i\geq 1\}$ are $\delta$-interleaved.
\end{proof}

\begin{theorem}[The Isometry Theorem] \label{T:iso}
If $\sheaf{F}$ and $\sheaf{G}$ are decomposable $p$-sheaves, then
\[
d_b (dgm(\sheaf{F}), dgm(\sheaf{G})) = d_I (\sheaf{F}, \sheaf{G}).
\]
\end{theorem}

\begin{proof}
This follows from Proposition \ref{P:cstability} and Theorem \ref{T:stab}.
\end{proof}


\section{Applications} \label{S:applications}

The formulation of persistent structures developed in this paper largely has been
motivated by applications. This section explores some of these applications, explaining how
correspondence modules relate to levelset zigzag persistence, to slices of 2-D persistence
modules, as well as how to obtain homological barcodes richer in geometric information than
those obtained from sublevel (or superlevel) set filtrations or extended persistence 
\cite{Cohen-Steiner2009}. Among other things, we establish a Mayer-Vietoris sequence 
relating sublevel and superlevel set homology modules to levelset homology modules.

\subsection{Levelset Persistence} \label{S:levelset}

Let $f \colon X \to \real$ be a continuous function defined on a topological space $X$.
It is of great interest to summarize the topological changes in the level sets of $f$ across
function values, as such summaries can provide valuable insights on $f$. However, unlike
sublevel sets, there are no natural mappings relating different level sets, so a common
practice is to use interlevel sets to interpolate level sets in a zigzag structure 
\cite{Carlsson2010,Carlsson2009,Carlsson2019}.
To be more precise, for $s \leq t$, denote the interlevel set between $s$ and $t$ by
$X_s^t := f^{-1} ([s,t])$. To further simplify notation, write $X[t]$ for the level sets $X_t^t$.
At the topological space level, we have inclusions $X[s] \hookrightarrow X_s^t
\hookleftarrow X[t]$ that induce homomorphisms
\begin{equation}
\begin{tikzcd}
H_\ast (X[s]) \ar[r, "\phi_s^t"] & H_\ast (X_s^t) &  H_\ast (X [t]) \ar[l, "\psi_s^t" ']
\end{tikzcd}
\end{equation}
on homology (with field coefficients). Thus, for an increasing sequence $(t_n)$,
we obtain a zigzag module
\begin{equation}
\begin{tikzcd}[cramped]
\ldots & H_\ast (X[t_{n-1}]) \ar[r, "\phi_{n-1}^n"]  \ar[l] & H_\ast (X_{t_{n-1}}^{t_n}) &
H_\ast (X [t_n]) \ar[l, "\psi_{n-1}^n" '] \ar[r] & \ldots
\end{tikzcd}
\label{E:zigzag}
\end{equation}
Using correspondences, we eliminate the homology of interlevel sets from
\eqref{E:zigzag}, treating the triple $(H_\ast (X_{t_{n-1}}^{t_n}), \phi_{n-1}^n, \psi_{n-1}^n)$
as a $\cat{CVec}$-morphism from $H_\ast (X[t_{n-1}])$ to $H_\ast (X[t_n])$. This has
the virtue of leaving only the homology of the level sets as objects in the sequence,
also leading to a categorical formulation of level set persistence that easily
extends to a continuous parameter $t \in \real$.

To state the next theorem, recall that we denote the graph of a mapping $T$
by $G_T$ and the operator that reverses correspondences by $^\ast$. For $s,t \in \real$,
$s \leq t$, define a correspondence $h_s^t \subseteq H_\ast(X[s]) \times H_\ast(X[t])$ by
$h_s^t = G_{\psi_s^t}^\ast \circ G_{\phi_s^t}$, and let
$\cmod{H}^-_\ast (f) := (H_\ast (X[t]), h_s^t)$, $s,t \in \real, s \leq t$, which we refer to
as the levelset $c$-module associated with $f$. As in \cite{Carlsson2019}, the homology theory
used is Steenrod-Sitnikov homology \cite{Milnor1961}.

\begin{theorem} \label{T:level}
Let $X$ be a locally compact polyhedron. If $f \colon X \to \real$ is a proper continuous
function and $H_\ast$ is Steenrod-Sitnikov homology with field coefficients, then 
$\cmod{H}^-_\ast (f)$ is a $c$-module. 
\end{theorem}

\begin{proof}
To prove that  $\cmod{H}^-_\ast (f)$ is a $c$-module, it suffices to verify the validity
of the composition rule $h_r^t = h_s^t \circ h_r^s$, for any $r \leq s \leq t$. For the
argument we present, it will be useful to consider the inclusions of interlevel sets
$X_r^s \hookrightarrow X_r^t \hookleftarrow X_s^t$, for $r \leq s \leq t$, and the induced
homomorphisms
\begin{equation}
\begin{tikzcd}
H_\ast (X_r^s) \ar[r, "\rho_r^{s,t}"] & H_\ast (X_r^t) &  
H_\ast (X_s^t) \ar[l, "\sigma_{r,s}^t" '] .
\end{tikzcd}
\end{equation}
The proof amounts to a chase in the commutative diagram
\begin{equation}
\begin{tikzcd}[column sep=tiny]
& & H_\ast(X_r^t) & & \\
& H_\ast (X_r^s) \ar[ur, "\rho_r^{s,t}" '] & & H_\ast (X_s^t) \ar[ul, "\sigma_{r,s}^t"] & \\
H_\ast (X[r]) \ar[ur, "\phi_r^s" '] \ar[uurr, bend left=35, "\phi_r^t"]
& & H_\ast(X[s]) \ar[ul, "\psi_r^s"] \ar[ur, "\phi_s^t" '] 
& & H_\ast(X[t]) \ar[ul, "\psi_s^t"] \ar[uull, bend right=35, "\psi_s^t" '] \,,
\end{tikzcd}
\label{E:diagram}
\end{equation}
noting that the assumptions on $X$ and $f$ along with the fact that $H_\ast$ is
Steenrod-Sitnikov homology imply that the center diamond in \eqref{E:diagram}
is exact (Proposition 3.7 of \cite{Carlsson2019}). This means that the sequence
\begin{equation}
\begin{tikzcd}
H_\ast (X[s]) \ar[r, "\alpha"] & H_\ast (X_r^s) \oplus  H_\ast (X_s^t) \ar[r, "\beta"]  &  
H_\ast (X_r^t)
\end{tikzcd}
\end{equation}
is exact, where $\alpha (a) = \phi_r^s (a) \oplus \psi_s^t (a)$ and 
$\beta (a,b) = \rho_r^{s,t} (a) - \sigma_{r,s}^t (b)$.

To check that $h_s^t \circ h_r^s \subseteq h_r^t$, let $(v_r, v_s) \in h_r^s$
and $(v_s, v_t) \in h_s^t$. Set $v = \rho_r^{s,t} \circ \phi_r^s (v_s) =
\sigma_{r,s}^t (v_s)$. From \eqref{E:diagram}, it follows that
$\phi_r^t (v_r) = \psi_s^t (v_t) = v$, which implies that $(v_r, v_t) \in h_r^t$.
For the reverse inclusion, let $(v_r, v_t) \in h_r^t$, which means that
$\phi_r^t (v_r) = \psi_s^t (v_t)$. The commutativity of the diagram implies that
the vectors $w_1 = \phi_r^s (v_r)$ and $w_2 = \psi_s^t (v_t)$ satisfy
$\rho_r^{s,t} (w_1) = \sigma_{r,s}^t (w_2)$. By exactness, there exists
$v_s \in H_\ast(X[s])$ such that $\phi_r^s (v_s) = w_1$ and 
$\psi_s^t (v_s) = w_2$. Thus, $(v_r, v_s) \in h_r^s$ and
$(v_s, v_t) \in h_s^t$, showing that $(v_r, v_t) \in h_s^t \circ h_r^s$.
\end{proof}

\begin{remark}
If $\dim H_\ast(X[t]) < \infty$, $\forall t \in \real$, then the levelset $c$-module
$\cmod{H}^-_\ast (f)$ is virtually tame, thus admitting an interval decomposition. This
is the case, for example, if $X$ is a locally compact polyhedron and $f$ is a proper
piecewise-linear map. As in \cite{Carlsson2019}, using rectangle measures, one may define
persistence diagrams under the more general setting of Theorem \ref{T:level}, without
requiring an interval decomposition of $\cmod{H}^-_\ast (f)$. However, we refrain
from exploring this point of view in this paper.
\end{remark}

\begin{example}
Let $f \colon X \to \real$ be as indicated in Fig.\,\ref{F:level}. The interval decomposition
of the levelset $c$-module $H_0^- (f)$ contains all four types of interval modules, as
indicated in the barcode. The bar types follow the convention made in Remark
\ref{R:arrows}.
\begin{center}
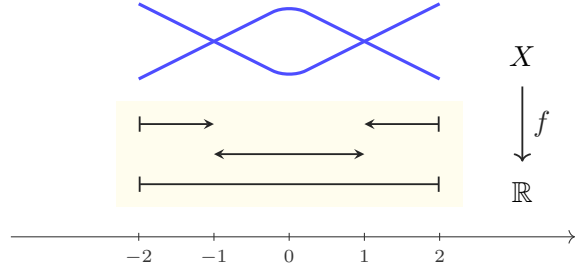
\begin{figure}[h]
\begin{tikzpicture}[line width=1.2pt]
\begin{scope}[blue!70]
\draw (-2, 3.1) {[rounded corners=7pt] -- (0, 2.1)} -- (2, 3.1);
\draw (-2, 2.1) {[rounded corners=7pt] -- (0, 3.1)} -- (2, 2.1);
\end{scope}
\draw (3.1, 2.4) node {$X$} (3.1, 0.6) node {$\boldsymbol{\real}$};
\draw[->, line width=0.7pt, black!85] (3.1, 2) -- node[right=0.1pt] {$f$} (3.1, 1);
\begin{scope}[line width=0.7pt, >=stealth, black!85]
\fill[yellow!8!white] (-2.3, 0.4) rectangle (2.3, 1.8);
\draw [|->] (-2, 1.5) -- (-1, 1.5); 
\draw [<-|] (1, 1.5) -- (2, 1.5);
\draw [<->] (-1, 1.1) -- (1, 1.1);
\draw [|-|] (-2, 0.7) -- (2, 0.7);
\end{scope}
\draw [->, thin, black!80]  (-3.7, 0) -- (3.8, 0);
\foreach \x in {-2, -1, ..., 2}
\draw[thin, black!80] (\x, -1.5pt) -- (\x, 1.5pt) node[below=2.5pt] {\tiny{$\x$}};
\end{tikzpicture}
\caption{Persistence diagram for $H_0^- (f)$.}
\label{F:level}
\end{figure}
\end{center}
\end{example}

\begin{remark} \label{R:2dext}
There is an interesting connection between the present formulation of
levelset homology using $c$-modules and 2-D  persistence modules. Consider 
$\real^2$  with the partial ordering in which $(s_0, t_0) \cle (s_1, t_1)$ if 
$s_o \geq s_1$ and $t_0 \leq t_1$. As a category, this ordering corresponds to 
$\real^{\text{op}} \times \real$. Let $\Delta^+ := \{(s,t) \in \real^2 \colon s \leq t\}$ 
and $\Delta^- := \{(s,t) \in \real^2 \colon s \geq t\}$. Similar to the extension of
zigzag modules to an exact 2-D persistence module \cite{Botnan2018,Cochoy2020}, 
it is possible to extend a $c$-module to a persistent module over $\real^2$.
Place the $c$-module  $H^-_\ast (f)$ along the diagonal $\Delta \subseteq \real^2$.
Under the assumptions of
Theorem \ref{T:level}, one can define an exact $p$-module over $\Delta^+$,
extending $H^-_\ast (f)$, whose vector space at $(s,t) \in \Delta^+$ is 
$V_{(s,t)} := H_\ast (X_s^t)$ and whose  morphisms $v_{(s_0, t_0)}^{(s_1, t_1)}$ 
are the mappings on homology induced  by the inclusions 
$X_{s_0}^{t_0} \hookrightarrow X_{s_1}^{t_1}$. Letting $\sheaf{F}$ denote the 
$p$-sheaf of  sections of the $c$-module  $\cmod{H}^-_\ast (f)$, define a persistence 
module over $\Delta^-$  by $V_{(s,t)} := \sheaf{F} ([t,s])$ and  
$v_{(s_0, t_0)}^{(s_1, t_1)} :=  \sheaf{F}^{[t_0, s_0]}_{[t_1, s_1]}$. One can verify 
that these two structures combine to yield a single exact 2-D persistence module 
$\cmod{V}$ over $\real^{\text{op}} \times \real$. 
\end{remark}


\subsection{A Persistent Mayer-Vietoris Sequence} \label{S:mv}

Let $f \colon X \to \real$ be a continuous function. Under the assumptions of
Theorem \ref{T:level}, in this section, we construct a Mayer-Vietoris (M-V) sequence of
$c$-modules  for covers of $X$ given by sublevel and superlevel sets of $f$. 
The level set $c$-modules $\cmod{H}_i^- (f)$, $i \geq 0$, constructed in
Section \ref{S:levelset}, represents the homology of the intersections of the
elements of these covers.  Recall that $\cat{CMod}$ is the category whose objects
are the $c$-modules (over $\real$) with natural transformations as morphisms.
In general, morphisms in $\cat{CMod}$ do not have kernels. However, in this
particular case, there is sufficient structure to formulate 
an exact M-V sequence. We begin with the construction of
the objects in the sequence. 

Denote the sublevel sets of $f$ by $X^t := f^{-1} ((-\infty, t])$ and the superlevel
sets of $f$ by $X_t := f^{-1} ([t, +\infty))$. For $s \leq t$ and $i \geq 0$, let $\imath_s^t \colon
H_i (X^s) \to H_i (X^t)$ and $\jmath_s^t \colon H_i (X_t) \to H_i (X_s)$ be
the morphisms induced by the inclusions $X^s \subseteq X^t$ and $X_t \subseteq X_s$,
respectively. We denote the graphs of $\imath_s^t$ and $\jmath_s^t$ by $I_s^t$ and
$J_s^t$, respectively. Then, $\cmod{H}_i^\vee (f) := (H_i X^t, \imath_s^t)$ and
$\cmod{H}_i^\wedge (f) := (H_i (X_t), \jmath_s^{t\ast})$ are the homology $c$-modules associated
to the sublevel and superlevel filtrations of $X$, respectively. The direct sum
$\cmod{H}_i^\vee (f) \oplus \cmod{H}_i^\wedge (f)$ is the homology $c$-module
associated with the covers $X = X^t \cup X_t$, $t \in \real$. Note that, for $s \leq t$, the
elements $(a_s, b_s) \in H_i (X^s) \oplus H_i (X_s)$ and $(a_t, b_t) \in
H_i (X^t) \oplus H_i (X_t)$  are in correspondence in $\cmod{H}_i^\vee (f)
\oplus \cmod{H}_i^\wedge (f)$ if and only if $\imath_s^t (a_s) = a_t$ and
$\jmath_s^t (b_t) = b_s$. We also define the ``constant''
$c$-module $\cmod{H}_i (X)$ in which the vector space over any
$t \in \real$ is $H_i(X)$ with the diagonal subspace $\Delta_{H_i (X)}$ as
correspondence, for any $s \leq t$.

Now we define the relevant $c$-module morphisms for the M-V sequence.
Let $p_i^t \colon  H_i (X^t) \to H_\ast (X)$ and $q_i^t \colon H_i (X_t)
\to H_\ast(X)$ be the mappings induced on homology by the inclusions
$X^t \subseteq X$ and $X_t \subseteq X$. The commutativity of the diagrams
\begin{equation}
\begin{tikzcd}[column sep=tiny]
H_\ast (X^s) \ar[rr, "\imath_s^t"] \ar[dr, "p_i^s" '] & & H_\ast (X^t) \ar[dl, "p_i^t"] \\
& H_\ast (X) &
\end{tikzcd}
\quad
\begin{tikzcd}[column sep=tiny]
H_\ast (X_s) \ar[dr, "q_i^s" '] & & H_\ast (X_t) \ar[dl, "q_i^t"] \ar[ll, "\jmath_s^t" '] \\
& H_\ast (X) &
\end{tikzcd}
\end{equation}
for any $s \leq t$, implies that the mappings $p_i^t - q_i^t \colon H_i (X^t)
\oplus H_i (X^t) \to H_i (X)$, given by $(a_t, b_t) \mapsto
p_i^t (a_t) - q_i^t (b_t)$, induce a $c$-module morphism
\begin{equation}
p_i - q_i \colon \cmod{H}_i^\vee (f) \oplus \cmod{H}_i^\wedge (f) \to
\cmod{H}_i (X) .
\label{E:mvmap1}
\end{equation}

Similarly, the inclusions $X[t] \subseteq X^t$ and $X[t] \subseteq X_t$,
$t \in \real$, induce mappings $\psi_i^t \colon H_i (X[t]) \to H_i (X^t)$
and $\phi_i^t \colon H_i (X[t]) \to H_i (X_t)$ on homology, which in turn
induce a $c$-module morphism
\begin{equation}
\psi_i \oplus \phi_i \colon \cmod{H}_i^- (f) \to
\cmod{H}_i^\vee (f) \oplus \cmod{H}_i^\wedge (f) .
\label{E:mvmap2}
\end{equation}

To define the connecting $c$-module morphisms, we work under the assumptions
of Theorem \ref{T:level}. For each $t \in \real$, consider the cover $X = X^t \cup X_t$,
as well as the coarser cover $X = X^t \cup X_s$, for $s \leq t$. Naturality of the
Mayer-Vietoris sequence implies that inclusions yield a commutative diagram
\begin{equation}
\begin{tikzcd}[column sep=small]
\ar[r, dotted] & H_i (X) \ar[r, "\partial_i"] \ar[d, "id"] & H_{i-1} (X[s]) \ar[r] \ar[d, "\phi_s^t"] &
H_{i-1} (X^s)  \oplus H_{i-1} (X_s) \ar[d] \ar[r, dotted, twoheadrightarrow] & H_0 (X) \ar[d, "id"] \\
\ar[r, dotted] & H_i (X) \ar[r, "\partial_i"] & H_{i-1} (X_s^t) \ar[r] &
H_{i-1} (X^t)  \oplus H_{i-1} (X_s) \ar[r, dotted, twoheadrightarrow] & H_0 (X) \\
\ar[r, dotted] & H_i (X) \ar[r, "\partial_i"] \ar[u, "id" '] & H_{i-1} (X[t]) \ar[r] \ar[u, "\psi_s^t" '] &
H_{i-1} (X^t)  \oplus H_{i-1} (X_t) \ar[u] \ar[r, dotted, twoheadrightarrow] & H_0 (X) \ar[u, "id" '] \,.
\end{tikzcd}
\end{equation}
In particular, for any $i \geq 1$ and $s \leq t$, the diagram
\begin{equation}
\begin{tikzcd}[column sep=small, row sep=small]
& H_{i-1} (X[s]) \ar[rd, "\phi_s^t"] & \\
H_i (X) \ar[ur, "\partial_i"] \ar[dr, "\partial_i" '] & & H_{i-1} (X_s^t)  \\
& H_{i-1} (X[t]) \ar[ur, "\psi_s^t" '] & 
\end{tikzcd}
\end{equation}
commutes, showing that the mappings $\partial_i \colon H_i (X) \to H_{i-1} (X[t])$,
$t \in \real$, induce a $c$-module connecting morphism
\begin{equation}
\Delta_i \colon \cmod{H}_i (X)
\to \cmod{H}^-_{i-1} (X) \,.
\label{E:mvmap3}
\end{equation}

\begin{theorem} \label{T:mv}
If $f \colon X \to \real$ is a proper, continuous function defined on a locally
compact polyhedron $X$ and $H_\ast$ is Steenrod-Sitnikov homology, then
the Mayer-Vietoris sequence
\[
\begin{tikzcd}
\dots \ar[r] & \cmod{H}_{i+1} (X) \ar[r, "\Delta_{i+1}"] & \cmod{H}_i^- (f) 
\ar[r, "\psi_i \oplus \phi_i "] & \cmod{H}_i^\vee (f) \oplus \cmod{H}_i^\wedge (f)
\ar[r, "p_i - q_i"] & \dots
\end{tikzcd}
\]
is exact in the  $\cat{CMod}$ category.
\end{theorem}
\begin{proof}
If $\begin{tikzcd}[column sep=small, cramped] \cmod{W} \ar[r, "G"] & \cmod{U} \ar[r, "F"] 
&  \cmod{V} \end{tikzcd}$ are any two consecutive morphisms in the sequence, by
construction, $\begin{tikzcd}[column sep=small, cramped] W_t \ar[r, "g_t"] & U_t \ar[r, "f_t"] 
&  V_t \end{tikzcd}$ is exact, $\forall t \in \real$. By Proposition \ref{P:imker}(i) and (ii), all 
morphisms in the sequence have well-defined image and kernel $c$-modules. 
$\cat{CMod}$ exactness follows from $\cat{CVec}$ exactness at each $t \in \real$.
\end{proof}

\begin{example}
Let $f \colon X \to \real$ be projection of the space depicted in Fig.\,\ref{F:cokernel} 
to a horizontal axis. By Proposition \ref{P:imker}(iii), $\text{coker} (p_1 - q_1)$ is
a $c$-module. The barcode for $\text{coker} (p_1 - q_1)$, also shown in the figure,
exactly captures the horizontal spread of the 1-dimensional cycles of $X$.
\begin{center}
\begin{figure}[h]
\begin{tikzpicture}[line width=1.2pt, scale=1.1]
\begin{scope}[blue!60]
\draw (-0.8, 3.3) -- (0.8, 3.3) {[rounded corners]-- (1, 2.6)} -- (0.4, 1.6)
-- (-0.4, 1.6) {[rounded corners] -- (-1, 2.6)} -- cycle;

\draw (-0.45, 2.8) circle (0.2); \filldraw (-0.45, 2.75) ellipse (1pt and 2pt);
\draw (0.45, 2.8) circle (0.2); \filldraw (0.45, 2.75) ellipse (1pt and 2pt);
\draw (0, 2.2) ellipse (0.45 and 0.1);
\end{scope}
\draw [->, thin, black!80]  (-2.5, 0.25) -- (2.8, 0.25);
\foreach \x in {-1.5, -1, ..., 1.5}
\draw[thin, black!80] (\x, 0.2) -- (\x, 0.3) node[below=2.5pt] {\tiny{$\x$}};
\begin{scope}[line width=0.5pt, >=stealth, color=black!90]
\fill[yellow!8!white] (-1.3, 0.55) rectangle (1.3, 1.4);
\draw [<->] (-0.65, 1.2) -- (-0.25, 1.2); 
\draw [<->] (0.25, 1.2) -- (0.65, 1.2);
\draw [<->] (-0.45, 1) -- (0.45, 1);
\draw [<->] (-1, 0.8) -- (1, 0.8);
\end{scope}
\draw (2, 2.6) node {$X$} (2, 0.8) node {$\boldsymbol{\real}$};
\draw[->, line width=0.7pt, black!85] (2, 2.2) -- node[right=0.1pt] {$f$} (2, 1.2);
\end{tikzpicture}
\caption{Barcode for the cokernel of $p_1-q_1 \colon
\cmod{H}_1^\vee (f) \oplus \cmod{H}_1^\wedge (f) \to \cmod{H}_1 (X)$.}
\label{F:cokernel}
\end{figure}
\end{center}
In this example, the barcode for extended persistence \cite{Cohen-Steiner2009} also encodes 
the same geometric properties, but $c$-modules present the information in a categorical
framework that naturally integrates various different types of persistence architectures.
\end{example}


\subsection{Slicing 2-D Persistence Modules}

Multidimensional persistent homology is of great practical interest, as topological analysis
of complex data frequently gives rise to topological or simplicial filtrations that depend
on multiple parameters. However, unlike the one-dimensional case in which the barcode of
a sufficiently tame $p$-module yields a complete invariant, it is impossible to obtain a
discrete, complete representation of the structure of multidimensional $p$-modules
\cite{Carlsson2009multiD}. As such, it is of interest to define, albeit incomplete, computable and
informative invariants for these modules, such as rank invariants 
\cite{Carlsson2009multiD,Patel2018} and some numeric invariants \cite{Skryzalin2017numeric}, 
to summarize their structural  properties. Lesnick and Wright slice 2-D persistence modules 
along affine lines of non-negative  slopes to obtain a family of  one-dimensional $p$-modules 
whose structures may be described by
persistence diagrams \cite{Lesnick2015interactive}. Here, we show how to define a 
$c$-module structure along affine lines of negative slope. If the original 2-D persistence 
module is pointwise finite dimensional, then these negatively sloped slices are virtually tame 
thus admitting an interval decomposition.

We view $(\real^2, \leq)$ as a poset, where $(x_1, y_1) \leq (x_2, y_2)$ if and only if 
$x_1 \leq x_2$ and $y_1 \leq y_2$. Let $\ell \subseteq \real^2$ be an affine line of
negative slope. We fix an orientation  for $\ell$ via the unit vector $u = e^{i \theta}$,
$-\pi/2 < \theta < 0$, parallel to $\ell$. This induces a linear ordering on $\ell$ given by
$s \cle t$ if and only if $(t-s) \cdot u \geq 0$. Given a 2-D persistence module
$\cmod{U} \colon \real^2 \to \cat{Vec}$, we define a $c$-module
$\cmod{V} \colon \real \to \cat{CVec}$, termed the slice of $\cmod{U}$ along $\ell$,
as follows. The vector space at $t \in \ell$ is $V_t := U_t$. To define the correspondences,
we introduce some notation. For $a,b \in \real^2$, let $a \vee b \in \real^2$ be the
(unique) element that is initial with respect
to the property that $a \leq a \vee b$ and $b \leq a \vee b$. If $a = (x_1, y_1)$
and $b = (x_2, y_2)$ satisfy $x_1 \leq x_2$ and $y_1 \geq y_2$, then
$a \vee b := (x_2, y_1)$. For $s, t \in \ell$ with $s \cle t$, let $P(s,t)$ be collection
of all finite sequences $T = (t_i)_{i=0}^n$ of points in $\ell$ satisfying
$t_0 \cle  \ldots \cle t_n$, $t_0 =s$ and $t_n = t$. Letting $r_i = t_{i-1} \vee t_i$,
define the staircase $\Gamma_T$ associated to $T \in P(s,t)$ as the
sequence in $\real^2$ given by
\begin{equation}
t_0 \leq r_1 \geq t_1 \leq \ldots \geq t_{n-1} \leq r_n \geq t_n \,,
\end{equation}
as depicted in Fig.\,\ref{F:staircase}(i). 
\begin{figure}
\begin{center}
\begin{tabular}{cc}
\begin{tikzpicture}[scale=0.6]
\draw (-3.5,3.5) -- (3.5,-3.5);
\fill (-3,3) circle (3pt) node[anchor=east] {$s=t_0$};
\fill (3,-3) circle (3pt);
\draw[-stealth] (-3,3) -- (-2.15,3);
\draw (-2,3) circle (2pt) node[anchor=west] {$r_1$};
\draw[-stealth] (-2,2) -- (-2,2.85);
\draw[-stealth] (-2,2) node[anchor=east]{$t_1$} -- (-1.15,2);
\draw[-stealth,dotted] (-1,1) -- (-1,1.85);
\draw (-1,2) circle (2pt) node[anchor=west] {$r_2$};
\draw[-stealth,dotted] (-1,1) node[anchor=east]{$t_i$} -- (0.55,1);
\draw (0.7,1) circle (2pt) node[anchor=west] {$r_i$};
\draw[-stealth] (3,-3) node[anchor=east] {$t=t_n$} -- (3,-2.15) node[anchor=west] {$r_n$};
\draw (3,-2) circle (2pt);
\draw[-stealth] (2,-2) node[anchor=east] {$t_{n-1}$} -- (2.85,-2);
\draw (3,-2) circle (2pt);
\draw[-stealth] (2,-2) -- (2,-0.85);
\draw (2,-0.7) circle (2pt) node[anchor=west] {$r_{n-1}$};
\draw[-stealth,dotted] (0.7,-0.7) -- (1.85,-0.7);
\draw[-stealth,dotted] (0.7,-0.7) -- (0.7,0.85);
\end{tikzpicture}
& 
\begin{tikzpicture}[scale=0.6]
\draw (-3.5,3.5) -- (3.5,-3.5);
\fill (-3,3) circle (3pt) node[anchor=east] {$v_s=w_0$};
\fill (3,-3) circle (3pt);
\draw[-stealth] (-3,3) -- (-2.15,3);
\draw (-2,3) circle (2pt) node[anchor=west] {$z_1$};
\draw[-stealth] (-2,2) -- (-2,2.85);
\draw[-stealth] (-2,2) node[anchor=east]{$w_1$} -- (-1.15,2);
\draw[-stealth,dotted] (-1,1) -- (-1,1.85);
\draw (-1,2) circle (2pt) node[anchor=west] {$z_2$};
\draw[-stealth,dotted] (-1,1) node[anchor=east]{$w_i$} -- (0.55,1);
\draw (0.7,1) circle (2pt) node[anchor=west] {$z_i$};
\draw[-stealth] (3,-3) node[anchor=east] {$v_t=w_n$} -- (3,-2.15) node[anchor=west] {$z_n$};
\draw (3,-2) circle (2pt);
\draw[-stealth] (2,-2) node[anchor=east] {$w_{n-1}$} -- (2.85,-2);
\draw (3,-2) circle (2pt);
\draw[-stealth] (2,-2) -- (2,-0.85);
\draw (2,-0.7) circle (2pt) node[anchor=west] {$z_{n-1}$};
\draw[-stealth,dotted] (0.7,-0.7) -- (1.85,-0.7);
\draw[-stealth,dotted] (0.7,-0.7) -- (0.7,0.85);
\end{tikzpicture} \smallskip \\
(i) a staircase & (ii) staircase interpolation
\end{tabular}
\caption{Interpolating $v_s \in V_s$ and $v_t \in V_t$ along a staircase.}
\label{F:staircase}
\end{center}
\end{figure}
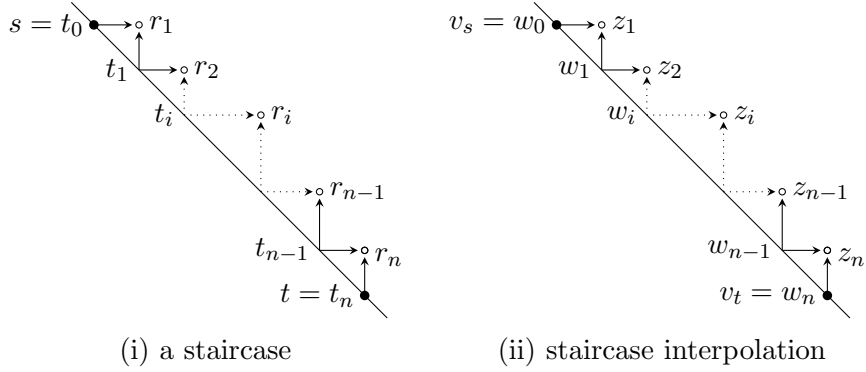
Using staircases, define a correspondence $v_s^t$, as follows.

\begin{definition} \label{D:stair}
A pair $(v_s, v_t) \in v_s^t$
if and only if for any staircase $\Gamma_T$, $T \in P(s,t)$, there are vectors 
that interpolate $v_s$ and $v_t$ along $\Gamma_T$, as illustrated in 
Fig.\,\ref{F:staircase}(ii). More precisely, there are vectors 
$w_i \in U_{t_i}$, $0 \leq i \leq n$, and $z_i \in U_{r_i}$, $1 \leq i \leq n$, such that
\begin{enumerate}[(i)]
\item $w_0 = v_s$ and $w_n = v_t$;
\item $u_{t_{i-1}}^{r_i} (w_{i-1}) = z_i$ and $u_{t_i}^{r_i} (w_i) = z_i$, for
$1 \leq i \leq n$.
\end{enumerate}
\end{definition}

\begin{lemma} \label{L:stairs}
Let $\cmod{V}$ be a slice of a 2-D persistence module $\cmod{U}$ along a line
of negative slope, and let $S, T \in P(s,t)$ with $S \subseteq T$. If $v_s \in V_s$
and $v_t \in V_t$ may be  interpolated along $\Gamma_T$, then $v_s$ and $v_t$
also may be interpolated along the coarser staircase $\Gamma_S$.
\end{lemma}
\begin{proof}
Using an iterative argument, it suffices to consider the case where $S$ and $T$
differ by a single element, say,
\begin{equation}
T = \{t_0, \ldots, t_j,  \ldots, t_n\}
\ \text{and} \
S = \{t_0, \ldots, t_{j-1}, \hat{t}_j,  t_{j+1}, \ldots, t_n\} ,
\end{equation}
where $\hat{t}_j$ indicates deletion of $t_j$. Using the notation of
Definition \ref{D:stair}, suppose $w_i \in U_{t_i}$ and $z_i \in U_{r_i}$
interpolate $v_s$ and $v_t$ along $\Gamma_T$. Let 
$\bar{r}_j = t_{j-1}\vee t_{j+1}$ and $\bar{z}_j = u_{t_{j-1}}^{\bar{r}_j} (w_{j-1})$.
Then,
\begin{equation}
\begin{split}
\bar{z}_j &= u_{t_{j-1}}^{\bar{r}_j} (w_{j-1}) = u_{r_j}^{\bar{r}_j} \circ 
u_{t_j}^{r_j}(w_j) = u_{r_{j+1}}^{\bar{r}_j} \circ u_{t_j}^{r_{j+1}} (w_j) \\
&= u_{r_{j+1}}^{\bar{r}_j} (z_{j+1}) = u_{r_{j+1}}^{\bar{r}_j} \circ 
u_{t_{j+1}}^{r_{j+1}} (w_{j+1}) = u_{t_{j+1}}^{\bar{r}_j}  (w_{j+1}) \,.
\end{split}
\end{equation}
Thus, $\bar{z}_j = u_{t_{j-1}}^{\bar{r}_j} (w_{j-1}) = 
u_{t_{j+1}}^{\bar{r}_j}  (w_{j+1})$, showing that the vectors
\begin{equation}
w_0, \ldots, w_{j-1}, w_{j+1}, \ldots, w_n
\quad \text{and} \quad
r_1, \ldots, r_{j-1}, \bar{r}_j, r_{j+2}, \ldots, r_n
\end{equation}
interpolate $v_s$ and $v_t$ along the staircase $\Gamma_S$.
\end{proof}

\begin{theorem}
Let $\cmod{U} \colon \real^2 \to \cat{Vec}$ be a pointwise finite-dimensional
persistence module. If $\ell \subseteq \real^2$ is a negatively sloped line, then
the slice of $\cmod{U}$ along $\ell$ is a virtually tame $c$-module.
\end{theorem}
\begin{proof}
Let $\cmod{V}$ be the slice of $\cmod{U}$ along $\ell$.
To show that $\cmod{V}$ is a $c$-module, it suffices to verify the composition rule
$v_r^t = v_s^t \circ v_r^s$ for morphisms. We begin with the inclusion 
$v_s^t \circ v_r^s \subseteq v_r^t $. Let $(v_r, v_s) \in v_r^s$ and $(v_s, v_t) \in
v_s^t$. Given $T \in P(r,t)$, let $\overline{T} = T \cup \{s\}$. Write $\overline{T}$ as
a union $\overline{T} = T_1 \cup T_2$, $T_1 \in P(r,s)$, and $T_2 \in P(s,t)$. By
assumption, we may interpolate $v_r$ and $v_s$ along $\Gamma_{T_1}$, as well
as $v_s$ and $v_t$ along $\Gamma_{T_2}$. Concatenating these interpolations,
we obtain an interpolation of $v_r$ and $v_t$ along $\Gamma_{\overline{T}}$.
By Lemma \ref{L:stairs}, $v_r$ and $v_t$ also
can be interpolated along the coarser staircase $\Gamma_T$. This proves that
$(v_r, v_t) \in v_r^t$.

For the reverse inclusion, let $(v_r, v_t) \in v_r^t$ and $s \in \ell$ be such that
$r \cle s \cle t$. Our goal is to show that there exists $v_s \in V_s = U_s$ such that
$(v_r, v_s) \in v_r^s$ and $(v_s, v_t) \in v_s^t$. Given $T_1 \in P(r,s)$ and
$T_2 \in P(s,t)$, let $V_{T_1, T_2}$ be the affine subspace of $V_s$ comprising
those vectors $v_s \in V_s$ such that:
\begin{enumerate}[(i)]
\item $v_r$ and $v_s$ may be interpolated along $\Gamma_{T_1}$;
\item $v_s$ and $v_t$ may be interpolated along $\Gamma_{T_2}$. 
\end{enumerate}
Note that $V_{T_1, T_2}$ is non-empty because $(v_s, v_t) \in v_s^t$ implies
that $v_r$ and $v_t$ may be interpolated along the staircase
$\Gamma_{T_1 \cup T_2}$. To conclude the argument, we show that 
\begin{equation}
W_s := \bigcap_{\substack{T_1 \in P(r,s) \\ T_2 \in P(s,t)}} V_{T_1, T_2}
\ne \emptyset \,,
\label{E:int}
\end{equation}
as this implies that any $v_s \in W_s$ satisfies $(v_r, v_s) \in v_r^s$
and $(v_s, v_t) \in v_s^t$, as desired. 

Let $A = \{V_{T_1, T_2} \colon T_1 \in P(r,s) \ \text{and} \ T_2 \in P(s,t)\}$,
partially ordered via inclusion. Since $V_s$ is finite dimensional, each descending
chain in $A$ stabilizes in finitely many steps, thus having a lower bound. By Zorn's
Lemma, $A$ has a minimal element $V_{R_1, R_2}$. We show that
$V_{R_1, R_2} \subseteq V_{T_1, T_2}$, for any $T_1 \in P(r,s)$ and $T_2 \in P(s,t)$. Let
$S_1 = T_1 \cup R_1$ and $S_2 = T_2 \cup R_2$. By Lemma \ref{L:stairs},
$V_{S_1, S_2} \subseteq V_{T_1, T_2}$ and $V_{S_1, S_2} \subseteq V_{R_1, R_2}$.
By minimality, $V_{R_1, R_2} = V_{S_1, S_2} \subseteq V_{T_1, T_2}$. 
Hence, $W_s = V_{R_1, R_2} \ne \emptyset$, as claimed.

The virtual tameness of $\cmod{V}$ follows from the assumption that
$\cmod{U}$ is pointwise finite dimensional. 
\end{proof}

\section{Closing Remarks} \label{S:remarks}

This paper introduced and developed two main concepts: (i) correspondence modules
that generalize structures such as persistence modules and zigzag modules and 
(ii) persistence sheaves that provide a 
pathway to the structural analysis of $c$-modules. Using sheaf-theoretical arguments, 
we proved interval decomposition theorems for sufficiently tame $c$-modules and 
$p$-sheaves parameterized over $\real$, as well as a stability theorem for 
persistence-diagram representations of $p$-sheaves. Applications
discussed in the paper include: (1) a new formulation of continuously parameterized
levelset persistence in a category theory framework; (2) a Mayer-Vietoris sequence
that brings together levelset, sublevelset and superlevelset homology
modules of a real-valued function; and (3) 1-dimensional slices of 2-dimensional 
persistence modules  along lines of negative slope.

This study of persistent homology from the viewpoint of $c$-modules and $p$-sheaves
opens up avenues for further investigation. We conclude with a discussion of some
of the questions raised by the results of this paper. 

\begin{enumerate}[(a)]

\item The Isometry Theorem was proven in the context of decomposable $p$-sheaves. 
On the other hand, Theorem \ref{T:intdec} shows that any virtually tame $c$-module 
$\cmod{V}$ admits a (unique) interval decomposition and therefore may be represented 
by a persistence diagram. Is there an isometry theorem for decomposable $c$-modules
under a suitable notion of  interleaving?

\item This paper has focused primarily on structural and stability questions associated 
with  $c$-modules. However, the practical relevance of $c$-modules depends heavily 
on computability of interval decompositions. Thus, a basic problem  is that 
of  developing and implementing an algorithm to calculate the persistence diagram 
of a sufficiently tame $c$-module.

\item There are algorithms to calculate the persistence diagrams for the sublevelset
and superlevelset homology modules of a function $f \colon X \to \real$. To what extent can the 
persistence diagrams for the levelset $c$-module of $f$ be inferred from these and the 
homology groups $H_n (X)$ via the persistent Mayer-Vietoris sequence of Section \ref{S:mv}?
More generally, given an exact sequence of $c$-modules in which the barcodes for
every other term is known, can we infer the other barcodes?

\end{enumerate}

\bibliographystyle{abbrv}
\bibliography{cmodule}

\end{document}